\documentclass[10pt]{article}
\usepackage[margin=1in]{geometry}
\usepackage{vmargin}
\usepackage{appendix}
\usepackage[hidelinks]{hyperref}
\usepackage{amssymb}
\usepackage[shortalphabetic]{amsrefs}
\usepackage{amsmath}
\usepackage{amsrefs}
\usepackage{upgreek}
\usepackage{amsthm}
\usepackage{mathdots}
\usepackage{url}
\usepackage{enumerate}
\usepackage{mathtools}
\usepackage{amsxtra}
\usepackage{xcolor}
 \usepackage[refpage]{nomencl}
\usepackage{enumitem}
\usepackage{graphicx}
\usepackage{comment}

\makenomenclature

\newtheorem{thm}{Theorem}[section]
\newtheorem{lem}[thm]{Lemma}
\newtheorem{cor}[thm]{Corollary}
\newtheorem{prop}[thm]{Proposition}

\newcommand{\R}{\mathbb R}

\newcommand{\stable}{\mathrm{st}}

\title{Equivalent definitions of Arthur packets for real 
  unitary groups} 

\author{Nicolas Arancibia Robert, Paul Mezo}

\begin{document}
\maketitle

\begin{abstract}
Mok and Moeglin-Renard have defined Arthur packets for unitary
groups.  Their definitions follow Arthur's work on classical groups and rely on harmonic analysis.
For real groups there is an alternative definition of Arthur packets, due to Adams-Barbasch-Vogan. It relies on sheaf-theoretic techniques instead of harmonic analysis. We prove that these two definitions of Arthur packets
are equivalent in the case of real unitary groups.
\end{abstract}

\tableofcontents

\section{Introduction}

In an effort to characterize the automorphic spectrum of a connected
reductive group, Arthur introduced a set of parameters together with a
collection of conjectures concerning them (\cite{Arthur84},
\cite{Arthur89}).  The parameters are commonly called
Arthur-parameters or simply A-parameters.  These automorphic A-parameters are global objects.  Conjecturally, each global A-parameter gives rise to a local A-parameter for every valuation of the underlying field.

In the present work, we study local A-parameters only for a real
valuation and only for unitary groups.  Moreover, our first main theorem concerns real quasisplit unitary groups.  A real quasisplit unitary group is a real form of a general linear group, which we denote here as $G(\mathbb{R})$.  An \emph{A-parameter} for $G$ is a group homomorphism
\begin{equation}
  \label{aparameter}
  \psi_{G}: W_{\mathbb{R}} \times \mathrm{SL}_{2} \rightarrow {^\vee}G^{\Gamma}
\end{equation}
in which, $\Gamma = \mathrm{Gal}(\mathbb{C}/\mathbb{R})$, ${^\vee}G^{\Gamma} = {^\vee}G \rtimes \Gamma$ is the Galois form of the L-group of $G$, and $W_{\mathbb{R}}$ is the real Weil group.
\nomenclature{$\psi_{G}$}{A-parameter for $G$}
\nomenclature{$W_{\mathbb{R}}$}{real Weil group}
\nomenclature{${^\vee}G^{\Gamma}$}{The Galois form of the L-group of $G$}
\nomenclature{$\Gamma$}{Galois group of $\mathbb{C}/\mathbb{R}$}
In addition, $\psi_{G}\mid_{W_{\mathbb{R}}}$ is a tempered L-parameter
and $\psi_{G}\mid_{\mathrm{SL}_{2}}$ is algebraic (see \cite{borel}
for the definitions).  Arthur conjectured the existence of a stable
virtual character $\eta_{\psi_{G}}$ of $G(\mathbb{R})$ with several
properties.  The set of irreducible characters appearing in
$\eta_{\psi_{G}}$ with non-zero multiplicity are, by definition, the
Arthur-packet (or A-packet) $\Pi_{\psi_{G}}$.    Arthur conjectured the
irreducible characters in $\Pi_{\psi_{G}}$ to be characters of irreducible unitary representations.

In this limited scope, Arthur's conjectures have been proven to a
great extent.  However, there have been two disparate methods
employed, and it has been unknown whether the two methods lead to the
same conclusions. The first method relies on global harmonic analysis
(\cite{rog}, \cite{Arthur}, \cite{Kaletha-Minguez}, \cite{Mok}).  We
denote the stable virtual character defined by Mok in \cite{Mok} by
$\eta_{\psi_{G}}^{\mathrm{Mok}}$.  The second method is local and
relies on sheaf theory (\cite{ABV}).  We denote the stable virtual
character defined by Adams, Barbasch and Vogan in \cite{ABV} by
$\eta^{\mathrm{ABV}}_{\psi_{G}}$.  Our main theorem for quasisplit
unitary groups is
\begin{equation}
  \label{mainthm}
  \eta_{\psi_{G}}^{\mathrm{Mok}} = \eta^{\mathrm{ABV}}_{\psi_{G}}
\end{equation}
(Theorem \ref{finalthm1}).
\nomenclature{$\eta_{\psi_{G}}^{\mathrm{Mok}}$}{Mok's stable virtual character}
An immediate consequence of this theorem is that the sets of irreducible characters appearing in each of the virtual characters are the same, that is
$$\Pi_{\psi_{G}}^{\mathrm{Mok}} = \Pi_{\psi_{G}}^{\mathrm{ABV}}.$$
The irreducible characters in $\Pi_{\psi_{G}}^{\mathrm{Mok}}$ are all
unitary (\cite{Mok}*{Theorem 3.2.1 (b)}), whereas this was not known
for $\Pi_{\psi_{G}}^{\mathrm{ABV}}$.  These identities have consequences
for all real forms of unitary groups, not just the quasisplit forms.
To explain why this is so requires some background which is also present in
the proof of (\ref{mainthm}).  For this reason we first provide an overview of
the proof of (\ref{mainthm}) for quasisplit forms, and return to the
remaining real forms thereafter.

The principal difficulty in proving (\ref{mainthm}) lies in the disparate manner in which the two virtual characters are defined.  Let us examine the definitions beginning with the virtual character $\eta_{\psi_{G}}^{\mathrm{Mok}}$.  The main idea here is to express the unitary group $G$ as a \emph{twisted endoscopic group} of a pair $(\mathrm{R}_{\mathbb{C}/\mathbb{R}} \mathrm{GL}_{N}, \vartheta)$. In this pair $\mathrm{R}_{\mathbb{C}/\mathbb{R}} \mathrm{GL}_{N}$ is the algebraic group obtained from the general linear group $\mathrm{GL}_{N}$ by restriction of scalars from $\mathbb{C}$ to $\mathbb{R}$ (\cite{springer}*{Proposition 11.4.22}, \cite{borel}*{I.5}). It may be regarded as $\mathrm{GL}_{N} \times \mathrm{GL}_{N}$ together with a real structure whose real points determine the real form $\mathrm{GL}_{N}(\mathbb{C})$.  The other member of the pair is an automorphism $\vartheta$ of $\mathrm{R}_{\mathbb{C}/\mathbb{R}} \mathrm{GL}_{N}$.  It is defined by
\begin{equation}
  \label{vartheta}
  \vartheta(g_{1}, g_{2}) =  (\tilde{J} (g_{2}^{-1})^{\intercal} \tilde{J}^{-1}, \tilde{J} (g_{1}^{-1})^{\intercal} \tilde{J}^{-1}),  \quad g_{1}, g_{2} \in \mathrm{GL}_{N},
\end{equation}
where $\tilde{J}$ is the anti-diagonal matrix\small
\begin{equation*}
  \tilde{J} = \begin{bmatrix}0 & & & 1 \\
    & & -1 & \\
    & \iddots & & \\
  (-1)^{N-1} & & &  0 
  \end{bmatrix}.
\end{equation*}
\normalsize
Clearly, $\vartheta$ is an automorphism of order two, and as a real form, the semidirect product $\mathrm{GL}_{N}(\mathbb{C}) \rtimes \langle \vartheta
\rangle$ is a disconnected algebraic group.
\nomenclature{$\vartheta$}{automorphism of $\mathrm{R}_{\mathbb{C}/\mathbb{R}} \mathrm{GL}_{N}$}
\nomenclature{$\tilde{J}$}{anti-diagonal matrix}
The group $G$ is attached to the
pair $(\mathrm{R}_{\mathbb{C}/\mathbb{R}}\mathrm{GL}_{N}, \vartheta)$ in that  ${^\vee}G$ is isomorphic to the 
identity component of the fixed-point set $({^\vee}\mathrm{R}_{\mathbb{C}/\mathbb{R}}\mathrm{GL}_{N})^{\vartheta}$.  This furnishes an inclusion
\begin{equation*}
\epsilon : {^\vee}G^{\Gamma} \hookrightarrow
{^\vee}\mathrm{R}_{\mathbb{C}/\mathbb{R}}\mathrm{GL}_{N}^{\Gamma}
\end{equation*}
which allows us to define the A-parameter
\begin{equation}
\label{prepsitilde}
\psi = \epsilon \circ \psi_{G}
\end{equation}
for $\mathrm{R}_{\mathbb{C}/\mathbb{R}}\mathrm{GL}_{N}$ using (\ref{aparameter}).
\nomenclature{$\psi$}{A-parameter for $\mathrm{R}_{\mathbb{C}/\mathbb{R}}\mathrm{GL}_{N}$}
Since the real form $\mathrm{GL}_{N}(\mathbb{C})$ is particularly well-understood, there is an obvious choice of stable virtual character $\eta_{\psi}^{\mathrm{Mok}}$.  It is in fact a single irreducible character
$$\eta_{\psi}^{\mathrm{Mok}} = \pi_{\psi}$$
of $\mathrm{GL}_{N}(\mathbb{C})$.  Furthermore, as a representation,
$\pi_{\psi}$ is stable under composition with $\vartheta$. This allows
one to extend $\pi_{\psi}$ to a representation
$\pi_{\psi}^{\thicksim}$ of the disconnected group
$\mathrm{GL}_{N}(\mathbb{C})
\rtimes \langle \vartheta \rangle$.  At this stage some care must be taken, as the extension is only unique up to a sign.  If we ignore this wrinkle for the time being, then we obtain the \emph{twisted character} of $\pi_{\psi}^{\thicksim}$ as the restriction of the character of $\pi_{\psi}^{\thicksim}$ to the non-identity component $\mathrm{GL}_{N}(\mathbb{C}) \rtimes \vartheta$.  If we identify $\pi_{\psi}^{\thicksim}$ with its twisted character then the stable virtual character $\eta_{\psi_{G}}^{\mathrm{Mok}}$ is  uniquely determined by the twisted endoscopic transfer identity
\begin{equation}
  \label{spectrans}
\pi_{\psi}^{\thicksim} = \mathrm{Trans}_{G(\mathbb{R})}^{\mathrm{GL}_{N}(\mathbb{C}) \rtimes \vartheta} (\eta_{\psi_{G}}^{\mathrm{Mok}})
\end{equation}
(\cite{Mok}*{Proposition 8.2.1}).
The twisted endoscopic transfer map $\mathrm{Trans}_{G(\mathbb{R})}^{\mathrm{GL}_{N}(\mathbb{C}) \rtimes \vartheta}$ is
defined on the space of stable virtual characters of $G(\mathbb{R})$.  It is defined for real reductive groups in \cite{Shelstad12}, \cite{Mezo} and \cite{Mezo2}.

Turning now to the definition of $\eta_{\psi_{G}}^{\mathrm{ABV}}$, we come upon completely different methods.  A remarkable innovation of Adams, Barbasch and Vogan is their introduction of a complex variety $X({^\vee}G^{\Gamma})$, together with a ${^\vee}G$-action, such that the ${^\vee}G$-orbits are in bijection with the equivalence classes of L-parameters (\cite{ABV}*{Section 6}).  The orbits stratify the variety.  Thus, one may consider ${^\vee}G$-equivariant local systems on the orbits, and ${^\vee}G$-equivariant \emph{constructible sheaves} or \emph{perverse sheaves} on $X({^\vee}G)$. 
We define a  {\it complete geometric parameter}  to be a pair 
\begin{equation}
\label{cgp}
\xi = (S, \mathcal{V})
\end{equation}
consisting of an orbit $S \subset
X({^\vee}G^{\Gamma})$, together with a ${^\vee}G$-equivariant local system
$\mathcal{V}$ on $S$ 
(\cite{ABV}*{Definition 7.6}).
The set of complete geometric parameters is denoted by $\Xi({^\vee}G^{\Gamma})$.
This definition ignores the more general local systems in \cite{ABV}, which are
equivariant for an algebraic cover of ${^\vee}G$.
By \cite{ABV}*{Theorem 10.11} there is  a canonical bijection
\begin{equation}
\label{Gparameters}
\Xi({^\vee}G^{\Gamma})\longleftrightarrow\Pi(G/\mathbb{R})
\end{equation}
which may be regarded as a more precise version of the local Langlands correspondence.
The set on the right is the set of (equivalence classes of)
irreducible admissible representations of  \emph{pure  forms} of $G$.  The pure forms include the quasisplit form $G(\mathbb{R})$. 
We write bijection (\ref{Gparameters}) as 
\begin{equation}
\label{pimap}
\xi\mapsto \pi(\xi).
\end{equation}
Let $K\Pi(G/\mathbb{R})$ be the Grothendieck group of the
finite-length admissible representations of pure forms of $G$.  This  Grothendieck group has $\{ \pi(\xi) \}$ as a $\mathbb{Z}$-basis.  It contains the virtual characters of $G(\mathbb{R})$ as a subgroup.

There is a similar picture for sheaves when considering the dual
group ${^\vee}G$. Suppose 
$\xi = (S, \mathcal{V})\in\Xi({^\vee}G^{\Gamma})$. The local system $\mathcal{V}$ is a ${^\vee}G$-equivariant sheaf on $S \subset X({^\vee}G^{\Gamma})$. 
Applying the functor of intermediate extension to the closure of $S$, and then taking the direct image to $X({^\vee}G)$ produces an 
irreducible ${^\vee}G$-equivariant  perverse sheaf
$P(\xi)$, and a bijection 
$$
\xi\mapsto P(\xi)
$$
(\cite{ABV}*{Section 7}).  Let $KX({^\vee}G^{\Gamma})$ be the Grothendieck group of the category of ${^\vee}G$-equivariant perverse sheaves on $X({^\vee}G^{\Gamma})$.  It has $\{ P(\xi) \}$ as a $\mathbb{Z}$-basis.

There is a perfect pairing 
\begin{equation}
\label{prepair}
\langle \cdot, \cdot \rangle_{G} : K \Pi(G/\mathbb{R}) \times K X({^\vee}
G^{\Gamma}) \rightarrow 
\mathbb{Z} 
\end{equation}
which satisfies
\begin{equation}
\label{irredpair}
\langle \pi(\xi), P(\xi') \rangle_{G} = e(\xi)\, (-1)^{d(\xi)} \,
\delta_{\xi, \xi'}, \quad \xi,\xi' \in \Xi({^\vee}G^{\Gamma})
\end{equation}
 (\cite{ABV}*{Theorem 1.24}).
Here, $d(\xi)$ is the dimension of the orbit $S$ in $\xi = (S,
\mathcal{V})$, $e(\xi)$ is the Kottwitz sign (\cite{ABV}*{Definition 15.8}), and
$\delta_{\xi, \xi'}$ is the Kronecker delta.

Using pairing (\ref{prepair}), we identify virtual characters as $\mathbb{Z}$-valued linear functionals on $K X({^\vee}
G^{\Gamma})$.  The theory of microlocal geometry provides a family of linear functionals
\begin{equation}
\label{mmm}
\chi^{\mathrm{mic}}_{S} : K X({^\vee}G^{\Gamma}) \rightarrow \mathbb{Z}
\end{equation}
parameterized by  the ${^\vee}G$-orbits $S \subset
X({^\vee}G^{\Gamma})$.
These microlocal multiplicity maps appear in the
theory of \emph{characteristic cycles} (\cite{ABV}*{Chapter 19},
\cite{Boreletal}), and are associated with ${^\vee}G$-equivariant
local systems on a conormal bundle over $X({^\vee}G^{\Gamma})$ (\cite{ABV}*{Section 24}, \cite{GM}). 
The virtual characters associated by the pairing to these linear functionals
are stable (\cite{ABV}*{Theorems 1.29 and 1.31}).

The stable virtual character $\eta_{\psi_{G}}^{\mathrm{ABV}}$ is defined from (\ref{mmm}) as follows.  There is an L-parameter $\phi_{\psi_{G}}$ associated to $\psi_{G}$ (\cite{ABV}*{Definition 22.4}).  Let $S_{\psi_{G}} \subset X({^\vee}G^{\Gamma})$ be the unique ${^\vee}G$-orbit associated to $\phi_{\psi_{G}}$, and  $\eta_{\psi_{G}}^{\mathrm{mic}}$ be the unique virtual character satisfying
\begin{equation*}
  \langle \eta^{\mathrm{mic}}_{\psi_{G}}, \mu \rangle_{G} = \chi_{S_{\psi_{G}}}^{\mathrm{mic}}(\mu),
\quad \mu \in KX({^\vee}G^{\Gamma}).
\end{equation*}
As a distribution, the stable virtual character $\eta^{\mathrm{mic}}_{\psi_{G}}$
is supported on real forms of $G$ which include the quasisplit form
$G(\mathbb{R})$.   For the purpose of proving (\ref{mainthm}), it
suffices to consider the restriction to the quasisplit form.  We
therefore define
\begin{equation}
  \label{eq:etaABVIntro}
\eta^{\mathrm{ABV}}_{\psi_{G}} = \eta^{\mathrm{mic}}_{\psi_{G}}|_{G(\mathbb{R})}.
\end{equation}

Recall that the definition of $\eta_{\psi_{G}}^{\mathrm{Mok}}$ in
(\ref{spectrans}) relies on the theory of twisted endoscopy.  One
might hope to find a bridge between $\eta_{\psi_{G}}^{\mathrm{Mok}}$
and  $\eta_{\psi_{G}}^{\mathrm{ABV}}$ by working in a theory of twisted endoscopy for $\eta_{\psi_{G}}^{\mathrm{ABV}}$.  Fortunately, the theory of standard endoscopy already appears in \cite{ABV} and is extended to the  twisted setting in \cite{Christie-Mezo}.  There are two main tasks in this extension.

The first task is the definition of a meaningful pairing
between the $\mathbb{Z}$-modules of {\it twisted characters}
$K\Pi(\mathrm{GL}_N(\mathbb{C}),\vartheta)$
and {\it twisted sheaves} $KX({^\vee}\mathrm{R}_{\mathbb{C}/\mathbb{R}} \mathrm{GL}_{N}^{\Gamma},\vartheta)$ (\cite{LV2014}*{Section 2.3})
\begin{equation}
\label{prepairtwist}
\langle \cdot, \cdot \rangle: K \Pi(\mathrm{GL}_{N}(\mathbb{C}),
\vartheta) \times KX 
({^\vee}\mathrm{GL}_{N}^{\Gamma}, \vartheta ) \rightarrow \mathbb{Z}.
\end{equation}
This is a serious task to which we shall return later in the introduction.

The second task is to define a twisted endoscopic lifting map
\begin{equation}
  \label{lift0}
  \mathrm{Lift}_{0}: K \Pi(G(\mathbb{R}))^{\mathrm{st}} \rightarrow K(\mathrm{GL}_{N}(\mathbb{C}),\vartheta)
  \end{equation}
from stable virtual characters to twisted characters, \emph{i.e.} the counterpart of $\mathrm{Trans}_{G(\mathbb{R})}^{\mathrm{GL}_{N}(\mathbb{C}) \rtimes \vartheta}$ in (\ref{spectrans}).  This is not that serious, for there is an inverse image functor on sheaves
\begin{equation}
  \label{epupstar}
  \epsilon^*:K X({^\vee}\mathrm{R}_{\mathbb{C}/\mathbb{R}} \mathrm{GL}_N^\Gamma,\vartheta)\rightarrow KX({^\vee}G^{\Gamma})
\end{equation}
which allows us to define $\mathrm{Lift}_{0}$ by the identity
\begin{equation}
  \label{Lift0}
\langle \mathrm{Lift}_0(\eta),\mu\rangle = \langle \eta,\epsilon^*(\mu)\rangle_G, \quad \mu\in KX({^\vee}\mathrm{R}_{\mathbb{C}/\mathbb{R}}\mathrm{GL}_N^\Gamma,\vartheta).
\end{equation}
In this identity both pairings (\ref{prepairtwist}) and (\ref{prepair}) are used.

The equality between $\eta_{\psi_{G}}^{\mathrm{Mok}}$ and $\eta_{\psi_{G}}^{\mathrm{ABV}}$  may then be established by returning to (\ref{spectrans}), proving
\begin{equation}
   \label{Lift0trans}
   \mathrm{Lift}_{0} = \mathrm{Trans}_{G(\mathbb{R})}^{\mathrm{GL}_{N}(\mathbb{C}) \rtimes \vartheta},
   \end{equation}
and 
\begin{equation}
  \label{26.25}
  \mathrm{Lift}_{0}(\eta_{\psi_{G}}^{\mathrm{ABV}}) = \pi_{\psi}^{\thicksim}.
  \end{equation}
Equation (\ref{Lift0trans}) ends up being a simple consequence of the definition of $\mathrm{Lift}_{0}$ and \cite{AMR}*{(1.0.3)}. Equation (\ref{26.25}) is to a large extent proven in \cite{ABV}*{Theorem 26.25}.
From these two equations it follows that
$$\mathrm{Lift}_{0}(\eta_{\psi_{G}}^{\mathrm{Mok}}) = \mathrm{Lift}_{0}(\eta_{\psi_{G}}^{\mathrm{ABV}})$$
and then an injectivity result yields the main theorem (\ref{mainthm}).

The proof we have sketched follows \cite{aam} entirely.  The classical
groups in \cite{aam} are twisted endoscopic groups of
$\mathrm{GL}_{N}$.   In the present work the classical groups are
replaced by unitary groups, and $\mathrm{GL}_{N}$ is replaced by
$\mathrm{R}_{\mathbb{C}/\mathbb{R}}\mathrm{GL}_{N}$.  The good
properties of $\mathrm{GL}_{N}$ which were harnessed in \cite{aam}
(\emph{e.g.} connected centralizers) are also properties of
$\mathrm{R}_{\mathbb{C}/\mathbb{R}}\mathrm{GL}_{N}$.  In truth, both
the structure and the representation theory of
$\mathrm{GL}_{N}(\mathbb{C}) =
\mathrm{R}_{\mathbb{C}/\mathbb{R}}\mathrm{GL}_{N}(\mathbb{R})$ are
simpler than for  $\mathrm{GL}_{N}(\mathbb{R})$.  We have made efforts
to highlight the simplifications.  Where we have not been able to
improve on  preliminary material, we have copied passages from
\cite{aam}.

Thus far, we have discussed identity (\ref{mainthm}) which
pertains only to quasisplit unitary groups.  We have indicated how the
theory of twisted endoscopy plays a crucial role in the proof of the
identity.  In the final two sections of this paper, we explore two
variants on the proof of (\ref{mainthm}).  The first
variant still concerns a quasisplit unitary group $G(\mathbb{R})$, but now
for \emph{standard} (non-twisted) endoscopy.  In this setting, $G'(\mathbb{R})$ is a quasisplit
endoscopic group of $G(\mathbb{R})$.  The relationship between 
$\mathrm{R}_{\mathbb{C}/\mathbb{R}}\mathrm{GL}_{N}$ and
$G(\mathbb{R})$ in the twisted setting is replaced by the relationship between
$G(\mathbb{R})$ and $G'(\mathbb{R})$ in the standard setting.  One may
choose $G'(\mathbb{R})$ 
to be a product of smaller unitary groups so that identity
(\ref{mainthm}) holds for $G'(\mathbb{R})$.  One may then examine the
standard endoscopic lifts of the stable virtual characters appearing in
(\ref{mainthm}) for $G'(\mathbb{R})$.  Explicit formulae for these lifts
are given in both \cite{Mok} and \cite{ABV}.  A detailed comparison of
these formulae is presented in Section \ref{consec}.

In the second variant, we again consider standard endoscopy.  However,
in this variant the quasisplit unitary group $G(\mathbb{R})$ is the
endoscopic group of a (pure) form $G(\mathbb{R},\delta) = \mathrm{U}(p,q)$.  Moeglin and Renard define a stable virtual
character $\eta_{\psi_{G}}^{\mathrm{MR}}$ on $G(\mathbb{R},\delta)$ as
the standard endoscopic lift of $\eta_{\psi_{G}}^{\mathrm{Mok}}$
(\cite{MR3}, \cite{MR4}).  The
irreducible characters appearing in $\eta_{\psi_{G}}^{\mathrm{MR}}$
form a packet $\Pi_{\psi_{G}}^{\mathrm{MR}}(G(\mathbb{R},\delta))$ of
unitary representations.
Analogues for these objects, $\eta_{\psi_{G}}^{\mathrm{ABV}}(\delta)$ and
$\Pi_{\psi_{G}}^{\mathrm{ABV}}(G(\mathbb{R},\delta))$, appear in
\cite{ABV}.  In Section \ref{puresec}, we prove our second main
theorem, namely that
$$e(\delta) \, \eta_{\psi_{G}}^{\mathrm{MR}} =
\eta_{\psi_{G}}^{\mathrm{ABV}}(\delta)$$
where $e(\delta) = \pm 1$ is a Kottwitz invariant.  It
is then immediate that
$$\Pi_{\psi_{G}}^{\mathrm{MR}}(G(\mathbb{R},\delta)) =
\Pi_{\psi_{G}}^{\mathrm{ABV}}(G(\mathbb{R},\delta)).$$
In fact, these identities are a special case of Theorem \ref{purethm},
which expresses identities for all standard endoscopic groups of
$G(\mathbb{R},\delta)$.  An immediate consequence of the identity of
packets is that $\Pi_{\psi_{G}}^{\mathrm{ABV}}(G(\mathbb{R},\delta))$
consists of unitary representations.  This extends  the unitarity
results for special unipotent representations of unitary groups in
\cite{BMSZ}. We expect that the methods used in the proof of Theorem
\ref{purethm} should carry over to pure inner forms of special
orthogonal groups--a work in progress.

We have just given a synopsis of the last two sections.   We now give
synopses of the remaining sections.  Section \ref{extgroups} is a review of the preliminary material of \cite{ABV}.  Two important objects appearing here are the extended group $G^{\Gamma}$ which mirrors the L-group ${^\vee}G^{\Gamma}$ and the complex variety $X({^\vee}\mathcal{O}, {^\vee}G^{\Gamma}) \subset X({^\vee}G^{\Gamma})$.  The term ${^\vee}\mathcal{O}$ is a semisimple ${^\vee}G$-orbit in the complex Lie algebra ${^\vee}\mathfrak{g}$, and is to be thought of as an infinitesimal character.  This infinitesimal character accompanies all of the arguments in the sequel and for the most part is assumed to be regular.   The notions of \emph{pure inner form} and \emph{pure strong involution} are also introduced and compared.  The section culminates with a description of the revised local Langlands correspondence (\ref{Gparameters}).

In Section \ref{extended} the specifics of some of this preliminary material are given for quasisplit unitary groups and $\mathrm{R}_{\mathbb{C}/\mathbb{R}} \mathrm{GL}_{N}$.  In particular it is proven that $\mathrm{R}_{\mathbb{C}/\mathbb{R}} \mathrm{GL}_{N}$ has only one pure strong involution which corresponds to the sole inner form, namely $\mathrm{GL}_{N}(\mathbb{C})$.  

Section \ref{atlasparam} presents an alternative  set of parameters to the complete geometric parameters (\ref{cgp}).  They are the parameters introduced  in \cite{Adams-Fokko} and \cite{AVParameters}, and so we call them \emph{Atlas parameters}.  The Atlas parameters are more convenient for the computations appearing in later sections.  The Atlas parameters also have an easily discernible involution which pertains to Vogan duality, a tool used later as well.

A third advantage to the Atlas parameters is that they may be extended to provide a parameterization for the irreducible representations of $\mathrm{GL}_{N}(\mathbb{C}) \rtimes \langle \vartheta \rangle$.  This is the subject of Section \ref{extrepsec}.  One of the favourable features of the Atlas parameters in this special context is the existence of a preferred extension $\pi(\xi)^{+}$ to $\mathrm{GL}_{N}(\mathbb{C}) \rtimes \langle \vartheta \rangle$ of any $\vartheta$-stable irreducible representation $\pi(\xi)$ of $\mathrm{GL}_{N}(\mathbb{C})$.  We refer to this preferred extension as the \emph{Atlas extension} of $\pi(\xi)$.

Section \ref{grothchar} lays out the notation for the Grothendieck group of admissible representations and provides the construction for the related concept of the $\mathbb{Z}$-module of twisted characters of $\mathrm{GL}_{N}(\mathbb{C}) \rtimes \langle \vartheta \rangle$.

Section \ref{sheaves} is devoted to the pairings, (\ref{prepair}) and (\ref{prepairtwist}), and the definitions of the virtual characters, $\eta_{\psi_{G}}^{\mathrm{ABV}}$ and $\eta_{\psi}^{\mathrm{ABV}}$, which are defined through them.  The values of the ordinary pairing (\ref{prepair}) were given in (\ref{irredpair}) for irreducible representations and perverse sheaves. It is equally important to understand the values of this pairing on \emph{standard representations} and irreducible \emph{constructible sheaves}.  Before saying why, we recall that any irreducible representation $\pi(\xi)$ in (\ref{pimap}) is the unique Langlands quotient of a standard representation, which we denote by $M(\xi)$.  If the parameter $\xi$ is $\vartheta$-stable then there is an Atlas extension $M(\xi)^{+}$ which contains $\pi(\xi)^{+}$ as a quotient.  On the other hand, if one replaces the intermediate extension with extension by zero in the construction of $P(\xi)$ above, then one arrives at an irreducible ${^\vee}G$-equivariant constructible sheaf $\mu(\xi)$.
The Grothendieck group of the ${^\vee}G$-equivariant constructible sheaves on $X({^\vee}\mathcal{O},{^\vee}G^{\Gamma})$ is isomorphic to the Grothendieck group  $KX({^\vee}\mathcal{O},{^\vee}G^{\Gamma})$ for the perverse sheaves (\cite{ABV}*{Lemma 7.8}).  It therefore makes sense to evaluate the pairing on these two objects.  As a matter of fact pairing (\ref{prepair}) is \emph{defined} by
$$\langle M(\xi), \mu(\xi') \rangle = e(\xi) \,\delta_{\xi, \xi'}, \quad \xi, \xi' \in \Xi({^\vee}\mathcal{O}, {^\vee}G^{\Gamma})$$
and the content of \cite{ABV}*{Theorem 1.24} is (\ref{irredpair}).  It is important to know the values of the pairing on these objects, since there is a well-known basis for the stable virtual characters (\emph{cf.} (\ref{lift0})) given in terms of standard representations (\cite{shelstad}).  In addition, the inverse image functor (\ref{epupstar}) is computed relative to constructible sheaves.

Much of the technical work in this paper is spent on defining the
twisted pairing (\ref{prepairtwist}) and proving that its values on
standard representations and constructible sheaves are related to its
values on irreducible representations and perverse sheaves as in the
ordinary case.  The first undertaking is to define preferred
$({^\vee}\mathrm{R}_{\mathbb{C}/\mathbb{R}} \mathrm{GL}_{N} \rtimes
\langle \vartheta \rangle)$-equivariant sheaves $\mu(\xi)^{+}$ and
$P(\xi)^{+}$ which restrict to $\mu(\xi)$ and $P(\xi)$ respectively as
${^\vee}\mathrm{R}_{\mathbb{C}/\mathbb{R}}
\mathrm{GL}_{N}$-equivariant sheaves.  These are the twisted sheaves
mentioned above.  We define a $\mathbb{Z}$-module $K(X({^\vee}\mathcal{O}, {^\vee}\mathrm{R}_{\mathbb{C}/\mathbb{R}} \mathrm{GL}_{N}), \vartheta)$ of twisted sheaves akin to the module of twisted characters $K \Pi(({^\vee}\mathcal{O},\mathrm{GL}_{N}\mathbb{C}), \vartheta)$.  We then define the twisted pairing (\ref{prepairtwist}) by setting
\begin{equation}
\label{standpairtwist}
  \langle M(\xi)^{+}, \mu(\xi')^{+} \rangle = (-1)^{l^{I}(\xi) -
  l^{I}_{\vartheta}(\xi)} \, \delta_{\xi, \xi'}
\end{equation}
The definition of the signs on the right appears in (\ref{intlength})
and (\ref{thetalength}).  As we shall see, these signs are crucial in making comparisons with other extensions $M(\xi)^{\thicksim}$ and $\pi(\xi)^{\thicksim}$.  The principal result pertaining to the twisted pairing is
\begin{equation}
  \label{irredpairtwist}
\langle \pi(\xi)^{+}, P(\xi')^{+} \rangle = (-1)^{d(\xi)} \,
(-1)^{l^{I}(\xi)-l^{I}_{\vartheta}(\xi)} \, \delta_{\xi, \xi'}
\end{equation}
(Theorem \ref{twistpairing}).

The sole objective of Section \ref{pairings} is to prove (\ref{irredpairtwist}).
  Our proof in this twisted setting is an adaptation of the proof
  in the ordinary setting (\cite{ABV}*{Sections 16-17}) using the tools of
\cite{AVParameters}.  Hecke operators are among
these tools.  There is a difference between \cite{ABV} and
\cite{AVParameters} in the objects upon which the Hecke operators act.
In \cite{ABV} Hecke operators are defined on both characters and
sheaves.  By contrast, the Hecke operators of \cite{AVParameters}*{Section 7}
are defined only on (twisted) characters.   The links between
characters and sheaves in the Hecke actions are the Riemann-Hilbert and  Beilinson-Bernstein
correspondences (\cite{ABV}*{Theorems 7.9 and 8.3}).  In Section \ref{duality} we describe these correspondences as a bijection 
$$P(\xi) \longleftrightarrow \pi( {^\vee}\xi), \quad \xi \in
\Xi({^\vee}\mathcal{O}, {^\vee}G^{\Gamma}),$$ 
where $\pi({^\vee}\xi)$ is the \emph{Vogan dual} of $\pi(\xi)$ (as the
equivalence class of a Harish-Chandra module) (6.1
\cite{AVParameters}).  For $\mathrm{R}_{\mathbb{C}/\mathbb{R}} \mathrm{GL}_{N}$ the correspondence is
extended to 
$$P(\xi)^{+} \longleftrightarrow \pi({^\vee}\xi)^{+}$$
for $\vartheta$-fixed complete geometric parameters $\xi$.  Once
sheaves are aligned with characters in this manner, the rest of the
proof of (\ref{irredpairtwist}) follows \cite{ABV} and \cite{aam}*{Section 4}.

In Section \ref{endosec} we describe the theory of endoscopy, both
standard and twisted, for $\mathrm{R}_{\mathbb{C}/\mathbb{R}}\mathrm{GL}_{N}$.  The standard theory of endoscopy is included to motivate the twisted
theory and is also used in Section \ref{glnpacket}.  The twisted theory of endoscopy in Section \ref{twistendsec} is a specialization of \cite{Christie-Mezo}*{Section 5.4}. 
We compute the values of the twisted endoscopic lifting map $\mathrm{Lift}_{0}$ on a basis of the stable virtual characters.  This is instrumental in proving (\ref{Lift0trans}) and in  proving that $\mathrm{Lift}_{0}$ is injective.  The value $\mathrm{Lift}_{0}(\eta_{\psi_{G}}^{\mathrm{ABV}})$ is
described as an element $\eta_{\psi}^{\mathrm{ABV}+} \in K
\Pi({^\vee}\mathcal{O}, \mathrm{GL}_{N}(\mathbb{C}), \vartheta)$,
which may be regarded as an extension of $\eta_{\psi}^{\mathrm{ABV}}$.

In Section \ref{glnpacket}  we prove that for \emph{any} A-parameter
$\psi$ of $\mathrm{R}_{\mathbb{C}/\mathbb{R}} \mathrm{GL}_{N}$ (not necessarily of the form
(\ref{prepsitilde})) $\eta_{\psi}^{\mathrm{ABV}} = \pi_{\psi}$. 
The proof begins under the assumption that $\psi$ is an
A-parameter studied by Adams and Johnson (\cite{Adams-Johnson}).
Adams and Johnson defined A-packets for these parameters, and it is
easily shown that their packets are singletons for $\mathrm{R}_{\mathbb{C}/\mathbb{R}} \mathrm{GL}_{N}$.
The equality of the Adams-Johnson packets with the
ABV-packets is proven in \cite{arancibia_characteristic}.  The proof that
ABV-packets are singletons follows from a decomposition of $\psi$ in
terms of Adams-Johnson A-parameters of smaller general linear groups,
and an application of  standard endoscopic lifting  from the direct
product of these smaller general linear groups (Proposition
\ref{prop:singletonGLN}). 

In section  \ref{whitsec} we describe another extension $\pi(\xi)^{\thicksim}$ of the $\vartheta$-stable irreducible representations $\pi(\xi)$ of $\mathrm{GL}_{N}(\mathbb{C})$--the so-called \emph{Whittaker extension}.  Regarded as a twisted character $\pi(\xi)^{\thicksim}$ differs from the Atlas extension $\pi(\xi)^{+}$ by at most a sign.  For special complete geometric parameters $\xi$ it is shown that the Whittaker extension agrees with the Atlas extension.
The key to determining the sign for other complete geometric parameters is to compute the sign  when $\pi(\xi)$ is \emph{generic}, \emph{i.e.} has a Whittaker model.  Indeed, the Whittaker extensions are built from extensions of generic representations and irreducible generic representations occur as subrepresentations of any standard representation.  If one knows the (signed) multiplicity
with which an irreducible twisted generic representation $\pi(\xi_{0})^{+}$
appears in the decomposition of a twisted standard
representation $M(\xi)^{+}$, then one can use this knowledge together with the agreement at the special complete geometric parameters to prove that
$$\pi(\xi)^{+}  = (-1)^{l^{I}(\xi) -
  l^{I}_{\vartheta}(\xi)} \ \pi(\xi)^{\thicksim}.$$
Observe that the sign on the right appears on the right of (\ref{standpairtwist}) and (\ref{irredpairtwist}).  Replacing $\pi(\xi)^{+}$ with $(-1)^{l^{I}(\xi) -  l^{I}_{\vartheta}(\xi)} \ \pi(\xi)^{\thicksim}$ yields a cosmetic simplification to the pairing.  More importantly, it is the Whittaker extension which appears in (\ref{spectrans}). With the substitution of the Whittaker extensions into the computed values of $\mathrm{Lift}_{0}$, the identities (\ref{Lift0trans}) and (\ref{26.25}) are established.

These last observations are spelled out in Section \ref{equalapacketreg}.  This short section assembles the essential ingredients already outlined in the introduction in proving the main theorem (\ref{mainthm}).  However, it works under the assumption that the infinitesimal character is regular in ${^\vee}\mathrm{R}_{\mathbb{C}/\mathbb{R}} \mathrm{GL}_{N}$.  This assumption is removed in Section \ref{equalapacketsing} by applying \emph{Jantzen-Zuckerman translation}.  There is nothing novel in this approach and the ideas are all present in \cite{ABV}*{Section 16}.

In closing, let us briefly mention a different approach to obtaining the equality between the  A-packets (\ref{mainthm}).
Moeglin and Renard give an explicit description of the representations in $\Pi_{\psi_G}^{\mathrm{Mok}}$. 
More precisely, by \cite{MR4}*{Equation (5.1)},
$\psi$ in (\ref{prepsitilde}) decomposes as
$
\psi=\psi_0 \oplus \psi_1
$,
where
$$
\psi_i:\ W_{\mathbb{R}}\times\mathrm{SL}_{2} \longrightarrow
{^\vee}\mathrm{R}_{\mathbb{C}/\mathbb{R}}\mathrm{GL}_{N_i}^{\Gamma},\quad i = 0,1, \ N_{0}+N_1=N.
$$
The A-parameter $\psi_{0}$ corresponds to an irreducible unitary representation  $\pi_{0}$ of $\mathrm{GL}_{N_{0}}(\mathbb{C})$.  In addition, the A-parameter $\psi_1$ factors  to an A-parameter $\psi_{G_1}$ of a 
smaller rank unitary group $G_1$
$$
\psi_1:\ {W_{\mathbb{R}}\times\mathrm{SL}_{2}\xrightarrow{\psi_{G_1}}{^\vee}G_1^{\Gamma} }
\ \hookrightarrow \ {^\vee}\mathrm{R}_{\mathbb{C}/\mathbb{R}}\mathrm{GL}_{N_1}^{\Gamma}.
$$
According to (\cite{MR4}*{Proposition 5.2}), 
every irreducible representation in $\Pi_{\psi_G}^{\mathrm{Mok}}$ is parabolically induced from a representation in 
$$
\left\{
\pi_{0} \otimes \pi_{1}: \pi_{1} \in \Pi_{\psi_{G_1}}^{\mathrm{Mok}} 
\right\}.
$$ 
Furthermore, each $\pi_{1} \in\Pi_{\psi_{G_1}}^{\mathrm{Mok}}
$ is cohomologically induced from a character of an inner form of $G_1$ determined by $\psi_{G_1}$  (\cite{MR4}*{Theorem 4.1}).
One  should also be able to 
prove that each representation in $\Pi_{\psi_G}^{\mathrm{ABV}}$ 
is parabolically induced from a representation in $\left\{
\pi_{0} \otimes \pi_{1}: \pi_{1} \in \Pi_{\psi_{G_1}}^{\mathrm{ABV}} 
\right\}$ by imitating Proposition \ref{injliftord}.
This would reduce the proof of the equality  of A-packets (\ref{mainthm}) 
to showing that
$
\Pi_{\psi_{G_1}}^{\mathrm{ABV}}=\Pi_{\psi_{G_1}}^{\mathrm{Mok}}.
$
The proof of this last equality 
should follow the proof of
\cite{arancibia_characteristic}*{Theorem 4.16},
which asserts the equality of  ABV-packets and packets defined by Adams-Johnson (\cite{Adams-Johnson}). Although to do this, one must first
extend the framework in \cite{arancibia_characteristic} to include unitary groups and  A-parameters with singular infinitesimal character.

\section{The local Langlands correspondence}
\label{llc}

Unless otherwise stated, the group $G$ in this section may be taken to
be an arbitrary connected complex reductive algebraic group.  Our goal
is to review the local Langlands correspondence as developed in
\cite{ABV}.  Our review differs in two ways from \cite{ABV}.  First,
we replace the notion of \emph{strong real form} with the equivalent
notion of \emph{strong involution} (\cite{Adams-Fokko}).  Second, we
limit the theory to the set of \emph{pure strong involutions}
(or equivalently \emph{pure strong real forms}).  This limitation
simplifies the review, while retaining the necessary information for
quasisplit forms of $G$.

\subsection{Extended groups and complete geometric parameters}
\label{extgroups}

The L-group of $G$ is an essential feature of the local Langlands correspondence.  The dual group ${^\vee}G$ is an index two subgroup of (the Galois form of) the L-group of $G$.  One of the innovations in the local Langlands correspondence of \cite{ABV} is the introduction of an \emph{extended
  group} for $G$, which mirrors the L-group in that
$G$, not ${^\vee}G$,  appears as an index two subgroup.

We begin the definition of an extended group for $G$ by fixing a
pinning
\begin{equation*}
  (B,H, \{ X_{\alpha} \})
\end{equation*}
in which $B \subset G$ is a Borel subgroup, $H \subset B$ is a maximal
torus and $\{X_{\alpha} \}$ is a set of simple root vectors relative to
the positive root system $R^{+}(G,H) = R(B,H)$ of $R(G,H)$.  We fix an
inner class of real forms for $G$, or equivalently, an algebraic
involution $\delta_{0}$ of $G$ fixing the pinning
(\cite{ABV}*{Proposition 2.12}, \cite{AVParameters}*{(5)}).  The \emph{(weak) extended group} defined by the inner
class is
$$G^{\Gamma} = G \rtimes  \langle \delta_{0} \rangle$$
(\cite{Adams-Fokko}*{Definition 5.1}, \emph{cf.} \cite{ABV}*{Definition 2.13}).  \nomenclature{$G^{\Gamma}$}{extended group}
\nomenclature{$\delta_{0}$}{strong involution defining extended group}

A \emph{strong involution} is an element $\delta \in G^{\Gamma} - G$
such that $\delta^{2} \in Z(G)$ is central and has finite order
(\cite{Adams-Fokko}*{Definition 5.5}).
\nomenclature{$\delta$}{strong involution}
Two strong involutions are
\emph{equivalent} if they are $G$-conjugate.  There is a surjective
map
\begin{equation}
  \label{formmap}
  \delta \mapsto G(\mathbb{R},\delta)
\end{equation}
from (equivalence classes of) strong involutions to  (isomorphism
classes of) real inner forms of $G$
(\cite{Adams-Fokko}*{Lemma 5.7}).
\nomenclature{$G(\mathbb{R},\delta)$}{real form of $G$ corresponding to strong involution $\delta$}
There is a well-known bijection 
between  the real inner forms of $G$ and $H^{1}(\mathbb{R}, G/Z(G))$ (\cite{springer}*{12.3.7}).  Using this bijection one may also think of (\ref{formmap}) as a surjection onto $H^{1}(\mathbb{R}, G/Z(G))$.
It is natural to juxtapose this surjection with the  the quotient map
\begin{equation}
  \label{puremap}
  H^{1}(\mathbb{R}, G) \rightarrow H^{1}(\mathbb{R}, G/Z(G)).
\end{equation}   The cohomology set  $ H^{1}(\mathbb{R}, G)$ is known as  the set of \emph{pure inner forms} (\cite{vogan_local_langlands}*{Section 2}). The pure inner forms may be realized as strong involutions in the following fashion.  
 Let $\sigma \in \Gamma$ be the nontrivial element of the Galois
 group. \nomenclature{${^\vee}\rho$}{half-sum of positive coroots} Set
 $${^\vee}\rho = \frac{1}{2} \sum_{\alpha \in R^{+}(G,H)}
 {^\vee}\alpha.$$
 Any 1-cocycle $z \in Z^{1}(\mathbb{R}, G)$ is equivalent to a
 1-cocycle taking values in $H$ under the $\delta_{0}$-action
 (\cite{Adams-Taibi}*{Proposition 7.4}). After replacing $z$ with such
 an equivalent cocycle, $z(\sigma) \in H$ and 
 \begin{equation}
   \label{pureinv}
  z(\sigma) \exp(\uppi i \, {^\vee}\rho) \delta_{0} \in G^{\Gamma}
 \end{equation}
is seen to be a strong involution.
 ($\exp(\uppi i \, {^\vee}\rho)\delta_{0}$ is the \emph{large} involution
in \cite{AVParameters}*{(11f)-(11h)}).  This assignment sends classes in
$H^{1}(\mathbb{R}, G)$ to equivalence classes of strong involutions.   
The (equivalence classes of) \emph{pure strong involutions} are
defined as the image of this map. 

The quasisplit real form corresponds to the trivial
cocycle of $H^{1}(\mathbb{R}, G/Z(G))$.  It lies in the image of
(\ref{puremap}) as the image of the trivial cocycle in $H^{1}(\mathbb{R}, G)$.  Equivalently, the quasisplit real form is the image under (\ref{formmap}) of the pure strong involution 
\begin{equation}
  \label{deltaq}
  \delta_{q} = \exp(\uppi i \, {^\vee}\rho)\delta_{0}.
\end{equation}
\nomenclature{$\delta_{q}$}{strong involution corresponding to quasisplit real form}

Given a strong involution $\delta$ we set
\begin{equation}
  \label{kdef}
  K = K_{\delta} \subset G
\end{equation}\nomenclature{$K$}{fixed-point set of a strong involution} to be the fixed-point subgroup of  $\mathrm{Int}(\delta)$.  The real form $G(\mathbb{R}, \delta)$ contains 
\begin{equation*}
K(\mathbb{R})=G(\mathbb{R}, \delta)\cap K
\end{equation*}
as a maximal compact subgroup and is determined by $K$ (\cite{AVParameters}*{(5f)-(5g)}).
By a representation of $G(\mathbb{R}, \delta)$ we usually mean an admissible
$(\mathfrak{g},K)$-module, although we will need veritable admissible group
representations (\cite{greenbook}*{Definition 1.1.5}) in Section \ref{whitsec}.
A \emph{representation of a strong involution} is a pair
$(\pi,\delta)\nomenclature{$(\pi,\delta)$}{representation of a strong involution}$ in which $\delta$ is a strong involution and $\pi$ is
an admissible $(\mathfrak{g},K)$-module. We let $\Pi(G(\mathbb{R}, \delta))$ be the set of
equivalence classes of irreducible representations $(\pi',\delta')$ of strong
involutions in which $\delta'$ is equivalent to $\delta$.
Let 
\begin{equation}
  \label{purereps}
\Pi(G/\mathbb{R}) =  \coprod_{\delta} \Pi(G(\mathbb{R}, \delta))
\nomenclature{$\Pi(G/\mathbb{R})$}{set of equivalence classes of irreducible representations}
\end{equation}
be the disjoint union over the (equivalence classes of) pure strong involutions $\delta$.\footnote{Warning! Identical notation is used in \cite{ABV} in which $\delta$ runs over all, not necessarily pure, involutions.}

Returning to the more familiar territory of L-groups, we fix a pinning
$$\left({^{\vee} B, {^\vee}H, \{ X_{{^\vee}\alpha}} \}\right)$$
of ${^\vee}G$.
The previous two pinnings and the involution $\delta_{0}$ fix an
involution ${^\vee} \delta_{0}$ of ${^\vee} G$ as prescribed in
\cite{AVParameters}*{(12)}. The group
$${^\vee}G^{\Gamma}={^\vee} G\rtimes\langle{^\vee}\delta_{0}\rangle
\nomenclature{${^\vee} G^{\Gamma}$}{L-group of $G$}
$$
is the L-group of our inner class. 

Suppose $\lambda$ is a semisimple element of the complex Lie algebra ${^\vee}\mathfrak{g}$.  After
conjugating by ${^\vee} G$ we may assume $\lambda\in {^\vee}\mathfrak{h}$.  Using the
canonical isomorphism ${^\vee}\mathfrak{h} \simeq \mathfrak{h}^*$ we identify $\lambda$ with an
element of $\mathfrak{h}^*$, and hence via the Harish-Chandra homomorphism, with
an infinitesimal character for $G$. This construction depends only on
the ${^\vee} G$-orbit of $\lambda$. We refer to a semisimple element
$\lambda\in{^\vee} \mathfrak{g}$, 
or a ${^\vee} G$-orbit ${^\vee}\mathcal{O} \subset  {^\vee}\mathfrak{g}$ of semisimple elements, as an \emph{infinitesimal
character} for $G$.  Let $$\Pi({^\vee}\mathcal{O},G/\mathbb{R}) \subset \Pi(G/\mathbb{R})
\nomenclature{$\Pi({^\vee}\mathcal{O},G/\mathbb{R})$}{set of equivalence classes of representations with infinitesimal
character ${^\vee}\mathcal{O}$}
$$ be the
representations (of pure strong involutions) with infinitesimal
character ${^\vee}\mathcal{O}$.
\nomenclature{${^\vee}\mathcal{O}$}{infinitesimal character}

Let $P\left({^\vee}G^{\Gamma} \right)\nomenclature{$P\left({^\vee}G^{\Gamma}\right)$}{set of quasiadmissible homomorphisms}{}$ be the set of \emph{quasiadmissible} homomorphisms
$\phi:W_{\mathbb{R}} \rightarrow {^\vee}G^{\Gamma}$ (\cite{ABV}*{Definition 5.2}). 
There is an infinitesimal character associated to $\phi\in P({^\vee}G^{\Gamma})$
(\cite{ABV}*{Proposition 5.6}).
Let
\begin{equation}
  \label{qlhomomorphisms}
  P\left({^\vee} \mathcal{O}, {^\vee}G^{\Gamma} \right)
\end{equation}
\nomenclature{$P\left({^\vee}\mathcal{O},{^\vee}G^{\Gamma}\right)$}{set of quasiadmissible homomorphisms with infinitesimal character ${^\vee}\mathcal{O}$}
be the set of quasiadmissible homomorphisms with
infinitesimal character ${^\vee}\mathcal{O}$.  The group ${^\vee}G$ acts on
$P\left({^\vee} \mathcal{O},{^\vee}G^{\Gamma} \right)$ by conjugation.  It is to  the set of ${^\vee}G$-orbits
\begin{equation}
  \label{lparameters}
  P({^\vee}\mathcal{O}, {^\vee}G^{\Gamma})/ {^\vee}G
\end{equation}
that the Langlands correspondence, in its original form, assigns L-packets of representations.

Another great innovation of \cite{ABV} is the introduction of the complex variety $X({^\vee}\mathcal{O}, {^\vee}G^{\Gamma})$ of \emph{geometric parameters}, which lies between (\ref{qlhomomorphisms}) and (\ref{lparameters}) (\cite{ABV}*{Definition 6.9}).  It may be regarded as a set of equivalence classes in $P({^\vee}\mathcal{O}, {^\vee}G^{\Gamma})$ upon which ${^\vee}G$ still acts by conjugation with finitely many orbits (\cite{aam}*{Section 2.2}, \cite{ABV}*{Proposition 6.16}).  Furthermore, the quotient map
\begin{equation}
  \label{abv6.17}
  P({^\vee}\mathcal{O}, {^\vee}G^{\Gamma}) \rightarrow X({^\vee}\mathcal{O}, {^\vee}G^{\Gamma})
\end{equation}
passes to a bijection at the level of ${^\vee}G$-orbits   (\cite{ABV}*{Proposition 6.17}).

Let
$S\subset X({^\vee}\mathcal{O},{^\vee}G^{\Gamma})$ be a ${^\vee} G$-orbit,
and for $p\in S$ let ${^\vee} G_p=\mathrm{Stab}_{{^\vee} G}(p)$.
A \emph{pure complete geometric parameter} for $X({^\vee}\mathcal{O},{^\vee}G^{\Gamma})$ is a pair
$(S,\tau_S)\nomenclature{$(S,\tau_S)$}{pure complete geometric parameter}$ where $\tau_S
\nomenclature{$\tau_S$}{irreducible representation of component group}
$
is (an equivalence class of) an irreducible representation of the component group ${^\vee} G_p/({^\vee} G_p)^0$ (\cite{ABV}*{Definitions 7.1 and 7.6}).
We denote the set of pure complete geometric
parameters for $X({^\vee}\mathcal{O}, {^\vee}G^{\Gamma})\nomenclature{$\Xi({^\vee}\mathcal{O}, {^\vee}G^{\Gamma})$}{set of pure complete geometric
parameters}$ by
$\Xi({^\vee}\mathcal{O},{^\vee}G^{\Gamma})$.

A special case of the local Langlands correspondence as stated in \cite{ABV}*{Theorem 10.11} is a bijection
\begin{equation}
\label{localLanglandspure}
\Pi\left({^\vee}\mathcal{O},G/\mathbb{R} \right)\longleftrightarrow \Xi\left({^\vee}\mathcal{O},{^\vee}G^{\Gamma}\right)
\end{equation}
between representations of pure strong involutions and pure complete geometric parameters.
It is important to bear in mind that the left-hand side of (\ref{localLanglandspure})
contains the subset $\Pi({^\vee}\mathcal{O},G(\mathbb{R},\delta_q))$ of representations of the quasisplit form of $G$.

\subsection{Extended groups for unitary groups and the complex general linear group}
\label{extended}

We specialize the discussion of the previous section to two groups, each with a fixed inner class.  The first group is $\mathrm{GL}_{N}$.  When $N$ is even we fix  the inner class whose quasisplit form is the indefinite unitary group $\mathrm{U}(N/2,N/2)$.  When $N$ is odd we fix the inner class to contain the quasisplit form $\mathrm{U}((N-1)/2,(N+1)/2)$.
The second group is
$$\mathrm{R}_{\mathbb{C}/\mathbb{R}} \mathrm{GL}_{N} = \mathrm{GL}_{N} \times \mathrm{GL}_{N}$$
together with the  inner class whose  quasisplit form is $\mathrm{GL}_{N}(\mathbb{C})$. \nomenclature{$\mathrm{R}_{\mathbb{C}/\mathbb{R}} \mathrm{GL}_{N}$}{restriction of scalars of $\mathrm{GL}_{N}$}

Let us begin with the unitary groups.  Fix the usual pinning for
$\mathrm{GL}_{N}$ in which $B$ is the upper-triangular Borel subgroup,
$H$ is the diagonal subgroup, and $X_{\alpha}$ is the matrix with $1$
in the entry corresponding to the simple root $\alpha$ and zeroes elsewhere.  
Regardless of whether $N$ is even or odd the specified inner classes contain the compact real form $\mathrm{U}(N)= \mathrm{U}(N,0)$.  This implies that $\delta_{0}$  acts as the trivial automorphism on $\mathrm{GL}_{N}$ (\cite{AVParameters}*{(5)}).  Consequently, the extended group for the inner class of unitary groups is
$$\mathrm{GL}_{N}^{\Gamma} = \mathrm{GL}_{N} \times \langle \delta_{0} \rangle \cong \mathrm{GL}_{N} \times \mathbb{Z}/ 2 \mathbb{Z}.$$
It follows from  \cite{AVParameters}*{(12c)} that the involution ${^\vee}\delta_{0}$ acts on ${^\vee}\mathrm{GL}_{N}$ as
$${^\vee}\delta_{0}(g) = \tilde{J} \, (g^{-1})^{\intercal} \, \tilde{J}^{-1}, \quad
  g \in {^\vee}\mathrm{GL}_{N}.$$
  The corresponding L-group  is
\begin{equation}
\label{unitarylgroup}
      {^\vee}\mathrm{GL}_{N}^{\Gamma} = {^\vee}\mathrm{GL}_{N} \rtimes \langle {^\vee}\delta_{0} \rangle.
\end{equation}
The strong involutions and pure strong involutions of real unitary
groups are presented in \cite{adams11}*{Section 9}.

For the group $\mathrm{R}_{\mathbb{C}/\mathbb{R}} \mathrm{GL}_{N}$ and the inner class containing $\mathrm{GL}_{N}(\mathbb{C})$, we shall follow \cite{Mok}*{Section 2.1} and begin with the L-group first.  The L-group is defined as
\begin{equation}
  \label{lgroupglnc}
{^\vee} \mathrm{R}_{\mathbb{C}/\mathbb{R}} \mathrm{GL}_{N}^{\Gamma}= {^\vee}(\mathrm{GL}_{N} \times \mathrm{GL}_{N})^{\Gamma} =
({^\vee}\mathrm{GL}_{N} \times {^\vee}\mathrm{GL}_{N}) \rtimes \langle {^\vee}\delta_{0} \rangle,
\end{equation}
where the involution ${^\vee}\delta_{0}$ is now defined by
$${^\vee}\delta_{0} (g_{1}, g_{2}) = (g_{2}, g_{1}), \quad g_{1}, g_{2} \in {^\vee} \mathrm{GL}_{N}.$$

Let us now fix the pinning $(B,H, \{X_{\alpha}\})$ for $\mathrm{GL}_{N} \times \mathrm{GL}_{N}$ (and its dual) by taking $B$ to be the direct product of the upper-triangular subgroups, $H$ to be the direct product of the diagonal subgroups and $\{X_{\alpha}\}$ to be the union of  simple root vectors for each of the two factors in the direct product.  Then,
by the prescription \cite{AVParameters}*{(12c)}, the involution $\delta_{0}$ on  $\mathrm{R}_{\mathbb{C}/\mathbb{R}} \mathrm{GL}_{N} = \mathrm{GL}_{N} \times \mathrm{GL}_{N}$ is defined by
\begin{equation}
  \label{del0action}
  \delta_{0}(g_{1}, g_{2}) = (\tilde{J} (g_{2}^{-1})^{\intercal} \tilde{J}^{-1}, \tilde{J} (g_{1}^{-1})^{\intercal} \tilde{J}^{-1}),  \quad g_{1}, g_{2} \in \mathrm{GL}_{N},
\end{equation}
and the extended group of $\mathrm{R}_{\mathbb{C}/\mathbb{R}} \mathrm{GL}_{N}$ is
\begin{equation}
  \label{egroupglnc}
\mathrm{R}_{\mathbb{C}/\mathbb{R}} \mathrm{GL}_{N}^{\Gamma} = (\mathrm{GL}_{N} \times \mathrm{GL}_{N})^{\Gamma} =
(\mathrm{GL}_{N} \times \mathrm{GL}_{N}) \rtimes \langle \delta_{0} \rangle.
\end{equation}
It is coincidental that $\delta_{0}$ defines the same automorphism as $\vartheta$ (\ref{vartheta}).  The real form $\sigma$ associated to $\delta_{0}$ is given by  composing $\delta_{0}$ with  the compact real form (\cite{AVParameters}*{(5)}).    This turns out to be
\begin{equation}
  \label{galoisglnc}
  \sigma (g_{1}, g_{2}) = (\tilde{J} \bar{g}_{2} \tilde{J}^{-1}, \tilde{J} \bar{g}_{1} \tilde{J}^{-1}), \quad g_{1}, g_{2} \in \mathrm{GL}_{N}.
\end{equation}
The group $\mathrm{R}_{\mathbb{C}/\mathbb{R}} \mathrm{GL}_{N}(\mathbb{R}, \delta_{0})$ from (\ref{formmap}) is by definition the fixed-point subgroup of $\sigma$,
$$\mathrm{R}_{\mathbb{C}/\mathbb{R}} \mathrm{GL}_{N}(\mathbb{R}, \delta_{0}) = \{ (g, \tilde{J} \bar{g} \tilde{J}^{-1}): g \in \mathrm{GL}_{N} \} \cong \mathrm{GL}_{N}(\mathbb{C}).$$
In what follows we will reserve the notation $\mathrm{GL}_{N}(\mathbb{C})$ for this particular \emph{real form}, and reserve the notation $\mathrm{GL}_{N}$ for the absolute theory of reductive groups. 
\begin{lem}
  \label{1cohomology}
  The Galois cohomology sets
  $$H^{1}(\mathbb{R}, \mathrm{R}_{\mathbb{C}/\mathbb{R}} \mathrm{GL}_{N}) \mbox{  and  } H^{1}(\mathbb{R}, \mathrm{R}_{\mathbb{C}/\mathbb{R}} \mathrm{GL}_{N}/ Z( \mathrm{R}_{\mathbb{C}/\mathbb{R}} \mathrm{GL}_{N}))$$
are both trivial.  In particular, $\mathrm{GL}_{N}(\mathbb{C})$ is the only real form in its inner class, there is only one equivalence class of pure strong involutions, and this equivalence class corresponds to  $\mathrm{GL}_{N}(\mathbb{C})$ via (\ref{puremap}).
\end{lem}
\begin{proof}
We begin with a cocycle  $z \in Z^{1}(\mathbb{R}, \mathrm{R}_{\mathbb{C}/\mathbb{R}} \mathrm{GL}_{N})$.  
Here, the implicit  action of the non-trivial element  $\sigma \in \Gamma$ on $\mathrm{R}_{\mathbb{C}/\mathbb{R}} \mathrm{GL}_{N}$ is given by (\ref{galoisglnc}).
Suppose $z(\sigma) = (g_{1}, g_{2})$.  Then by definition,
$$(1,1) = z(\sigma^{2}) = z(\sigma) \, \sigma(z(\sigma)) = (g_{1} \tilde{J} \bar{g}_{2} \tilde{J}^{-1}, g_{2} \tilde{J} \bar{g}_{1} \tilde{J}^{-1}),$$
which implies
$$z(\sigma) = (g_{1}, \tilde{J} \bar{g}_{1}^{-1} \tilde{J}^{-1}).$$
The cocycle $z$ is trivial in  $H^{1}(\mathbb{R}, \mathrm{R}_{\mathbb{C}/\mathbb{R}} \mathrm{GL}_{N})$ since
$$z(\sigma) = (g_{1},1) \ \sigma((g_{1},1)^{-1}).$$
This proves the triviality of $H^{1}(\mathbb{R}, \mathrm{R}_{\mathbb{C}/\mathbb{R}} \mathrm{GL}_{N})$.  The proof of the triviality of 
the second cohomology set follows in the same manner.
\end{proof}
Lemma \ref{1cohomology} reduces the set of representations in (\ref{purereps}) to the unique real form $\mathrm{GL}_{N}(\mathbb{C})$. Thus, we  write
\begin{equation}
  \label{reducepi}
  \Pi( {^\vee}\mathcal{O}, (\mathrm{R}_{\mathbb{C}/\mathbb{R}} \mathrm{GL}_{N})/ \mathbb{R}) = \Pi({^\vee} \mathcal{O}, \mathrm{GL}_{N}(\mathbb{C})).
  \nomenclature{$\Pi({^\vee} \mathcal{O}, \mathrm{GL}_{N}(\mathbb{C}))$}{set of infinitesimal equivalence classes of irreducible representations of $\mathrm{GL}_{N}(\mathbb{C})$ with infinitesimal character ${^\vee}\mathcal{O}$}
\end{equation}

This is an opportune moment to bring up two peculiarities of $\mathrm{R}_{\mathbb{C}/\mathbb{R}} \mathrm{GL}_{N}$ that will be of importance later.  The first is that we could equally well have reversed the roles of the L-group and extended group by defining the L-group as (\ref{egroupglnc}) and the extended group as (\ref{lgroupglnc}).  Indeed, with this reversal the (unique) real form corresponding to the extended group is the fixed-point set of
$$(g_{1}, g_{2}) \mapsto ((\bar{g}_{2}^{\intercal})^{-1}, (\bar{g}_{1}^{\intercal})^{-1})$$
(\cite{AVParameters}*{(5)}), which is easily seen to be
$$\{ (g, (\bar{g}^{\intercal})^{-1}) : g \in \mathrm{GL}_{N}\} \cong \mathrm{GL}_{N}(\mathbb{C}).$$
The recovery of the same inner class, namely $\mathrm{GL}_{N}(\mathbb{C})$, under this reversal shall be useful when we explore Vogan duality in Section \ref{duality}.  

The second peculiarity has to do with the connectedness of centralizers in the dual group ${^\vee}\mathrm{R}_{\mathbb{C}/\mathbb{R}} \mathrm{GL}_{N} = {^\vee}\mathrm{GL}_{N} \times {^\vee}\mathrm{GL}_{N}$.  It is well-known that the centralizer of the image of any L-parameter for the general linear group $\mathrm{GL}_{N}(\mathbb{R})$ is connected.  This peculiarity is shared by $\mathrm{R}_{\mathbb{C}/\mathbb{R}} \mathrm{GL}_{N}(\mathbb{R}) = \mathrm{GL}_{N}(\mathbb{C})$ and may be seen as follows.  To lighten the notation take $N=2$.  It is a simple exercise to show that after conjugation any L-parameter $\phi \in P({^\vee} \mathrm{R}_{\mathbb{C}/\mathbb{R}} \mathrm{GL}_{2}^{\Gamma})$ may be taken to have the form
$$\phi(z) =\left( \begin{bmatrix} z^{\lambda_{1}}\bar{z}^{\lambda_{1}'} & 0 \\
0 & z^{\lambda_{2}}  \bar{z}^{\lambda_{2}'} \end{bmatrix},  \begin{bmatrix} z^{\lambda'_{1}}\bar{z}^{\lambda_{1}} & 0 \\
  0 & z^{\lambda_{2}'}  \bar{z}^{\lambda_{2}} \end{bmatrix}, 1 \right), \quad z \in \mathbb{C}^{\times}$$
$$\phi(j) = \left( \begin{bmatrix} (-1)^{(\lambda_{1} - \lambda_{1}')/2} & 0 \\
0 & (-1)^{(\lambda_{2} - \lambda_{2}')/2} \end{bmatrix}, \begin{bmatrix} (-1)^{(\lambda_{1} - \lambda_{1}')/2} & 0 \\
0 & (-1)^{(\lambda_{2} - \lambda_{2}')/2} \end{bmatrix} , {^\vee}\delta_{0} \right)$$
for some $\lambda_{1}, \lambda_{1}', \lambda_{2}, \lambda_{2}' \in \mathbb{C}$.  It is not difficult to see that the centralizer of $\phi(j)$ in ${^\vee}\mathrm{GL}_{2} \times {^\vee}\mathrm{GL}_{2}$ reduces 
the further computation of the centralizer of
$\phi(\mathbb{C}^{\times})$ to the well-known case of
${^\vee}\mathrm{GL}_{2}$, where it is known to be connected.  The
connectedness of the centralizer of $\phi$ is the same as the
triviality of the component group, which we write as 
\begin{equation}
  \label{compgptriv}
  ({^\vee}\mathrm{R}_{\mathbb{C}/\mathbb{R}} \mathrm{GL}_{N})_{\phi}/ ({^\vee}\mathrm{R}_{\mathbb{C}/\mathbb{R}} \mathrm{GL}_{N})_{\phi}^{0} = \{ 1\}.
  \end{equation}

\subsection{Atlas parameters for $\mathrm{R}_{\mathbb{C}/\mathbb{R}} \mathrm{GL}_{N}$}
\label{atlasparam}

We introduce another set of parameters for representations of $\mathrm{R}_{\mathbb{C}/\mathbb{R}} \mathrm{GL}_{N}$ which may be used in place of the complete geometric parameters of Section \ref{extgroups}.  These \emph{Atlas parameters} are convenient for Hecke algebra computations and are
well-suited to a description of Vogan duality  (see Section \ref{duality}). 
The main references for this section are \cite{Adams-Fokko} and
\cite{AVParameters}*{Section 3}. 

We start by working in the context of the extended group \eqref{egroupglnc}.  We take $H \subset \mathrm{R}_{\mathbb{C}/\mathbb{R}} \mathrm{GL}_{N}$ to be the direct product of the diagonal subgroups as in the pinning of the previous section.  
Following \cite{AVParameters}*{Section 3} we set
$$
\mathcal{X}_{{^\vee}\rho}=\left\{\delta \in \mathrm{Norm}_{(\mathrm{GL}_N \times \mathrm{GL}_{N}) \delta_{0}}(H):
\delta^2=\exp(2\uppi i\, {^\vee}\rho)\right\}/H 
\nomenclature{$\mathcal{X}_{{^\vee}\rho}$}{
set of conjugacy classes of strong involutions 
 with infinitesimal
cocharacter ${^\vee}\rho$}
$$
where the quotient is by the conjugation action of $H$.  This is a set
of $H$-conjugacy classes of strong involutions with \emph{infinitesimal
cocharacter} ${^\vee}\rho$ (\cite{AVParameters}*{(16e)}, \emph{cf.} (\ref{deltaq})).  By Lemma \ref{1cohomology}, these strong
involutions are all pure and correspond to the real form
$\mathrm{GL}_{N}(\mathbb{C})$. 

We fix a $\vartheta$-fixed, regular, integrally dominant
element $\lambda \in {^\vee} \mathfrak{h}$. This means
\begin{equation}\label{regintdom}  
\begin{aligned}
     \vartheta(\lambda) &= \lambda\\
   \langle \lambda, {^\vee}\alpha
   \rangle &\neq  0, \quad \alpha\in R(\mathrm{R}_{\mathbb{C}/\mathbb{R}} \mathrm{GL}_N,H)\\
    \langle \lambda, {^\vee} \alpha \rangle &\notin \{ -1, -2, -3,
   \ldots \}, \ \alpha \in R^{+}(\mathrm{R}_{\mathbb{C}/\mathbb{R}} \mathrm{GL}_{N},H). 
\end{aligned}
\end{equation}
The ${^\vee}\mathrm{R}_{\mathbb{C}/\mathbb{R}} \mathrm{GL}_{N}$-orbit ${^\vee}\mathcal{O}$ of $\lambda$ will be the infinitesimal character of our representations of $\mathrm{GL}_N(\mathbb{C})$. 
The assumption of integral dominance is harmless 
(\cite{AVParameters}*{Lemma 4.1}). We 
shall remove the regularity assumption at the beginning of Section
\ref{equalapacketsing}.

The action of $\delta_{0}$ in the extended group (\ref{egroupglnc}) induces an action on the Weyl group
$W(\mathrm{R}_{\mathbb{C}/\mathbb{R}} \mathrm{GL}_{N},H)
\nomenclature{$W(\mathrm{R}_{\mathbb{C}/\mathbb{R}} \mathrm{GL}_{N},H)$}{Weyl group}
$.
Consider the set
\begin{equation}
\label{twistinv}
\left\{w\in W(\mathrm{R}_{\mathbb{C}/\mathbb{R}} \mathrm{GL}_N,H) : w\, \delta_{0}(w)=1\right\}.
\end{equation}
If $x\in\mathcal{X}_{{^\vee} \rho}$ then the action by conjugation of $x$ on
$H$ is equal to $w\delta_0$ for some $w$ in the set \eqref{twistinv}.
The map $x \mapsto w$ is surjective (\cite{AVParameters}*{Proposition 3.2}). Let
$\mathcal{X}_{{^\vee}\rho}^w$ be the fibre of this map over $w$, \emph{i.e.}
\begin{equation}
\label{e:p}
\mathcal{X}_{{^\vee}\rho}^w=\left\{x\in\mathcal{X}_{{^\vee}\rho} : xhx^{-1} =w\delta_0 \cdot h, \ \mbox{ for all }  h\in H \right\}.
\nomenclature{$\mathcal{X}_{{^\vee}\rho}^w$}{}
\end{equation}

Turning to the L-group (\ref{lgroupglnc}), we have an analogous set 
in which the infinitesimal cocharacter ${^\vee}\rho$
is replaced by the infinitesimal character $\lambda$, and $\delta_{0}$ is replaced by ${^\vee}\delta_{0}$, namely
$$
{^\vee}\mathcal{X}_{\lambda}=\left\{{^\vee}\delta \in \mathrm{Norm}_{{^\vee}\mathrm{GL}_N \, {^\vee}\delta_{0}}({^\vee}
H): {^\vee}\delta^2=\exp(2\uppi i\lambda)\right\}/\, {^\vee} H. 
\nomenclature{${^\vee}\mathcal{X}_{\lambda}$}{} 
$$
The analogue of the set (\ref{twistinv}) is
\begin{equation}
  \label{twistinvdual}
\{ w \in
W({^\vee}\mathrm{R}_{\mathbb{C}/\mathbb{R}} \mathrm{GL}_{N}, {^\vee}H): w {^\vee}\delta_{0}(w) =1\}
\end{equation}
and there is an obvious analogue of (\ref{e:p}), which we denote by  ${^\vee}\mathcal{X}_\lambda^w 
\nomenclature{${^\vee}\mathcal{X}_\lambda^w$}{}$.

If we identify ${^\vee}\mathrm{GL}_{N}$ with $\mathrm{GL}_{N}$ then the actions of $\delta_{0}$ and ${^\vee}\delta_{0}$ on $H$ are related by
$${^\vee}\delta_{0}(h) = w_{0} \delta_{0} (h^{-1}), \quad h \in H,$$
where $w_{0}$ is the long Weyl group element (\cite{AVParameters}*{(12c)}).
From this it is easily verified that for all $h \in H$ and  $w \in
W(\mathrm{R}_{\mathbb{C}/\mathbb{R}} \mathrm{GL}_{N},H)$
\begin{equation*}
w\, \delta_{0}(w) \cdot h = w\delta_{0} w \delta_{0} \cdot h = w w_{0} {^\vee}\delta_{0} ww_{0} {^\vee}\delta_{0} \cdot h =  w w_{0} {^\vee}\delta_{0}(ww_{0}) \cdot h.
\end{equation*}
It follows that  
$$w \mapsto ww_{0}$$
defines a bijection from (\ref{twistinv}) onto (\ref{twistinvdual}).  
This map allows us to pair any set $\mathcal{X}_{{^\vee}\rho}^{w}$ with the set ${^\vee}\mathcal{X}_{\lambda}^{ww_{0}}$.

The next result follows from \cite{Adams-Fokko}, \cite{ABV} and
\cite{AVParameters}*{Theorem 3.11}. Our proof follows that of \cite{aam}*{Lemma 2.2}.
\begin{lem}
  \label{XXXi}
  There is a canonical bijection
  $$\coprod_{\{w: w\delta_0(w)=1\}} \mathcal{X}_{{^\vee}\rho}^{w} \times
  {^\vee}\mathcal{X}_{\lambda}^{ww_{0}} \longleftrightarrow \Xi\left({^\vee}\mathcal{O},{^\vee}\mathrm{R}_{\mathbb{C}/\mathbb{R}} \mathrm{GL}_N^{\Gamma} \right). 
  $$
\end{lem}

\begin{proof}
  First of all $|\mathcal{X}_{{^\vee}\rho}^w|=1$  for all $w$. This follows from
  \cite{Adams-Fokko}*{Proposition 12.19(5)} which equates the cardinality with that of the component group of a dual Cartan subgroup.  As indicated in the first peculiarity near the end of Section \ref{extended}, the dual inner class consists of (products of) complex general linear groups.  The dual Cartan subgroup is therefore isomorphic to $N$ copies of $\mathbb{C}^{\times}$ and is evidently connected, so its component group is trivial.

The lemma is now reduced to defining a canonical bijection
  $$
  \coprod_{w\delta_0(w)=1} {^\vee}\mathcal{X}^{ww_{0}}_\lambda\longleftrightarrow \Xi\left({^\vee}\mathcal{O},{^\vee}\mathrm{R}_{\mathbb{C}/\mathbb{R}} \mathrm{GL}_N^{\Gamma}\right).
  $$
As explained in \cite{aam}*{Lemma 2.2}, the triviality of the component groups (\ref{compgptriv}) further reduces the task to defining a canonical bijection with ${^\vee} \mathrm{R}_{\mathbb{C}/\mathbb{R}} \mathrm{GL}_N$-orbits in  $X\left({^\vee}\mathcal{O}, {^\vee}\mathrm{R}_{\mathbb{C}/\mathbb{R}} \mathrm{GL}_{N}^{\Gamma} \right)$, or equivalently, with ${^\vee} \mathrm{R}_{\mathbb{C}/\mathbb{R}} \mathrm{GL}_N$-orbits of quasiadmissible homomorphisms.  The definition of the latter bijection is almost identical to the one in \cite{aam}*{Lemma 2.2} and is left to the reader.
\end{proof}
Together with  \eqref{localLanglandspure} this gives
\begin{thm}
  \label{paramsGLn}
Let ${^\vee}\mathcal{O}$ be the ${^\vee}\mathrm{R}_{\mathbb{C}/\mathbb{R}} \mathrm{GL}_N$-orbit of $\lambda$.
  There are canonical bijections
  $$
\coprod_{\{ w : w \delta_{0}(w) = 1\}}
\mathcal{X}^{w}_{{^\vee}\rho} \times {^\vee}\mathcal{X}^{ww_{0}}_{\lambda} \longleftrightarrow
\Xi\left({^\vee}\mathcal{O},{^\vee}\mathrm{R}_{\mathbb{C}/\mathbb{R}} \mathrm{GL}_N^{\Gamma}\right) \longleftrightarrow \Pi\left({^\vee}\mathcal{O},\mathrm{GL}_N(\mathbb{C})\right).
$$
\end{thm}
As in \cite{AVParameters}*{Theorem 3.11} 
 the bijection of  Theorem \ref{paramsGLn} is written as
\begin{equation}
\label{abvij}
\mathcal{X}_{{^\vee}\rho}^w\times {^\vee}\mathcal{X}_{\lambda}^{w w_0}\ni(x,y) \mapsto J(x,y,\lambda)
\end{equation}
We call the pair $(x,y)\nomenclature{$(x,y)$}{Atlas parameter}$ on the left the \emph{Atlas parameter} of the
irreducible representation $J(x,y, \lambda)\nomenclature{$J(x,y,\lambda)$}{Langlands quotient}$. 
By Lemma \ref{XXXi}, the Atlas parameter $(x,y)$ is equivalent to a unique complete geometric parameter $\xi \in \Xi({^\vee}\mathcal{O}, {^\vee}\mathrm{R}_{\mathbb{C}/\mathbb{R}} \mathrm{GL}_{N}^{\Gamma})$, and accordingly we define
$$\pi(\xi) = J(x,y, \lambda). \nomenclature{$\pi(\xi)$}{(equivalence class of the) irreducible Harish-Chandra module (representation) associated to $\xi$}$$
The representation $\pi(\xi)$ is the Langlands quotient of a standard representation which we denote by $M(\xi)$ or $M(x,y)$.  \nomenclature{$M(\xi), \ M(x,y)$}{(equivalence class of the) standard module (representation) associated to $\xi$}

\subsection{Twisted Atlas parameters for $\mathrm{R}_{\mathbb{C}/\mathbb{R}} \mathrm{GL}_N$}
\label{extrepsec}

We next describe the generalization of 
Theorem \ref{paramsGLn}
to the $\vartheta$-twisted setting, which involves
representations of the group
$\mathrm{GL}_N(\mathbb{C})\rtimes\langle\vartheta\rangle$.
We specialize the results of \cite{AVParameters}*{Sections 3-5}
to this case. Some of the more
complicated issues that arise in \cite{AVParameters}
do not occur for $\mathrm{R}_{\mathbb{C}/\mathbb{R}} \mathrm{GL}_{N}$.

We continue with the hypotheses of (\ref{regintdom}). Recall that both
${^\vee}\rho$
and $\lambda$ are fixed by $\vartheta$.
By Clifford theory, an irreducible representation of
$\mathrm{GL}_N(\mathbb{C})\rtimes\langle\vartheta\rangle$ restricted to $\mathrm{GL}_N(\mathbb{C})$
is either an irreducible $\vartheta$-fixed representation, or the direct
sum of two irreducible representations which are exchanged  by the action of $\vartheta$.
Since we shall be restricting our attention to twisted characters, we only need representations of the first type. 

By \cite{Christie-Mezo}*{Theorem 4.1} and Lemma \ref{XXXi}, the map
(\ref{abvij}) is $\vartheta$-equivariant.  Therefore $J(x,y, \lambda)$ is
$\vartheta$-stable if and only if
$(x,y)\in \mathcal{X}^{w}_{{^\vee}\rho} \times
{^\vee}\mathcal{X}^{{ww_{0}}}_{\lambda} $ is fixed by $\vartheta$.
Let
$$\Pi({^\vee}\mathcal{O}, \mathrm{GL}_{N}(\mathbb{C}))^\vartheta \subset \Pi({^\vee}\mathcal{O}, \mathrm{GL}_{N}(\mathbb{C}))
\nomenclature{$\Pi({^\vee}\mathcal{O}, \mathrm{GL}_{N}(\mathbb{C}))^\vartheta$}{
$\vartheta$-fixed equivalence classes of representations
}
$$
be the subset of $\vartheta$-stable irreducible representations and set
\begin{equation*}
W(\delta_{0},\vartheta)=\{w\in W(\mathrm{R}_{\mathbb{C}/\mathbb{R}} \mathrm{GL}_{N},H)   : w\delta_{0}(w)=1, w=\vartheta(w)\}
\end{equation*}
(\emph{cf.} (\ref{twistinv})).  The $\vartheta$-equivariance of (\ref{abvij})  allows us to restrict
Theorem \ref{paramsGLn} to these sets and we obtain 
\begin{cor}
\label{twistclass}
Suppose $\lambda$ satisfies the hypotheses of (\ref{regintdom}) and let ${^\vee}\mathcal{O}$ be its ${^\vee}\mathrm{R}_{\mathbb{C}/\mathbb{R}} \mathrm{GL}_{N}$-orbit.  Then there is a canonical bijection
$$
\coprod_{\{w\in W(\delta_{0},\vartheta)\}}\mathcal{X}_{{^\vee}\rho}^w\times {^\vee}\mathcal{X}^{ww_0}_\lambda
\longleftrightarrow
\Pi\left({^\vee}\mathcal{O},\mathrm{GL}_N(\mathbb{C})\right)^{\vartheta} 
$$
written $(x,y)\mapsto J(x,y,\lambda)$. 
\end{cor}

We now introduce the \emph{extended parameters} of
\cite{AVParameters}*{Sections 3-5},
and summarize some facts.
Fix $w\in W(\delta_{0},\vartheta)$. 
An \emph{extended parameter} for $w$ is a set 
\begin{equation*}
E  = (\uplambda,\uptau,\ell,t), \quad \uplambda,\uptau\in X^*(H),
\ \ell,t \in X_*(H)
\nomenclature{$(\uplambda,\uptau,\ell,t)$}{extended parameter}
\end{equation*}
satisfying certain conditions depending on $w$ (see
\cite{AVParameters}*{Definition 5.4}).\footnote{Warning!  The symbols
  $\uplambda$ and $\uptau$ here are not to be confused with symbols
  $\lambda$ and $\tau$ appearing elsewhere.  Note the slight
  difference in font.  We have chosen to use $\uplambda$ and $\uptau$
  for ease of comparison with \cite{AVParameters}.}  There is a
surjective map 
\begin{equation}
\label{extsurj}
E \mapsto (x(E), y(E))
\end{equation}
carrying extended parameters for $w$ to $\mathcal{X}_{{^\vee}\rho}^{w} \times {^\vee} \mathcal{X}_{\lambda}^{ww_{0}}$.  This map only 
depends on $\uplambda$ and $\ell$.
In addition,
$$
J(x(E),y(E),\lambda)\in\Pi({^\vee}\mathcal{O},\mathrm{GL}_N(\mathbb{C}))^\vartheta, 
$$
and every $\vartheta$-fixed irreducible representation arises this way.
The remaining parameters $\uptau$ and $t$ in $E$ define an irreducible representation $J(E,\lambda)$ of 
$\mathrm{GL}_N(\mathbb{R})\rtimes\langle\vartheta\rangle$ satisfying
$$
J(E,\lambda)|_{\mathrm{GL}_N(\mathbb{C})}=J(x(E),y(E),\lambda).
$$
The representation $J(x(E),y(E), \lambda)$ is determined by a
quasicharacter of a  Cartan subgroup of $\mathrm{GL}_{N}(\mathbb{C})$.
The representation $J(E,\lambda)$ is determined by the semidirect
product of this Cartan subgroup with an element $h \vartheta \in
\mathrm{R}_{\mathbb{C}/\mathbb{R}} \mathrm{GL}_{N} \rtimes \vartheta$ (\cite{AVParameters}*{(24e)}) and a
  choice of extension of the 
quasicharacter to the semidirect product.  The value of the extended
quasicharacter on the element $h\vartheta$ depends on a choice of sign
\cite{AVParameters}*{Definition 5.2}, and the square root of this sign is
given by 
\begin{equation*}
z(E)=i^{\langle \uptau,(1+w)t\rangle} \, (-1)^{\langle \uplambda,t\rangle}.
\end{equation*}
The preceding discussion is a specialization of a general framework to $\mathrm{R}_{\mathbb{C}/\mathbb{R}} \mathrm{GL}_{N} \rtimes \langle \vartheta \rangle$.  One of the special properties of $\mathrm{R}_{\mathbb{C}/\mathbb{R}} \mathrm{GL}_N$ is that the preimage of any $(x,y) \in \mathcal{X}_{{^\vee}\rho}^{w} \times {^\vee}\mathcal{X}_{\lambda}^{ww_{0}}$ under (\ref{extsurj}) has a \emph{preferred extended parameter} of the form 
\begin{equation}
  \label{pextparam}
  (\uplambda,\uptau,0,0).
\end{equation}
This comes down to the fact that $X_{{^\vee}\rho}^{w}$ is a singleton (see the proof of Lemma \ref{XXXi}).
In taking $t=0$ we see $z(\uplambda, \uptau,0,0)=1$, and this in turn amounts to taking the aforementioned semidirect product of the Cartan subgroup with $h\vartheta = \vartheta$, and setting the value of the extended quasicharacter at $\vartheta$ equal to $1$.  In this way, the preferred extended parameter defines a canonical extension 
\begin{equation}
\label{canext}
J(x,y, \lambda)^{+} = J((\uplambda, \uptau,0,0), \lambda)
\nomenclature{$J(x,y, \lambda)^{+}$}{Atlas extension}
\end{equation}
of $J(x,y,\lambda)$ to $\mathrm{GL}_{N}(\mathbb{C}) \rtimes \langle \vartheta \rangle$. 
We call this extension the \emph{Atlas extension} of $J(x,y, \lambda)$.

Going back to Theorem \ref{paramsGLn} and Corollary \ref{twistclass},
we may formulate the result as follows.
\begin{cor}
  \label{twistclass2}
  There is a natural bijection of $\vartheta$-fixed sets
  $$
  \coprod_{\{w\in W(\delta_{0},\vartheta)\}}\mathcal{X}_{{^\vee}\rho}^w\times {^\vee}\mathcal{X}^{ww_0}_\lambda
 \longleftrightarrow \Xi({^\vee}\mathcal{O},{^\vee}\mathrm{R}_{\mathbb{C}/\mathbb{R}} \mathrm{GL}_{N}^{\Gamma})^\vartheta \longleftrightarrow \Pi({^\vee}\mathcal{O},\mathrm{GL}_N(\mathbb{C}))^\vartheta.
$$ 
Furthermore, if $\xi \in \Xi({^\vee}\mathcal{O},{^\vee}\mathrm{R}_{\mathbb{C}/\mathbb{R}} \mathrm{GL}_{N}^{\Gamma})^\vartheta$  is identified with $(x,y)$ under the first bijection, then there is a canonical representation 
$$
\pi(\xi)^+ = J(x,y,\lambda)^{+}
\nomenclature{$\pi(\xi)^+$}{(equivalence class of) the Atlas extension of $\pi(\xi)$}
$$
extending $\pi(\xi)$ to $\mathrm{GL}_N(\mathbb{C})\rtimes\langle\vartheta\rangle$.
\end{cor}
The irreducible representation $\pi(\xi)^{+}$ is defined as the unique (Langlands) quotient of a representation 
$M(\xi)^+$
\nomenclature{$M(\xi)^+$}{(equivalence class of) the Atlas extension of $M(\xi)$}
such that $M(\xi)^{+}_{|\mathrm{GL}_{N}(\mathbb{C})} = M(\xi)$.  We call $\pi(\xi)^{+}$ and $M(\xi)^{+}$ the \emph{Atlas extensions} of $\pi(\xi)$ and $M(\xi)$ respectively.

\subsection{Grothendieck groups of characters and twisted characters}
\label{grothchar}

The setting for studying characters of reductive groups is the
Grothendieck group of admissible representations.
There is  a corresponding notion in the twisted setting.
In this section we establish the notation for these objects.

For the moment let $G$ be an arbitrary complex connected reductive group.  Fix a semisimple orbit ${^\vee}\mathcal{O}\subset{^\vee}\mathfrak{g}$.
Recall from Section \ref{extgroups} that $\Pi({^\vee}\mathcal{O},G/\mathbb{R})$ is the set of equivalence classes of irreducible 
representations $(\pi, \delta)$ of pure strong involutions with infinitesimal
character ${^\vee}\mathcal{O}$.  We
define $K\Pi({^\vee}\mathcal{O},G/\mathbb{R})
\nomenclature{$K\Pi({^\vee}\mathcal{O},G/\mathbb{R})$}{Grothendieck group of
representations of pure strong involutions}
$ to be the Grothendieck group of admissible
representations of pure strong involutions with infinitesimal
character ${^\vee}\mathcal{O}$ (see \cite{ABV}*{(15.5)-(15.6)}). We identify this with
the $\mathbb{Z}$-span of distribution characters of the irreducible
representations in $\Pi({^\vee}\mathcal{O},G/\mathbb{R})$, and refer to elements of this $\mathbb{Z}$-module as virtual characters.

For the unitary groups we often refer to the submodule of stable
characters for the quasisplit form. So we define
$$
K\Pi({^\vee}\mathcal{O},G(\mathbb{R},\delta_q))^{\mathrm{st}} \subset K\Pi({^\vee}\mathcal{O},G(\mathbb{R},\delta_q))
\nomenclature{$K\Pi({^\vee}\mathcal{O},G(\mathbb{R},\delta_q))^\stable$}{stable virtual characters}
$$
to be the subspace spanned by the (strongly) stable  virtual characters.
If we identify virtual characters with functions on $G(\mathbb{R},\delta_{q})$
these are the virtual characters $\eta$ which satisfy $\eta(g)=\eta(g')$
whenever  strongly regular semisimple elements $g,g' \in G(\mathbb{R},\delta_q)$ are $G$-conjugate.
See \cite{shelstad}*{Section 5} or \cite{ABV}*{Definition 18.2}.

We now turn the twisted setting of the inner class $\mathrm{GL}_{N}(\mathbb{C})$ of $\mathrm{R}_{\mathbb{C}/\mathbb{R}}\mathrm{GL}_N$, equipped with the involution $\vartheta$.
An immediate consequence of (\ref{reducepi})  is the expression of the Grothendieck group
$$
K\Pi(\mathrm{R}_{\mathbb{C}/\mathbb{R}} \mathrm{GL}_{N}/\mathbb{R}) = K \Pi(\mathrm{GL}_{N}(\mathbb{C}))
\nomenclature{$K \Pi(\mathrm{GL}_{N}(\mathbb{C}))$}{}
$$
in terms of the unique real form.
We define
$$
K\Pi({^\vee}\mathcal{O},\mathrm{GL}_N(\mathbb{C}))^\vartheta\subset K\Pi({^\vee}\mathcal{O},\mathrm{GL}_N(\mathbb{C}))
\nomenclature{$K\Pi({^\vee}\mathcal{O},\mathrm{GL}_N(\mathbb{C}))^\vartheta$}{}
$$
to be the submodule spanned by $\Pi({^\vee}\mathcal{O}, \mathrm{GL}_{N}(\mathbb{C}))^\vartheta$.  This is not the Grothendieck group of $\vartheta$-stable representations of $\mathrm{GL}_{N}(\mathbb{C})$, but we retain the ``$K$" to help align the object with its ambient Grothendieck group.
On the other hand we let 
\begin{equation}
\label{KPiOGLext}
K\Pi({^\vee}\mathcal{O},\mathrm{GL}_{N}(\mathbb{C}) \rtimes \langle \vartheta \rangle)
\end{equation}
be the veritable Grothendieck group of admissible representations of $\mathrm{GL}_{N}(\mathbb{C}) \rtimes \langle \vartheta \rangle$ with infinitesimal character ${^\vee}\mathcal{O}$.

We are ready to construct the $\mathbb{Z}$-module of twisted characters of $\mathrm{GL}_N(\mathbb{C})$.
An irreducible character in $K\Pi({^\vee}\mathcal{O},\mathrm{GL}_{N}(\mathbb{C}) \rtimes \langle \vartheta \rangle)$
is the usual distribution on $\mathrm{GL}_{N}
  (\mathbb{C}) \rtimes \langle \vartheta \rangle$ given by  $\mathrm{Tr}\,\pi$, where $\pi$ is an
irreducible representation of $\mathrm{GL}_{N}(\mathbb{C}) \rtimes
\langle \vartheta \rangle$.  The restriction of such a distribution
character to the non-identity component $\mathrm{GL}_{N}(\mathbb{C}) \rtimes
\vartheta$ (when non-zero) is what we we mean by an irreducible \emph{twisted character} of $\mathrm{GL}_{N}(\mathbb{C})$ (\emph{cf.}  \cite{aam}*{(42)}) .  We define
$$K\Pi({^\vee}\mathcal{O}, \mathrm{GL}_{N}(\mathbb{C}), \vartheta)$$
to be the
$\mathbb{Z}$-module generated by the irreducible twisted
characters of $\mathrm{GL}_{N}(\mathbb{C})$ of
infinitesimal character ${^\vee}\mathcal{O}$.
\nomenclature{$K\Pi({^\vee}\mathcal{O}, \mathrm{GL}_{N}(\mathbb{C}), \vartheta)$}{virtual twisted characters}

As noted in Section \ref{extrepsec}, an irreducible representation of $\mathrm{GL}_{N}(\mathbb{C}) \rtimes \langle \vartheta \rangle$
restricts either to an irreducible $\vartheta$-fixed representation of $\mathrm{GL}_N(\mathbb{C})$,
or to a direct sum $\pi\oplus (\pi \circ \vartheta)$ of inequivalent irreducible representations. In the second case the twisted character is $0$, so we only need to consider the first case.  The first case describes the irreducible representations in $K\Pi({^\vee}\mathcal{O},\mathrm{GL}_N(\mathbb{C}))^\vartheta$.
If $\pi\in\Pi({^\vee}\mathcal{O},\mathrm{GL}_N(\mathbb{C}))^\vartheta$ then
it has two extensions $\pi^{\pm}$ to $\mathrm{GL}_{N}(\mathbb{C}) \rtimes \langle \vartheta \rangle$, satisfying
\begin{equation}
  \label{minuspi}
\pi^-(\vartheta)=-\pi^+(\vartheta).
\nomenclature{$\pi^{-}$}{}
\end{equation}
Consequently the twisted characters of $\pi^{\pm}$ agree up to sign.  If we set $U_{2} = \{ \pm 1\}\nomenclature{$U_2=\{\pm 1\}$}{}$ then it follows that the homomorphism
$$K \Pi({^\vee}\mathcal{O}, \mathrm{GL}_{N}(\mathbb{C}))^\vartheta
\otimes _{\mathbb{Z}} \mathbb{Z}[U_{2}] \rightarrow K
\Pi({^\vee}\mathcal{O},\mathrm{GL}_{N}(\mathbb{C}), \vartheta),$$
which restricts the distribution character of $\pi(\xi)^{+}$ to the
non-identity component, is surjective.  By (\ref{minuspi}), the homomorphism
passes to an isomorphism 
\begin{subequations}
\renewcommand{\theequation}{\theparentequation)(\alph{equation}}  
\label{twistgroth1}  
\begin{equation}
K\Pi({^\vee}\mathcal{O}, \mathrm{GL}_{N}(\mathbb{C}),
\vartheta)\cong K \Pi({^\vee}\mathcal{O}, \mathrm{GL}_{N}(\mathbb{C}))^\vartheta
\otimes _{\mathbb{Z}} \mathbb{Z}[U_{2}]/ \langle (\pi \otimes 1 ) +
  (\pi\otimes -1)\rangle,
\end{equation}
where the quotient runs over $\pi\in \Pi({^\vee}\mathcal{O},\mathrm{GL}_N(\mathbb{C}))^\vartheta$.
The map carrying
$\pi(\xi) \in \Pi({^\vee}\mathcal{O},
\mathrm{GL}_{N}(\mathbb{C}))^\vartheta$ to the twisted character of the Atlas extension $\pi(\xi)^{+}$
extends to an isomorphism
\begin{equation*}
 K\Pi({^\vee}\mathcal{O},\mathrm{GL}_N(\mathbb{C}))^\vartheta \cong K\Pi({^\vee}\mathcal{O},\mathrm{GL}_N(\mathbb{C}),\vartheta).
\end{equation*}
\end{subequations}
We again
remind the reader  that the $\mathbb{Z}$-modules appearing
in (\ref{twistgroth1}) are not Grothendieck groups in any natural fashion, notwithstanding the
appearance of the ``$K$''.

\section{Sheaves, Pairings and Characteristic Cycles}
\label{sheaves}

Suppose $\psi_G$ is an Arthur parameter for $G$ as in
(\ref{aparameter}).  In this section we give more details on the
definition of the ABV-packet $\Pi_{\psi_{G}}^{\mathrm{ABV}}$ and its stable
virtual character $\eta_{\psi_{G}}^{\mathrm{ABV}}$ (\ref{mainthm}).  The results apply in the more general context of
complex connected reductive groups $G$ (\cite{ABV}*{Sections 19, 22}).
However, for this section $G$ will be
$\mathrm{GL}_{N}$ or $\mathrm{R}_{\mathbb{C}/\mathbb{R}} \mathrm{GL}_{N}$, with the setup of Section \ref{extended}.

The definitions depend on a pairing between characters and sheaves. 
We also define a pairing between \emph{twisted} characters and \emph{twisted}
sheaves for $\mathrm{R}_{\mathbb{C}/\mathbb{R}}\mathrm{GL}_{N}$
(\cite{Christie-Mezo}*{Sections 5-6}, \cite{aam}*{Section 3}). The key properties of this
twisted pairing are listed in this section and shall be proved in
Section \ref{pairings}. 

\subsection{The pairing and the ABV-packets in the non-twisted case}
\label{sec:ABV-packetsdef}
Let $\phi_{\psi_G}$ be the Langlands parameter associated to $\psi_{G}$
(\cite{ABV}*{Definition 22.4}), ${^\vee}\mathcal{O}$ be the infinitesimal character of $\phi_{\psi_G}$, and
$S_{\psi_{G}} \subset X({^\vee}\mathcal{O},{^\vee}G^{\Gamma})$ be the corresponding orbit
(\ref{abv6.17}).  
Recall that  $\Xi({^\vee}\mathcal{O},{^\vee}G^{\Gamma})$ is the set of
pure complete geometric parameters.

Let $\mathcal{C}(X({^\vee}\mathcal{O},{^\vee}G^{\Gamma}))$ be
\nomenclature{$\mathcal{C}(X({^\vee}\mathcal{O},{^\vee}G^{\Gamma}))$}{category
  of constructible sheaves} 
the category of ${^\vee}G$-equivariant constructible sheaves of
complex vector spaces on $X({^\vee}\mathcal{O},{^\vee}G^{\Gamma})$. This is
an abelian category and its simple objects are parameterized by the
set of complete geometric parameters  $\xi = (S, \tau_{S}) \in
\Xi({^\vee}\mathcal{O}, {^\vee}G^{\Gamma})$ as follows.
Choose $p\in S$, let  ${^\vee}G_p=\mathrm{Stab}_{{^\vee}G}(p)$, and choose 
a character $\tau_\xi$  of the component group of ${^\vee}G_p$ so that
$(p,\tau_\xi)$ is a representative of $\xi$. 
Then $\tau_\xi$ pulled back to ${^\vee}G_p$ defines 
an algebraic vector bundle
\begin{equation}
  \label{bundle}
        {^\vee}G \times_{{^\vee}G_{p} }V \rightarrow S.
\end{equation}
The sheaf of
sections of this vector bundle is, by definition, a
${^\vee}G$-equivariant local system on $S$ (\cite{ABV}*{Section 7, Lemma 7.3}). Extend this local
system to the closure $\bar{S}$ by zero and then take the direct image
into $X({^\vee}\mathcal{O},{^\vee}G^{\Gamma})$ to obtain an irreducible
(\emph{i.e.} simple) ${^\vee}G$-equivariant constructible sheaf
denoted by $\mu(\xi)$ (\cite{ABV}*{(7.10)(c)}). 
\nomenclature{$\mu(\xi)$}{irreducible constructible sheaf}

Now let
$\mathcal{P}(X({^\vee}\mathcal{O},{^\vee}G^{\Gamma}))$ be the abelian category of
\nomenclature{$\mathcal{P}(X({^\vee}\mathcal{O},{^\vee}G^{\Gamma}))$}{category of perverse sheaves}
${^\vee}G$-equivariant perverse sheaves of complex vector spaces on
$X({^\vee}\mathcal{O},{^\vee}G^{\Gamma})$  (\cite{Lunts}*{Section 5}).  The simple
objects of $\mathcal{P}(X({^\vee}\mathcal{O},{^\vee}G^{\Gamma}))$ are defined
from $\xi = (S, \tau_{S}) \in \Xi({^\vee}G^{\Gamma}, {^\vee}\mathcal{O})$
and the algebraic vector bundle (\ref{bundle}) by taking the
\emph{intermediate} extension (\cite{bbd}*{Section 2}) to the closure $\bar{S}$
instead of the extension by zero.  This is denoted $P(\xi)$
 (\cite{ABV}*{(7.10)(d)}).
It is an irreducible ${^\vee}G$-equivariant
perverse sheaf on $X({^\vee}\mathcal{O},{^\vee}G^{\Gamma})$.
\nomenclature{$P(\xi)$}{irreducible perverse sheaf}

The Grothendieck groups of the two categories
$\mathcal{C}(X({^\vee}\mathcal{O},{^\vee}G^{\Gamma}))$ and
$\mathcal{P}(X({^\vee}\mathcal{O},{^\vee}G^{\Gamma}))$ are canonically isomorphic
(\cite{bbd}, \cite{ABV}*{Lemma 7.8}).  We identify the two Grothendieck
groups via this 
isomorphism and denote them by $KX({^\vee}\mathcal{O},{^\vee}G^{\Gamma})$.
\nomenclature{$KX({^\vee}\mathcal{O},{^\vee}G^{\Gamma})$}{Grothendieck group}
This Grothendieck group has two natural bases
$$\{\mu(\xi) : \xi \in
\Xi({^\vee}\mathcal{O},{^\vee}G^{\Gamma})\} 
\quad \mbox{ and }\quad \{P(\xi) : \xi \in \Xi({^\vee}\mathcal{O},{^\vee}G^{\Gamma})\}.$$

Suppose $\xi=(S,\tau)\in \Xi({^\vee}\mathcal{O},{^\vee}G^{\Gamma})$.
We define two invariants associated to $\xi$.
First, let $d(\xi)\nomenclature{$d(\xi)$}{dimension of the orbit attached to a complete geometric parameter}$ be the dimension of $S_\xi$.
Second, associated to $\xi$
is the representation $\pi(\xi)$ of a pure strong involution of
$G$ (\ref{localLanglandspure}). Let $e(\xi) = \pm1 \nomenclature{$e(\xi)$}{Kottwitz invariant attached to a complete geometric parameter}$ be the Kottwitz invariant of the underlying real
form of this strong involution (\cite{ABV}*{Definition 15.8}).

As stated in the introduction, we define a perfect pairing
\begin{equation}
\label{pairing}
\langle\,\cdot,\cdot\rangle:K\Pi({^\vee}\mathcal{O},G/\mathbb{R})\times KX({^\vee}\mathcal{O},{^\vee}G^{\Gamma}) \rightarrow \mathbb{Z}
\end{equation}
by
\begin{equation*}
\langle M(\xi), \mu(\xi') \rangle = e(\xi)\, \delta_{\xi, \xi'}.
\end{equation*}
The pairing  takes a deceptively simple form relative to the bases given by
$\pi(\xi)$ and $P(\xi')$ (\cite{ABV}*{Theorem 1.24, Sections
  15-17}).  We state it as a theorem.
\begin{thm}
  \label{ordpairing}
The pairing \eqref{pairing} satisfies
$$
\langle \pi(\xi), P(\xi') \rangle = (-1)^{d(\xi)} \, e(\xi) \,
\delta_{\xi, \xi'}, \quad \xi,\xi'\in \Xi({^\vee}\mathcal{O},{^\vee}G^{\Gamma}).
$$
\end{thm}
This pairing allows us to regard elements of $K\Pi({^\vee}\mathcal{O},G/\mathbb{R})$ as
$\mathbb{Z}$-linear functionals of $ KX({^\vee}\mathcal{O},{^\vee}G^{\Gamma})$.
The microlocal multiplicity maps
$\chi^{\mathrm{mic}}_S$  discussed in (\ref{mmm}) are $\mathbb{Z}$-linear
functionals on  $KX({^\vee}\mathcal{O},{^\vee}G^{\Gamma})$.  Before making the  connection with the
pairing \eqref{pairing}, we review some facts needed to define $\chi_{S}^{\mathrm{mic}}$.
To begin, we consider the category of
${^\vee} G$-equivariant coherent $\mathcal{D}$-modules on
$X({^\vee}\mathcal{O},{^\vee}G^{\Gamma})$. We
denote this category by
$\mathcal{D}(X({^\vee}\mathcal{O},{^\vee}G^{\Gamma}))$.
Here,
$\mathcal{D}$ is the sheaf of algebraic differential operators on
$X({^\vee}\mathcal{O},{^\vee}G^{\Gamma})$ 
(\cite{Boreletal}*{VIII.14.4}, \cite{ABV}*{Section 7},
\cite{Hotta}*{Part I}).
\nomenclature{$\mathcal{D} (X({^\vee}\mathcal{O},
  {^\vee}G^{\Gamma}))$}{category of $\mathcal{D}$-modules} 

The equivariant Riemann-Hilbert correspondence (\cite{Boreletal}*
{Theorem VIII.14.4}) induces an isomorphism
\begin{equation}
\label{dr}
DR : K\mathcal{D}(X({^\vee}\mathcal{O},{^\vee}G^{\Gamma}))\rightarrow K X({^\vee}\mathcal{O},{^\vee}G^{\Gamma}).
\nomenclature{$DR$}{Riemann-Hilbert correspondence}
\end{equation}
For simplicity we write $X=X({^\vee}\mathcal{O},{^\vee}G^{\Gamma})$, and $\mathcal{D} X=\mathcal{D}(X({^\vee}\mathcal{O},{^\vee}G^{\Gamma}))$.

The sheaf $\mathcal{D}$ is filtered by the order of the differential operators, and
the associated graded ring is canonically isomorphic to
$\mathcal{O}_{T^{*}(X)}$, the coordinate ring of the cotangent bundle
of $X$ (\cite{Hotta}*{Section 1.1}).
Suppose $\mathcal M\in\mathcal{D} X$. Then $\mathcal M$ has a filtration 
such that the resulting graded sheaf 
$\mathrm{gr} \mathcal{M} \nomenclature{$\mathrm{gr} \mathcal{M}$}{graded sheaf of a $\mathcal D$-module $\mathcal M$}$ is a
coherent $\mathcal{O}_{T^{*}(X)}$-module (\cite{Hotta}*{Section 2.1}).

The support of $\mathrm{gr} \mathcal{M}$ is a
closed subvariety of $T^{*}(X)$ (\cite{ABV}*{Definition 19.7}).  Each minimal ${^\vee}G$-invariant
component
of this closed subvariety is the closure of a conormal
bundle $T^{*}_{S}(X)$, where
\nomenclature{$T^{*}_{S}(X)$}{conormal bundle}
 $S \subset X$ is a ${^\vee}G$-orbit (\cite{ABV}*{Proposition 19.12(c)}).
Therefore to each conormal bundle $T^{*}_{S}(X)$ we may
attach a non-negative integer, denoted by
$\chi^{\mathrm{mic}}_{S}(\mathcal{M})$, which (when nonzero) is the length of
the module $\mathrm{gr}\mathcal{M}$ localized at $T^{*}_{S}(X)$
(\cite{Hotta}*{Section 2.2}). 

The
characteristic cycle of $\mathcal{M}$ is defined as
$$\mathrm{Ch}(\mathcal{M}) = \sum_{S\in X/{^\vee} G} \chi^{\mathrm{mic}}_{S}( \mathcal{M})
\ \overline{T_{S}^{*}(X)}.
\nomenclature{$\mathrm{Ch}(\mathcal{M})$}{characteristic cycle
of a $\mathcal D$-module $\mathcal M$}
$$
For a given ${^\vee} G$-orbit $S$  we may regard $\chi_{S}^{\mathrm{mic}}$ as a
function on $\mathcal{D}$-modules which is additive for
short exact sequences (\cite{ABV}*{Proposition 19.12(e)}).  It
 therefore defines a homomorphism
$K \mathcal{D}(X({^\vee}\mathcal{O}, {^\vee}G^{\Gamma}))
\rightarrow \mathbb{Z}$,
called the \emph{microlocal multiplicity} along $S$.
\nomenclature{$\chi^{\mathrm{mic}}_{S}$}{microlocal multiplicity map}
Using the isomorphism (\ref{dr}), we interpret this as a homomorphism
\begin{equation*} 
\chi^{\mathrm{mic}}_{S} : K X({^\vee}\mathcal{O}, {^\vee}G^{\Gamma})
\rightarrow \mathbb{Z}.
\end{equation*}

We now return to the pairing (\ref{pairing}) and its relationship to
$\chi^{\mathrm{mic}}_{S}$.  This relationship defines $\eta^{\mathrm{ABV}}_{\psi_G}$.
We first define $\eta^{\mathrm{mic}}_{\psi_G}\in K\Pi({^\vee}\mathcal{O},G/\mathbb{R})$ to be the element
of $K\Pi({^\vee}\mathcal{O},G/\mathbb{R})$ corresponding via the pairing to
the $\mathbb{Z}$-linear functional $\chi^{\mathrm{mic}}_S$ on
$KX({^\vee}\mathcal{O},{^\vee}G^{\Gamma})$. 
Explicitly working through the identifications in the definition we see
\begin{equation}
  \label{etapsi}
  \eta^{\mathrm{mic}}_{\psi_G}=
 \sum_{\xi \in \Xi({^\vee}\mathcal{O},{^\vee}G^{\Gamma})}
(-1)^{d(S_{\xi}) - d(S_{\psi_{G}})} \ \chi^{\mathrm{mic}}_{S_{\psi_{G}}}(P(\xi)) \, \pi(\xi).
\nomenclature{$\eta^{\mathrm{mic}}_{\psi_{G}}$}{stable virtual character}
\nomenclature{$d(S)$}{dimension of $S$}
\end{equation}

An important result of Kashiwara and Adams-Barbasch-Vogan is
\begin{prop}[\cite{ABV}*{Theorem 1.31, Corollary 19.16}]
The virtual character $\eta^{\mathrm{mic}}_{\psi_G}$ is stable.
\end{prop}

The \emph{microlocal packet} $\Pi^{\mathrm{mic}}_{\psi_{G}}$ of $\psi_{G}$  is
defined to be the irreducible representations 
in the support of $\eta^{\mathrm{mic}}_{\psi_{G}}$. In other words
$$
\Pi^{\mathrm{mic}}_{\psi_{G}}  = \{ \pi(\xi) : \xi \in \Xi({^\vee}\mathcal{O}, {^\vee}G^{\Gamma}), \ \chi^{\mathrm{mic}}_{S_{\psi_{G}}}(P(\xi)) \neq 0\}.
\nomenclature{$\Pi_{\psi_{G}}^{\mathrm{mic}}$}{microlocal packet}
$$
This is a set of irreducible representations of pure strong involutions of $G$.
We are primarily interested in the packet for the quasisplit strong
involutions.
We therefore define
\begin{equation}
\label{etapsiabv}
  \eta^{\mathrm{ABV}}_{\psi_{G}} = \eta_{\psi_{G}}^{\mathrm{mic}}(\delta_{q})
\nomenclature{$\eta^{\mathrm{ABV}}_{\psi_{G}}$}{
stable virtual character defining the ABV-packet}
\nomenclature{$\eta_{\psi_{G}}^{\mathrm{mic}}(\delta_{q})$}{}
\end{equation}
to be the restriction of $\eta^{\mathrm{mic}}_{\psi_{G}}$ to the submodule of $K
\Pi({^\vee}\mathcal{O}, G/\mathbb{R})$ generated by the representations in
$\Pi({^\vee}\mathcal{O}, G(\mathbb{R}, \delta_{q}))$.  The $\mathrm{ABV}$-packet
$\Pi_{\psi_{G}}^{\mathrm{ABV}}$ 
is defined as the support of $\eta^{\mathrm{ABV}}_{\psi_{G}}$, that is
\begin{equation}
\label{abvdef}
\Pi^{\mathrm{ABV}}_{\psi_{G}} = \{ \pi(\xi) : \xi \in \Xi({^\vee}\mathcal{O}, {^\vee}G^{\Gamma}), 
\chi^{\mathrm{mic}}_{S_{\psi_{G}}}(P(\xi)) \neq 0, \pi(\xi)\in \Pi(G(\mathbb{R},\delta_q))
\}.
\nomenclature{$\Pi_{\psi_{G}}^{\mathrm{ABV}}$}{ABV-packet}
\end{equation}

In definitions (\ref{etapsiabv}) and (\ref{abvdef}) we may easily replace $\delta_{q}$ with any other pure strong
involution $\delta$.  Although we shall only use these further objects in
Section \ref{puresec}, it seems appropriate to define them now.  Let
\begin{equation}
\label{etapsiabvdelta}
  \eta^{\mathrm{ABV}}_{\psi_{G}}(\delta) = \eta_{\psi_{G}}^{\mathrm{mic}}(\delta)
\nomenclature{$\eta^{\mathrm{ABV}}_{\psi_{G}}(\delta)$}{
stable virtual character defining an ABV-packet}
\end{equation}
be the restriction of $\eta^{\mathrm{mic}}_{\psi_{G}}$ to the submodule of $K
\Pi({^\vee}\mathcal{O}, G/\mathbb{R})$ generated by the representations in
$\Pi({^\vee}\mathcal{O}, G(\mathbb{R}, \delta))$.  In addition, let 
\begin{equation}
\label{abvdefdelta}
\Pi^{\mathrm{ABV}}_{\psi_{G}}(G(\mathbb{R},\delta)) = \{ \pi(\xi) : \xi \in \Xi({^\vee}\mathcal{O}, {^\vee}G^{\Gamma}), 
\chi^{\mathrm{mic}}_{S_{\psi_{G}}}(P(\xi)) \neq 0, \pi(\xi)\in \Pi(G(\mathbb{R},\delta))
\}.
\end{equation}

We conclude this section with a restatement of Theorem \ref{ordpairing} which will be valuable later on.
Define the representation-theoretic transition
matrix $m_r$ by
\begin{equation*}
M(\xi)  = \sum_{\xi' \in \Xi({^\vee}\mathcal{O}, {^\vee}G^{\Gamma})}
m_r(\xi',\xi) \, \pi(\xi').
\nomenclature{$m_r(\xi',\xi)$}{representation-theoretic multiplicity}
\end{equation*}
Define the geometric ``transition matrix'' $c_{g}$ by
\begin{equation*}
P(\xi) = \sum_{\xi' \in \Xi({^\vee}\mathcal{O},
  {^\vee}G^{\Gamma})} (-1)^{d(\xi)}\, c_{g}(\xi', \xi) \,
\mu(\xi').
\end{equation*}
(see \cite{ABV}*{(7.11)(c)}).
\nomenclature{$c_{g}(\xi', \xi)$}{signed sheaf-theoretic multiplicity}
Then \cite{ABV}*{Corollary 15.13} says
\begin{prop}
  \label{ordpairingequiv}
Theorem \ref{ordpairing} is
equivalent to the identity
\begin{equation*}
m_r(\xi', \xi) = (-1)^{d(\xi) - d(\xi')} \, c_g(\xi, \xi').
\end{equation*}  
\end{prop}
This equation relates the decomposition of characters with the
decomposition of sheaves.

\subsection{The pairing in the twisted case}
\label{conperv}

In the previous section the pairing \eqref{pairing}
plays a fundamental role in the
definition of ABV-packets. We now develop a twisted version of this
pairing  for $\mathrm{R}_{\mathbb{C}/\mathbb{R}} \mathrm{GL}_N$.

We replace $K\Pi({^\vee}\mathcal{O},\mathrm{GL}_N(\mathbb{C}))$ with the $\mathbb{Z}$-module
$K\Pi({^\vee}\mathcal{O},\mathrm{GL}_N(\mathbb{C}),\vartheta)$ of twisted characters 
(\ref{twistgroth1}).
Associated to
$\xi\in\Xi({^\vee}\mathcal{O}, {^\vee}\mathrm{R}_{\mathbb{C}/\mathbb{R}} \mathrm{GL}_{N}^{\Gamma})^{\vartheta}$ are an irreducible
representation $\pi(\xi) \in \Pi({^\vee}\mathcal{O} ,\mathrm{GL}_N(\mathbb{C}))^\vartheta$ as well as the Atlas extension 
$\pi(\xi)^+$ to $\mathrm{GL}_{N}(\mathbb{C}) \rtimes \langle \vartheta \rangle$ (Corollary \ref{twistclass2}). The twisted character of $\pi(\xi)^+$ is an
element of   
$K\Pi({^\vee}\mathcal{O},\mathrm{GL}_N(\mathbb{C}),\vartheta)$, the $\mathbb{Z}$-module of twisted characters,
and this gives a basis of $K\Pi({^\vee}\mathcal{O},\mathrm{GL}_N(\mathbb{C}),\vartheta)$ parameterized
by $\Xi({^\vee}\mathcal{O}, {^\vee}\mathrm{R}_{\mathbb{C }/\mathbb{R}} 
\mathrm{GL}_N^{\Gamma})^{\vartheta}$. 
See \eqref{KPiOGLext} and the end of Section \ref{extrepsec}.

The twisted characters are to be paired with twisted sheaves which are
elements in a $\mathbb{Z}$-module generalizing $KX({^\vee}\mathcal{O},{^\vee}G^{\Gamma})$.
The twisted objects for this pairing are given
in \cite{ABV}*{(25.7)} (see also \cite{Christie-Mezo}*{Section 5.4}).  We
provide a short summary.

  The automorphism $\vartheta$ acts on
$X({^\vee}\mathcal{O}, {^\vee}\mathrm{R}_{\mathbb{C}/\mathbb{R}} \mathrm{GL}_{N}^{\Gamma})$ in a manner which
is compatible with its ${^\vee}\mathrm{R}_{\mathbb{C}/\mathbb{R}}\mathrm{GL}_{N}$-action (\cite{ABV}*{(25.1)}),
and so also acts on $^{\vee}\mathrm{R}_{\mathbb{C}/\mathbb{R}} \mathrm{GL}_{N}$-equivariant sheaves.
\nomenclature{$\vartheta$}{automorphism of $\mathrm{R}_{\mathbb{C}/\mathbb{R}} \mathrm{GL}_{N}$}
Let $$\mathcal{P}(X({^\vee}\mathcal{O}, {^\vee}\mathrm{R}_{\mathbb{C}/\mathbb{R}} \mathrm{GL}_{N}^{\Gamma});
\vartheta)$$
\nomenclature{$\mathcal{P}(X({^\vee}\mathcal{O}, {^\vee}\mathrm{R}_{\mathbb{C}/\mathbb{R}} \mathrm{GL}_{N}^{\Gamma});
\vartheta)$}{category of perverse sheaves with a compatible $\vartheta$-action}
 be the category of ${^\vee}\mathrm{R}_{\mathbb{C}/\mathbb{R}} \mathrm{GL}_{N}$-equivariant perverse
sheaves with a compatible $\vartheta$-action.  An object in this category
is a pair $(P,\vartheta_{P})$ in which $P$ is an equivariant perverse sheaf and
$\vartheta_{P}$ is an automorphism of $P$ which is compatible with
$\vartheta$ (\cite{Christie-Mezo}*{Section 5.4}).    Similarly, we define
$$\mathcal{C}(X({^\vee}\mathcal{O}, {^\vee}\mathrm{R}_{\mathbb{C}/\mathbb{R}} \mathrm{GL}_{N}^{\Gamma}); 
\vartheta)$$ 
\nomenclature{$\mathcal{C}(X({^\vee}\mathcal{O}, {^\vee}\mathrm{R}_{\mathbb{C}/\mathbb{R}} \mathrm{GL}_{N}^{\Gamma}); 
\vartheta)$}{category of constructible sheaves  with a compatible $\vartheta$-action}
to be the category of ${^\vee}\mathrm{R}_{\mathbb{C}/\mathbb{R}} \mathrm{GL}_{N}$-equivariant
constructible sheaves with a compatible $\vartheta$-action.  An object in this category is a pair $(\mu,\vartheta_{\mu})$ in which $\mu$ is an equivariant
constructible sheaf and 
$\vartheta_{\mu}$ is an automorphism of $\mu$ which is compatible with
$\vartheta$.

\nomenclature{$K(X({^\vee}\mathcal{O},
  {^\vee}\mathrm{R}_{\mathbb{C}/\mathbb{R}} \mathrm{GL}_{N}^{\Gamma});\vartheta)$}{Grothendieck group} 
The Grothendieck groups of these two categories are isomorphic
(\cite{Christie-Mezo}*{(35)}).  We identify them and denote  their
Grothendieck groups by 
$K(X({^\vee}\mathcal{O}, {^\vee}\mathrm{R}_{\mathbb{C}/\mathbb{R}} \mathrm{GL}_{N}^{\Gamma});\vartheta)$.
This is the sheaf-theoretic analogue  
of $K \Pi(\mathrm{GL}_{N}(\mathbb{C}) \rtimes \langle \vartheta
\rangle)$.

As with the representations (see (\ref{canext})), we seek a canonical choice of
extension of $P(\xi)$, \emph{i.e.} a canonical automorphism 
$\vartheta_{P(\xi)}$ of $P(\xi)$.  This is achieved by the following lemma, whose proof follows exactly as for \cite{aam}*{Lemma 3.4} by virtue of (\ref{compgptriv}).
\begin{lem}
\label{cansheaf}
Let ${^\vee}G = {^\vee}\mathrm{R}_{\mathbb{C}/\mathbb{R}} \mathrm{GL}_{N}$,  $\xi = (S, \tau_{S}) \in
\Xi({^\vee}\mathcal{O}, 
{^\vee}\mathrm{R}_{\mathbb{C}/\mathbb{R}} \mathrm{GL}_{N}^{\Gamma})^{\vartheta}$, $p \in S$, and
(\ref{bundle}) be the equivariant vector bundle representing
$\mu(\xi)$.  
\begin{enumerate}[label={(\alph*)}] 
\item Suppose $p' \in S$ and $p'= a\cdot p$ for some $a \in {^\vee}
  \mathrm{R}_{\mathbb{C}/\mathbb{R}} \mathrm{GL}_{N}$. Then the maps 
\begin{align*}
(g,v) &\mapsto (ga^{-1},v)\\
\nonumber g\cdot p &\mapsto (ga^{-1}) \cdot p'
\end{align*}
define an isomorphism of equivariant vector bundles 
\begin{equation*}
{^\vee}G \times_{{^\vee}G_{p}} V \rightarrow {^\vee}G \times_{{^\vee}G_{p'}} V.
\end{equation*}
which is  independent of the choice of $a$.

\item There exist canonical choices of pairs
$$
  \mu(\xi)^{+} = (\mu(\xi), \vartheta_{\mu(\xi)}^{+}) \in
  \mathcal{C}(X({^\vee}\mathcal{O}, {^\vee}\mathrm{R}_{\mathbb{C}/\mathbb{R}} \mathrm{GL}_{N}^{\Gamma});
  \vartheta),$$ 
$$P(\xi)^{+} = (P(\xi), \vartheta_{P(\xi)}^{+}) \in
  \mathcal{P}(X({^\vee}\mathcal{O}, {^\vee}\mathrm{R}_{\mathbb{C}/\mathbb{R}} \mathrm{GL}_{N}^{\Gamma}); \vartheta)$$
 
such that  
if $p \in S$ is fixed by $\vartheta$ then  $\vartheta_{\mu(\xi)}^{+}$
(and $\vartheta_{P(\xi)}^{+}$) acts trivially on the stalk  of
$\mu(\xi)$ (and $P(\xi) \in KX({^\vee}\mathcal{O},
{^\vee}\mathrm{R}_{\mathbb{C}/\mathbb{R}}\mathrm{GL}_{N}^{\Gamma})$) at $p$.  
\end{enumerate}
\nomenclature{$\mu(\xi)^{+}$}{twisted constructible sheaf}
 \nomenclature{$P(\xi)^{+}$}{twisted perverse sheaf}
\end{lem}

We now imitate the definition of $K\Pi({^\vee}\mathcal{O},\mathrm{GL}_N(\mathbb{C}),\vartheta)$
\eqref{twistgroth1} for the sheaves appearing in Lemma \ref{cansheaf}. 
Attached to $\xi\in \Xi({^\vee}\mathcal{O},{^\vee}\mathrm{R}_{\mathbb{C}/\mathbb{R}} \mathrm{GL}_N^\Gamma)^{\vartheta}$ are
perverse sheaves $P(\xi)^{\pm}$, where $P(\xi)^+$ is defined in 
Lemma \ref{cansheaf}, and $P(\xi)^-$ is the unique other choice of extension.
Furthermore, the {\it microlocal traces} of $P(\xi)^\pm$  differ by
sign (\cite{ABV}*{(25.1)(j)}).  Similar comments apply to $\mu(\xi)^{\pm}$.  The microlocal traces are used in Proposition \ref{24.9c} and are the analogues of twisted characters.

We are interested only in irreducible sheaves with non-vanishing microlocal
trace.  We consequently follow the definition of
(\ref{twistgroth1}) in defining the quotient
\begin{equation}
  \label{twistsheafgroth}
  K X({^\vee}\mathcal{O},{^\vee}\mathrm{R}_{\mathbb{C}/\mathbb{R}} \mathrm{GL}_{N}^{\Gamma},\vartheta)=
  K(X({^\vee}\mathcal{O},{^\vee}\mathrm{R}_{\mathbb{C}/\mathbb{R}} \mathrm{GL}_{N}^{\Gamma}))^\vartheta\otimes 
    \mathbb{Z}[U_{2}]/ \langle (P(\xi) \otimes 1) + (P(\xi)\otimes -1)\rangle
\nomenclature{$K (X({^\vee}\mathcal{O}, {^\vee}\mathrm{R}_{\mathbb{C}/\mathbb{R}} \mathrm{GL}_{N}^{\Gamma},
    \vartheta)$}{$\mathbb{Z}$-module of twisted sheaves}   
  \end{equation}
  where the quotient runs over $\xi\in\Xi({^\vee}\mathcal{O},{^\vee}\mathrm{R}_{\mathbb{C}/\mathbb{R}} \mathrm{GL}_{N}^{\Gamma})^\vartheta$.

  This is the $\mathbb{Z}$-module which we shall pair with
  $$K\Pi({^\vee}\mathcal{O}, \mathrm{GL}_{N}(\mathbb{C}))^{\vartheta} \cong
  K\Pi({^\vee}\mathcal{O}, \mathrm{GL}_{N}(\mathbb{C}), \vartheta)$$
  in Section \ref{pairings}.  We call the elements of this module \emph{twisted
  sheaves}, and remind the reader that these modules are not naturally
Grothendieck groups, even though we have kept the ``$K$'' in the notation.
  
For reasons that will only become clear in Section \ref{whitsec}, the
definition of our twisted pairing involves some additional signs. The
signs depend on the integral lengths of parameters, which may be
described as follows.

From now on we assume  $\lambda \in {^\vee}\mathcal{O}$ satisfies the regularity condition (\ref{regintdom}).
Let $\xi\in\Xi({^\vee}\mathcal{O},{^\vee}\mathrm{R}_{\mathbb{C}/\mathbb{R}} \mathrm{GL}_N^\Gamma)^{\vartheta}$. 
 Lemma \ref{XXXi}  tells us that associated to $\xi$ is an element $x\in \mathcal{X}_{{^\vee}\rho}$. Set
 \begin{equation}
   \label{thetax}
   \theta_x=\mathrm{Int}(x)_{|H}.
\end{equation}
Let
\begin{subequations}
\renewcommand{\theequation}{\theparentequation)(\alph{equation}}  
\label{R}  
\begin{equation}
R(\lambda)=\left\{\alpha\in R(\mathrm{R}_{\mathbb{C}/\mathbb{R}} \mathrm{GL}_N,H): \langle\lambda,{^\vee}\alpha\rangle\in\mathbb{Z}\right\}
\nomenclature{$R(\lambda)$}{$\lambda$-integral root system}
\end{equation}
be the $\lambda$-integral roots, with positive $\lambda$-integral roots 
\begin{equation}
R^{+}(\lambda) = \left\{ \alpha \in R(\lambda): \langle\lambda,{^\vee}\alpha\rangle>0\right\}.
\nomenclature{$R^{+}(\lambda)$}{positive $\lambda$-integral roots}
\end{equation}
\end{subequations}
Define the
\emph{integral   length}, following \cite{ABV}*{(16.16)}, as
\begin{equation}
 \label{intlength}
l^{I}(\xi) = -\frac{1}{2} \left( | \{\alpha \in R^{+}(\lambda) :
\theta_x(\alpha) \in R^{+}(\lambda) \}| + \dim(H^{\theta_x}) \right) + \frac{N}{2}.
\nomenclature{$l^{I}(\xi)$}{integral length}
\end{equation}
The integral length takes values in the non-positive integers.

Furthermore define
$$
R_{\vartheta}^{+}(\lambda) = \{ \alpha \in
R((\mathrm{R}_{\mathbb{C}/\mathbb{R}} \mathrm{GL}_{N}^{\vartheta})^{0}, (H^{\vartheta})^{0}):
\langle\lambda,{^\vee}\alpha\rangle\in \mathbb{Z}_{>0}\}.
$$
We define the $\vartheta$\emph{-integral length} by
\begin{equation}
  \label{thetalength}
l^{I}_{\vartheta}(\xi) = -\frac{1}{2} \left( | \{\alpha \in
R^{+}_{\vartheta}(\lambda)  : 
\theta_x(\alpha) \in R^{+}_{\vartheta}(\lambda) \} |+
\dim((H^{\vartheta})^{\theta_x}) \right) + \frac{\lceil N/2 \rceil}{2}, 
\end{equation}
where $\lceil N/2 \rceil$ is the least integer greater than or equal to $N/2$.
This is the integral length for the fixed-point group $(\mathrm{R}_{\mathbb{C}/\mathbb{R}} \mathrm{GL}_N)^\vartheta \cong \mathrm{GL}_{N}$, or more precisely, for a quasisplit unitary group of Section \ref{extended}.
\nomenclature{$l^{I}_{\vartheta}(\xi) $}{$\vartheta$-integral length}

Now we define a perfect pairing (under the assumption \eqref{regintdom})
\begin{equation}
  \label{pair2}
\langle \cdot, \cdot \rangle:  K \Pi({^\vee}\mathcal{O},
\mathrm{GL}_{N}(\mathbb{C}), \vartheta) 
\times K X({^\vee}\mathcal{O}, {^\vee}\mathrm{R}_{\mathbb{C}/\mathbb{R}} \mathrm{GL}_{N}^{\Gamma}, \vartheta)
\rightarrow \mathbb{Z}
\end{equation}
by setting
\begin{equation}
\label{pairdef2}
\langle M(\xi)^{+}, \mu(\xi')^{+} \rangle = (-1)^{l^{I}(\xi) -
  l^{I}_{\vartheta}(\xi)} \, \delta_{\xi, \xi'}
\end{equation}
for $\xi, \xi' \in \Xi({^\vee}\mathcal{O},
{^\vee}\mathrm{R}_{\mathbb{C}/\mathbb{R}} \mathrm{GL}_{N}^{\Gamma})^{\vartheta}$.
 The analogue of Theorem \ref{ordpairing} is
\begin{thm}
  \label{twistpairing}
Suppose $\lambda \in {^\vee}\mathcal{O}$ satisfies (\ref{regintdom}).  Define the pairing (\ref{pair2}) by (\ref{pairdef2}).  Then
  $$\langle \pi(\xi)^{+}, P(\xi')^{+} \rangle = (-1)^{d(\xi)} \,
  (-1)^{l^{I}(\xi)-l^{I}_{\vartheta}(\xi)} \, \delta_{\xi, \xi'}$$
where $\xi, \xi' \in
\Xi({^\vee}\mathcal{O}, {^\vee}\mathrm{R}_{\mathbb{C}/\mathbb{R}} \mathrm{GL}_{N}^{\Gamma})^{\vartheta}$.
\end{thm}
The proof of this theorem is the primary purpose of Section
\ref{pairings}. Its proof is modelled on the proof of Theorem \ref{ordpairing} 
in \cite{ABV}*{Sections 15-17}.

We conclude this section by giving a twisted analogue of Proposition \ref{ordpairingequiv}.
This analogue will only be needed in Sections 7 and 9, so the reader may wish to skip this discussion and return to it 
later. 

For $\xi, \xi' \in \Xi({^\vee}\mathcal{O},
{^\vee}\mathrm{R}_{\mathbb{C}/\mathbb{R}} \mathrm{GL}_{N}^{\Gamma})^{\vartheta}$, define $m_{r}(\xi'_{\pm},
\xi_{+})$ to be the multiplicity of the 
representation $\pi(\xi')^{\pm}$ in $M(\xi)^{+}$ as elements of the Grothendieck group
$K\Pi({^\vee}\mathcal{O}, \mathrm{GL}_{N}(\mathbb{C}) \rtimes
\langle \vartheta \rangle)$ (Section \ref{grothchar}).  In other words
\begin{equation*}
M(\xi)^{+}  = \sum_{\xi' \in \Xi({^\vee}\mathcal{O}, {^\vee}G^{\Gamma})^{\vartheta}}
m_{r}(\xi'_{+},\xi_{+}) \, \pi(\xi')^{+}+ m_{r}(\xi'_{-},\xi_{+}) \,
\pi(\xi')^{-}+ \cdots
\end{equation*}
where the omitted summands are irreducible representations of
$\mathrm{GL}_{N}(\mathbb{C}) \rtimes \langle \vartheta \rangle$ which restrict to the sum of two irreducible representations of $\mathrm{GL}_{N}(\mathbb{C})$.  Define the \emph{twisted multiplicity of} $\pi(\xi')^{+}$ in $M(\xi)^{+}$ by
\begin{equation}
  \label{twistmult}
  m^{\vartheta}_{r}(\xi',\xi) = m_{r}(\xi'_{+}, \xi_{+})
  -m_{r}(\xi'_{-}, \xi_{+}), \quad \xi, \xi' \in \Xi({^\vee}\mathcal{O},
{^\vee}\mathrm{R}_{\mathbb{C}/\mathbb{R}} \mathrm{GL}_{N}^{\Gamma})^{\vartheta}
\end{equation}
 (\emph{cf.} 
\cite{AvLTV}*{(19.3d)}).  
\nomenclature{$m^{\vartheta}_{r}(\xi',\xi)$}{}
By construction, the image of $M(\xi)^{+}$ in $K
\Pi({^\vee}\mathcal{O}, \mathrm{GL}_{N}(\mathbb{C}), \vartheta)$
(\ref{twistgroth1}) decomposes as
\begin{equation}
  \label{twistmult1}
M(\xi)^{+} = \sum_{\xi' \in \Xi({^\vee}\mathcal{O},
  {^\vee}\mathrm{GL}_{N}^{\Gamma})^{\vartheta}} m^{\vartheta}_{r}(\xi',\xi) \, \pi(\xi')^{+}.
\end{equation}
The matrix given by (\ref{twistmult}) is invertible (\cite{aam}*{Lemma 3.6}).

In a parallel fashion, we define  $c_{g}(\xi'_{\pm},\xi_{+})$ for $\xi, \xi' \in
\Xi({^\vee}\mathcal{O}, {^\vee}G^{\Gamma})^{\vartheta}$ by
\begin{equation*}
P(\xi)^{+} = \sum_{\xi' \in \Xi({^\vee}\mathcal{O},
  {^\vee}G^{\Gamma})^{\vartheta}} (-1)^{d(\xi')}\, c_{g}(\xi'_{+},
\xi_{+}) \, \mu(\xi')^{+} + (-1)^{d(\xi')}\, c_{g}(\xi'_{-}, \xi_{+}) \,
\mu(\xi')^{-}+ \cdots
\end{equation*}
in the Grothendieck group $K X({^\vee}\mathcal{O}, {^\vee}\mathrm{R}_{\mathbb{C}/\mathbb{R}} \mathrm{GL}_{N};
\vartheta)$ of Section \ref{conperv}.  Setting
\begin{equation}
  \label{twistgmult}
c_{g}^{\vartheta}(\xi',\xi) = c_{g}(\xi'_{+},\xi_{+}) -
c_{g}(\xi'_{-}, \xi_{+}).
\nomenclature{$c_{g}^{\vartheta}(\xi',\xi)$}{}
\end{equation}
we see that the image of $P(\xi)^{+}$ in $K X({^\vee}\mathcal{O},
{^\vee}\mathrm{R}_{\mathbb{C}/\mathbb{R}}\mathrm{GL}_{N}, \vartheta)$ is
\begin{equation*}
\sum_{\xi' \in \Xi({^\vee}\mathcal{O},
  {^\vee}G^{\Gamma})^{\vartheta}} (-1)^{d(\xi')}\, c_{g}^{\vartheta}(\xi',
\xi) \, \mu(\xi')^{+}.
\end{equation*}
Just as Theorem \ref{ordpairing} is equivalent to Proposition \ref{ordpairingequiv}. We 
have the following equivalence. 
\begin{prop}[\cite{aam}*{Proposition 3.7}]
\label{p:twist}
Theorem \ref{twistpairing} is equivalent to the identity
\begin{equation*}
m_{r}^{\vartheta}(\xi', \xi) = (-1)^{l^{I}_{\vartheta}(\xi) -
  l^{I}_{\vartheta}(\xi')}\ c_{g}^{\vartheta}(\xi, \xi')
\end{equation*}
for all $\xi, \xi' \in \Xi({^\vee}\mathcal{O},
{^\vee}\mathrm{R}_{\mathbb{C}/\mathbb{R}} \mathrm{GL}_{N}^{\Gamma})^{\vartheta}$. 
\end{prop}

\section{The proof of Theorem \ref{twistpairing}}
\label{pairings}

We continue working with an infinitesimal character $\lambda \in {^\vee}\mathcal{O}$ which satisfies the hypothesis of (\ref{regintdom}).
The proof of Theorem \ref{twistpairing} is nearly identical to the proof of \cite{aam}*{Theorem 3.5}.  In fact, the proof of Theorem \ref{twistpairing} is somewhat simpler.  In order to convince the reader of these claims we review the proof of \cite{aam}*{Theorem 3.5} in our context, calling attention to the points which are simplified.  

The general idea of the proof is already given in the non-twisted context of \cite{ABV}*{Sections 15-17}.
In the proof one first extends the $\mathbb{Z}$-modules appearing in Theorem \ref{twistpairing} to \emph{Hecke modules} acted upon by a \emph{Hecke algebra}. The extended pairing is then shown to furnish an isomorphism between one of the Hecke modules and the dual of the other.   Special bases are chosen for the Hecke modules.  The values of the special bases are explicitly computed under the pairing.  Theorem \ref{twistpairing} then follows by restricting these values to the setting of the original $\mathbb{Z}$-modules.

There are numerous Hecke module computations underlying this proof, and many of them have been completed in \cite{LV2014} and \cite{AVParameters}.  The computations of \cite{AVParameters} are given in  representation-theoretic language and are therefore suitable when working with $K \Pi( {^\vee}\mathcal{O}, \mathrm{GL}_{N}(\mathbb{C}), \vartheta)$.  In order to adapt the computations of \cite{AVParameters} to $K X({^\vee}\mathcal{O}, {^\vee}\mathrm{R}_{\mathbb{C}/\mathbb{R}} \mathrm{GL}_{N}^{\Gamma}, \vartheta)$ we replace this module with an equivalent module of representations.  In so doing, we touch on the notion of \emph{Vogan duality} (\cite{AVParameters}*{Section 6.1}).  We attend to this preliminary work in the next section.

\subsection{Vogan duality and $K X({^\vee}\mathcal{O}, R_{\mathbb{C}/\mathbb{R}} \mathrm{GL}_{N}^{\Gamma}, \vartheta)$}
\label{duality}

We wish to replace the sheaf-theoretic module $K X({^\vee}\mathcal{O}, R_{\mathbb{C}/\mathbb{R}} \mathrm{GL}_{N}^{\Gamma}, \vartheta)$ with an equivalent  module of representations.  In the non-twisted setting this is achieved by \cite{ABV}*{Theorem 8.5}.  This theorem relies on two correspondences.  The first correspondence is the Riemann-Hilbert correspondence, which furnishes an equivalence between the category of equivariant perverse sheaves on $X({^\vee}\mathcal{O}, \mathrm{R}_{\mathbb{C}/\mathbb{R}} \mathrm{GL}_{N}^{\Gamma})$ and a category of equivariant D-modules on the same space (\cite{ABV}*{Theorem 7.9}, \cite{bbd}*{Theorem VIII.14.4}).  The second correspondence is the Beilinson-Bernstein localization theorem, which furnishes an equivalence between the category of equivariant D-modules and a category of Harish-Chandra modules (\cite{ABV}*{Theorem 8.3}, \cite{ICIII}*{Proposition 1.2}).

Combining the two correspondences produces a bijection between a set of irreducible perverse sheaves and a set of irreducible representations (Harish-Chandra modules).   The set of irreducible perverse sheaves is
$$\{P(\xi): \xi \in \Xi({^\vee}\mathcal{O}, {^\vee}\mathrm{R}_{\mathbb{C}/\mathbb{R}} \mathrm{GL}_{N}^{\Gamma})\}$$
(Section \ref{sec:ABV-packetsdef}).
We denote the set of corresponding representations by ${^\vee}\Pi({^\vee}\mathcal{O}, \mathrm{GL}_{N}(\mathbb{C}))$, so that the bijection may be written as
\begin{equation}
  \label{RHBB}
\{P(\xi): \xi \in \Xi({^\vee}\mathcal{O}, {^\vee}\mathrm{R}_{\mathbb{C}/\mathbb{R}} \mathrm{GL}_{N}^{\Gamma})\}
  \longleftrightarrow {^\vee}\Pi({^\vee}\mathcal{O}, \mathrm{GL}_{N}(\mathbb{C})).
\end{equation}

The notation for the set of representations on the right hints at some manner of duality with $\Pi({^\vee}\mathcal{O}, \mathrm{GL}_{N}(\mathbb{C}))$.  The particulars of this duality are given in \cite{aam}*{Section 4.2} for $\mathrm{GL}_{N}(\mathbb{R})$.  The arguments there apply just as well to $\mathrm{GL}_{N}(\mathbb{C})$ and are summarized as follows.  By Lemma \ref{XXXi}  a complete geometric parameter $\xi \in \Xi({^\vee}\mathcal{O}, {^\vee}\mathrm{R}_{\mathbb{C}/\mathbb{R}} \mathrm{GL}_{N}^{\Gamma})$ corresponds to a unique Atlas parameters $(x,y) \in \mathcal{X}_{{^\vee}\rho} \times
{^\vee}\mathcal{X}_{\lambda}$.  By reversing the order of the entries in the Atlas parameter, one obtains an Atlas parameter
$$(y,x)  \in {^\vee}\mathcal{X}_{\lambda} \times \mathcal{X}_{{^\vee}\rho}$$
for the group
\begin{equation}
  \label{vogandualgroup}
        {^\vee}\mathrm{R}_{\mathbb{C}/ \mathbb{R}} \mathrm{GL}_{N}(\lambda) = \mbox{  centralizer in  }{^\vee}\mathrm{R}_{\mathbb{C}/ \mathbb{R}} \mathrm{GL}_{N} \mbox{  of  } \exp(2 \uppi i \lambda)
\end{equation}
(\emph{cf.} \cite{aam}*{Lemma 4.2}).  The irreducible representation $J(y,x, {^\vee}\rho)$ (\emph{cf.} (\ref{abvij})) is a $(({^\vee}\mathfrak{gl}_{N} \times {^\vee}\mathfrak{gl}_{N})(\lambda), {^\vee}\tilde{K}_{y})$-module, where ${^\vee}\tilde{K}_{y}$ is a two-fold cover of ${^\vee}K_{y}$ (\emph{cf.} (\ref{kdef})).  It has infinitesimal character ${^\vee}\rho$.  This representation may be made plainer by computing that ${^\vee}\mathrm{R}_{\mathbb{C}/ \mathbb{R}} \mathrm{GL}_{N}(\lambda)$ is a product of groups $\prod_{i} {^\vee}\mathrm{R}_{\mathbb{C}/ \mathbb{R}} \mathrm{GL}_{n_{i}}$.  As  noted in Section \ref{extended}, this product has only one real form, namely $\prod_{i} \mathrm{GL}_{n_{i}}(\mathbb{C})$.  Consequently, $J(y,x, {^\vee}\rho)$ is a representation of a double-cover of $\prod_{i} \mathrm{GL}_{n_{i}}(\mathbb{C})$.

Bijection (\ref{RHBB}) is given by 
$$P(\xi) \mapsto J(y,x, {^\vee}\rho)$$
(\cite{aam}*{Proposition 4.3}).  We streamline the notation by the identification $\xi = (x,y)$ using Lemma \ref{XXXi}, and defining the dual parameter ${^\vee}\xi$ by
\begin{equation}
  \label{vogandual1}
        {^\vee}\xi = (y,x) \Longleftrightarrow \xi = (x,y).
\end{equation}
We denote the dual representation by
$$\pi({^\vee}\xi) = J(y,x, {^\vee}\rho).$$
In this way, bijection (\ref{RHBB}) takes the form
$$P(\xi) \mapsto \pi({^\vee}\xi), \quad \xi \in \Xi({^\vee}\mathcal{O}, {^\vee}\mathrm{R}_{\mathbb{C}/\mathbb{R}} \mathrm{GL}_{N}^{\Gamma}).$$
Every representation $\pi({^\vee}\xi)$ is the unique irreducible (Langlands) quotient of a standard module, which we denote by $M({^\vee}\xi)$ (\cite{AVParameters}*{(20)}).  Bijection (\ref{RHBB}) extends to an isomorphism of Grothendieck groups
\begin{equation}
    \label{rhbb}
K X\left({^\vee}\mathcal{O}, {^\vee}\mathrm{R}_{\mathbb{C}/\mathbb{R}} \mathrm{GL}_{N}^{\Gamma}\right) \cong K
{^\vee}\Pi\left({^\vee}\mathcal{O}, \mathrm{GL}_{N}(\mathbb{C})\right).
\nomenclature{$K{^\vee}\Pi\left({^\vee}\mathcal{O},\mathrm{GL}_{N}(\mathbb{C})\right)$}{dual Grothendieck group}  
\end{equation}
which satisfies
$$(-1)^{d(\xi)} \mu(\xi) \rightarrow M({^\vee}\xi), \quad \xi \in \Xi({^\vee}\mathcal{O}, {^\vee}\mathrm{R}_{\mathbb{C}/\mathbb{R}} \mathrm{GL}_{N}^{\Gamma})$$
(\cite{aam}*{Proposition 4.3}).  

The isomorphism (\ref{rhbb}) may be generalized to the twisted setting as well.
Recall from (\ref{canext}) that  the Atlas extension $\pi(\xi)^{+}$ is defined from a preferred extended parameter.   The same extended parameter also determines a unique extension $\pi({^\vee}\xi)^{+}$ of $\pi({^\vee}\xi)$ (\cite{AVParameters}*{(39h)}). It is a $(({^\vee}\mathfrak{gl}_{N} \times {^\vee}\mathfrak{gl}_{N})(\lambda), {^\vee}\tilde{K}_{y} \rtimes \langle \vartheta \rangle)$-module, which is the unique irreducible quotient of an extension $M({^\vee}\xi)^{+}$ of $M({^\vee}\xi)$.
Define a bijection
$$P(\xi)^+\mapsto \pi({^\vee}\xi)^+, \quad \xi \in
\Xi({^\vee}\mathcal{O},{^\vee}\mathrm{R}_{\mathbb{C}/\mathbb{R}} \mathrm{GL}_{N}^{\Gamma})^\vartheta.$$
 The bijection induces an isomorphism of $\mathbb{Z}$-modules
\begin{equation*}
K X({^\vee}\mathcal{O}, {^\vee}\mathrm{R}_{\mathbb{C}/\mathbb{R}} \mathrm{GL}_{N}^{\Gamma}, \vartheta) \cong K
       {^\vee}\Pi({^\vee}\mathcal{O}, 
       \mathrm{GL}_{N}(\mathbb{C}), \vartheta)
\end{equation*}
which satisfies
$$(-1)^{d(\xi)} \mu(\xi)^{+} \mapsto M({^\vee}\xi)^{+}, \quad \xi \in \Xi({^\vee}\mathcal{O},{^\vee}\mathrm{R}_{\mathbb{C}/\mathbb{R}} \mathrm{GL}_{N}^{\Gamma})^\vartheta$$
(\cite{aam}*{Proposition 4.5}).

The representation-theoretic replacement for $K X({^\vee}\mathcal{O}, {^\vee}\mathrm{R}_{\mathbb{C}/\mathbb{R}} \mathrm{GL}_{N}^{\Gamma}, \vartheta)$ is the $\mathbb{Z}$-module  $K{^\vee}\Pi({^\vee}\mathcal{O}, \mathrm{GL}_{N}(\mathbb{C}), \vartheta)$.  Making this replacement in Theorem \ref{twistpairing}, and taking into account the fixed relationship between $d(\xi)$ and $l^{I}(\xi)$ (\cite{AMR1}*{Proposition B.1}), we obtain
\begin{lem}[\cite{aam}*{Lemma 4.6}]
\label{twistpairing2}
Theorem \ref{twistpairing} is equivalent to the following assertion.
The pairing
\begin{equation}
  \label{pair3}
\langle \cdot , \cdot \rangle: K\Pi({^\vee}\mathcal{O},
  \mathrm{GL}_{N}(\mathbb{C}), \vartheta) \times K
         {^\vee}\Pi({^\vee}\mathcal{O}, \mathrm{GL}_{N}(\mathbb{C}),
\vartheta) \rightarrow \mathbb{Z}
\end{equation}
  defined by
$$\langle M(\xi)^{+}, M({^\vee}\xi')^{+} \rangle  =  (-1)^{
    l^{I}_{\vartheta}(\xi)} \, \delta_{\xi, \xi'}$$
satisfies
$$\langle \pi(\xi)^{+}, \pi({^\vee}\xi')^{+} \rangle = 
  (-1)^{l^{I}_{\vartheta}(\xi)} \, \delta_{\xi, \xi'}$$
where $\xi, \xi' \in
 \Xi({^\vee}\mathcal{O}, {^\vee}\mathrm{R}_{\mathbb{C}/\mathbb{R}} \mathrm{GL}_{N}^{\Gamma})^\vartheta$.
\end{lem}

\subsection{Twisted Hecke modules}
\label{twisthmodule}

The proof of Theorem \ref{twistpairing} relies on a Hecke algebra
and Hecke modules, which we introduce in the context of (\ref{pair3}).
In the twisted setting, Lusztig and Vogan
define a Hecke algebra which we denote by $\mathcal{H}(\lambda)$
(\cite{LV2014}*{Section 3.1}).  This Hecke algebra
acts on the Hecke modules
$$
\mathcal{K} \Pi({^\vee}\mathcal{O},
\mathrm{GL}_{N}(\mathbb{C}), \vartheta) = K \Pi({^\vee}\mathcal{O},
\mathrm{GL}_{N}(\mathbb{C}), \vartheta) \otimes_{\mathbb{Z}}
\mathbb{Z}[q^{1/2}, q^{-1/2}]
\nomenclature{$\mathcal{K} \Pi({^\vee}\mathcal{O},
\mathrm{GL}_{N}(\mathbb{C}), \vartheta)$}{twisted Hecke module}
$$
and
$$
\mathcal{K} {^\vee} \Pi({^\vee}\mathcal{O}, \mathrm{GL}_{N}(\mathbb{C}), \vartheta) =
K {^\vee} \Pi({^\vee}\mathcal{O}, \mathrm{GL}_{N}(\mathbb{C}), \vartheta)
\otimes_{\mathbb{Z}} \mathbb{Z}[q^{1/2}, q^{-1/2}]
\nomenclature{$\mathcal{K} {^\vee} \Pi({^\vee}\mathcal{O}, \mathrm{GL}_{N}(\mathbb{C}))$}{dual twisted Hecke module}
$$
as in \cite{LV2014}*{Section 7}.  We  extend the pairing
(\ref{pair3}) to these Hecke modules
\begin{equation}
  \label{pair4}
\langle \cdot  , \cdot \rangle:\mathcal{K}\Pi({^\vee}\mathcal{O},
\mathrm{GL}_{N}(\mathbb{C}),\vartheta) \times \mathcal{K} {^\vee}\Pi(  
{^\vee}\mathcal{O},{^\vee} \mathrm{GL}_{N}(\mathbb{C}), \vartheta) \rightarrow 
\mathbb{Z}[q^{1/2},q^{-1/2}],
\end{equation}
by setting
$$\langle M(\xi)^{+}, M({^\vee}\xi')^{+}\rangle  = (-1)^{l^{I}_{\vartheta}(\xi)}
\, q^{\left(l^{I}(\xi)+l^{I}({^\vee}\xi')\right)/2}\, \delta_{\xi, \xi'}$$
for all $\xi, \xi' \in \Xi({^\vee}\mathcal{O},
{^\vee}\mathrm{R}_{\mathbb{C}/\mathbb{R}} \mathrm{GL}_{N}^{\Gamma})^\vartheta$.  In view of the
Kronecker delta, the term
$q^{1/2\left(l^{I}(\xi)+l^{I}({^\vee}\xi')\right)}$ in the pairing
could be replaced by
$q^{1/2\left(l^{I}(\xi)+l^{I}({^\vee}\xi)\right)}$ or
$q^{1/2\left(l^{I}(\xi')+l^{I}({^\vee}\xi')\right)}$.  In fact, both
of the latter terms are independent of $\xi$ or $\xi'$ (\cite{aam}*{Lemma 4.7}).

To say more about the Hecke algebra $\mathcal{H}(\lambda)$, we must examine the set of integral roots $R(\lambda)$ (\ref{R}).
Let $\kappa$ be a $\vartheta$-orbit on the set of simple
\nomenclature{$\kappa$}{$\vartheta$-orbit of simple roots}
roots  of $R^{+}(\lambda)$.  
The orbit $\kappa$ is equal to one of the following:
\begin{align} 
\nonumber\text{one root }&\{\alpha=\vartheta(\alpha)\} &(\text{type }
1)~ & \\  
 \text{two roots
}&\{\alpha, \beta=\vartheta(\alpha)\},~\quad\left<\alpha, {^\vee}\beta\right>=0 
&(\text{type } 2)~ \label{simplekappa}\\  
\nonumber \text{two roots }&\{\alpha,\beta=\vartheta(\alpha)\},~\quad\left<
\alpha, {^\vee}\beta\right>=-1 & (\text{type } 3).
\end{align}
It is clear that our automorphism $\vartheta$ renders all orbits $\kappa$ to be of type 2.  This is a notable simplification in our setting.

Write $W(\lambda)$ for the Weyl group of the integral roots
$R(\lambda)$, and let 
$$W(\lambda)^{\vartheta}=\{w\in W(\lambda):\vartheta(w)=w \}.$$
\nomenclature{$W(\lambda)^{\vartheta}$}{$\vartheta$-fixed integral Weyl group}
The group $W(\lambda)^{\vartheta}$ 
is a Coxeter group (\cite{LV2014}*{Section 4.3}) with generators 
\begin{equation}
\label{simplekappa1}
w_{\kappa}=
s_{\alpha}s_{\vartheta( \alpha)}, \quad \kappa = \{\alpha, \vartheta(\alpha) \}.
\end{equation}
The Hecke algebra $\mathcal{H}(\lambda)$ (\cite{AVParameters}*{Section 10},
\cite{LV2014}*{Section 4.7})  
\nomenclature{$\mathcal{H}(\lambda)$}{twisted Hecke algebra}
is a free  $\mathbb{Z}[q^{1/2},q^{-1/2}]$-algebra with basis 
\begin{equation*}
  \{T_{w}:w\in W(\lambda)^{\vartheta} \}.
  \end{equation*}
It is a consequence of \cite{LV2014}*{Equation 4.7 (a)}
that $\mathcal{H}(\lambda)$ is generated by
the operators  $T_{\kappa}:=T_{w_{\kappa}}
\nomenclature{$T_{\kappa}$}{Hecke operator}
$, where  $\kappa$  is a
$\vartheta$-orbit as in (\ref{simplekappa}).

The action of $T_{\kappa}$  is defined in terms of the types listed in (\ref{simplekappa}), which for us are only of type 2.   The action also depends on the relationship of $\kappa$ relative to a fixed parameter $\xi \in \Xi({^\vee}\mathcal{O},
{^\vee}\mathrm{R}_{\mathbb{C}/\mathbb{R}} \mathrm{GL}_{N}^{\Gamma})^\vartheta$. To say a bit more about this dependence, recall that the parameter $\xi$ 
is equivalent to an Atlas parameter $(x,y)$ as in Lemma \ref{XXXi}.  The
adjoint action of $x$ acts as an involution on $R(\lambda)$ (see (\ref{e:p})).
This action separates the
$\vartheta$-orbits of roots $\kappa$ into various types, \emph{e.g.} real,
imaginary, etc.  Lusztig and Vogan
combine this information with the types of (\ref{simplekappa}) and
also with the types defined by Vogan in \cite{greenbook}*{Section 8.3}.  The list of
combined types may be found in \cite{LV2014}*{Section 7} or
\cite{AVParameters}*{Table 1}.

Very few of  the types that appear in these lists
are relevant for 
$\mathrm{GL}_{N}(\mathbb{C})$.  We have already observed that only type 2 orbits appear in the sense of (\ref{simplekappa}).  Furthermore, $\mathrm{GL}_{N}(\mathbb{C})$ has only complex roots relative to $x$ (\ref{del0action}). 
The only relevant types for
$\mathrm{GL}_{N}(\mathbb{C})$ in \cite{AVParameters}*{Table 1} are labelled as
\begin{equation}
  \label{glntypes}
  \mathtt{2C+, 2C-,2Ci, 2Cr}.\\
\end{equation}
That only these four types are relevant to our setting is another notable simplification.  
Any $\vartheta$-orbit $\kappa$ also has a type relative to the dual
parameter ${^\vee}\xi$ (\ref{vogandual1}).  The dual parameter is
equivalent to the Atlas parameter $(y,x)$ and the adjoint action of
$y$ is essentially the negative of the adjoint action of $x$
(\cite{AVParameters}*{Definition 3.10}). In
consequence, it is easy to compute the types and see that we again recover exactly those listed in  (\ref{glntypes}).

In \cite{LV2014}*{Section 4} and \cite{AVParameters}*{Section7} the Hecke algebra action on 
$\mathcal{K}\Pi({^\vee}\mathcal{O}, \mathrm{GL}_{N}(\mathbb{C}),\vartheta)$
is given by defining the action of the operators $T_{\kappa}$ 
on the generating set $\{M(\xi)^{+}:\xi\in
\Xi({^\vee}\mathcal{O},{^\vee}\mathrm{R}_{\mathbb{C}/ \mathbb{R}} \mathrm{GL}_{N}^{\Gamma})^\vartheta \}$.  The 
actions are presented in terms of extended Atlas parameters in \cite{AVParameters}*{Proposition 10.4}.  A case-by-case summary of the
actions is given in \cite{AVParameters}*{Table 5}.

There are obvious parallel constructions for (\ref{vogandualgroup}) which define a Hecke algebra
${^\vee}\mathcal{H}(\lambda)$ and a Hecke module structure for
$\mathcal{K} {^\vee}\Pi({^\vee}\mathcal{O},
\mathrm{GL}_{N}(\mathbb{C}), \vartheta)$. 
The Hecke algebra 
${^\vee}\mathcal{H}(\lambda)$ for (\ref{vogandualgroup}) is
generated by Hecke operators 
$T_{{^\vee}\kappa}$,  where ${^\vee}\kappa$ runs over the simple coroots
corresponding to $\kappa$.
\nomenclature{$T_{{^\vee}\kappa}$}{Hecke operator}
The bijection between the two sets of operators
$$\{T_{\kappa}: \kappa \in R^{+}(\lambda) \mbox{ simple}\}
\longleftrightarrow \{T_{{^\vee}\kappa} : \kappa \in R^{+}(\lambda) \mbox{
  simple}\}$$
extends to an isomorphism $\mathcal{H}(\lambda) \cong
{^\vee}\mathcal{H}(\lambda)$.
For this reason, we also regard
$\mathcal{K} {^\vee}\Pi({^\vee}\mathcal{O}, 
{^\vee}\mathrm{GL}_{N}(\mathbb{C}), \vartheta)$ as an
$\mathcal{H}(\lambda)$-module.

\subsection{A Hecke module isomorphism}
\label{heckesection}

Let
$$\mathcal{K} {^\vee}\Pi({^\vee}\mathcal{O}, \mathrm{GL}_{N}(\mathbb{C}),\vartheta)^{*} = \mathrm{Hom}_{\mathbb{Z}[q^{1/2} q^{-1/2}]} \left( \mathcal{K} {^\vee}\Pi({^\vee}\mathcal{O}, \mathrm{GL}_{N}(\mathbb{C}),\vartheta),\ \mathbb{Z}[q^{1/2} q^{-1/2}] \right).$$
The extended pairing  (\ref{pair4}) induces a $\mathbb{Z}[q^{1/2}, q^{-1/2}]$-module isomorphism
\begin{align}
\label{eq:dualformalmap}
\mathcal{K} \Pi({^\vee}\mathcal{O}, \mathrm{GL}_{N}(\mathbb{C}),
\vartheta) &\rightarrow   
\mathcal{K} {^\vee}\Pi({^\vee}\mathcal{O}, \mathrm{GL}_{N}(\mathbb{C}),\vartheta)^{*}\\
M(\xi)^{+} &\mapsto \langle M(\xi)^{+}, \cdot \,  \rangle. \nonumber
\end{align}
We endow $\mathcal{K}
{^\vee}\Pi({^\vee}\mathcal{O}, \mathrm{GL}_{N}(\mathbb{C}),\vartheta)^{*}$
with the Hecke module structure given in \cite{AVParameters}*{Section 11} (\emph{cf.} \cite{aam}*{Section 4.5}).  Specifically, for any $\mu \in \mathcal{K}
{^\vee}\Pi({^\vee}\mathcal{O}, \mathrm{GL}_{N}(\mathbb{C}),\vartheta)^{*}$ and $\vartheta$-orbit $\kappa$ as in (\ref{simplekappa})
\begin{equation}
\label{17.15e}
T^{*}_{w_{\kappa}}\cdot \mu = -(T_{w_{\kappa}})^{t}\cdot
\mu+(q^{l(w_{\kappa})}-1)\mu,
\end{equation}
where $l(w)$ is the usual length of $w$ with
respect to the simple reflections, and
\nomenclature{$l(w)$}{length in Weyl group}
$(T_{w})^{t}$ is the transpose of $T_{w}$.   Since both the domain and codomain of (\ref{eq:dualformalmap}) are $\mathcal{H}(\lambda)$-modules it is natural to ask whether  (\ref{eq:dualformalmap}) extends to a $\mathcal{H}(\lambda)$-module isomorphism.
\begin{prop}
  \label{Hisomorphism}
The map (\ref{eq:dualformalmap}) is an isomorphism of
$\mathcal{H}(\lambda)$-modules. 
\end{prop}
The proof of this proposition is a simplified version of the proof of \cite{aam}*{Proposition 4.8}. We provide a sketch.
In view of (\ref{17.15e}), Proposition \ref{Hisomorphism} is equivalent to
\begin{equation}
\label{17.17}
\langle T_{w_{\kappa}}M(\xi_{1}) ^{+},  M({^\vee} \xi_{2})^{+} \rangle
= \langle M(\xi_{1})^{+},  -T_{w_{\kappa}} M({^\vee}\xi_{2})^{+}
+(q^{l(w_{\kappa})}-1) M({^\vee}\xi_{2}) ^{+}\rangle 
\end{equation}
for all $\xi_{1}, \xi_{2} \in \Xi({^\vee}\mathcal{O}, 
{^\vee}\mathrm{R}_{\mathbb{C}/\mathbb{R}} \mathrm{GL}_{N}^{\Gamma})^\vartheta$ and $w_{\kappa}$ as
in (\ref{simplekappa1}).  Looking back to the definition of
(\ref{pair4}), the left-hand side of (\ref{17.17})  may be expressed as
\begin{align*}
\langle T_{w_{\kappa}}M(\xi_{1})^{+},  M({^\vee} \xi_{2})^{+} \rangle
 =
(-1)^{l^{I}_{\vartheta}(\xi_{2})} q^{(l^{I}(\xi_{2}) +
    l^{I}({^\vee}\xi_{2}))/2} \cdot ( \mbox{the coefficient of }
  M(\xi_{2})^{+}  \mbox{ in } T_{w_{\kappa}} M(\xi_{1})^{+}).  
\end{align*}
Similarly, the right-hand side of (\ref{17.17}) may be expressed as 
the product of $(-1)^{l^{I}_{\vartheta}(\xi_{1})} q^{(l^{I}(\xi_{1}) +
  l^{I}({^\vee}\xi_{1}))/2}$ with
$$ \mbox{the coefficient of }
M({^\vee}\xi_{1})^{+}  \mbox{ in }   -T_{w_{\kappa}}
M({^\vee}\xi_{2})^{+} +(q^{l(w_{\kappa})}-1) M({^\vee}\xi_{2})^{+}. $$ 
As in \cite{aam}*{Lemma 4.7}, 
$$l^{I}(\xi_{1}) +  l^{I}({^\vee}\xi_{1}) = l^{I}(\xi_{2}) + l^{I}({^\vee}\xi_{2}),$$
so Equation (\ref{17.17}) is equivalent to 
\begin{align}
\label{17.17a}
 &(-1)^{l^{I}_{\vartheta}(\xi_{2}) - l^{I}_{\vartheta}(\xi_{1})} \cdot
( \mbox{the coefficient of } M(\xi_{2})^{+}  \mbox{ in }
T_{w_{\kappa}} M(\xi_{1})^{+})\\ 
 \nonumber
 & = \mbox{ the coefficient of } M({^\vee}\xi_{1})^{+}  \mbox{ in }
 -T_{w_{\kappa}} M({^\vee}\xi_{2})^{+} +(q^{l(w_{\kappa})}-1)
 M({^\vee}\xi_{2})^{+}. 
\end{align}
The values of $T_{w_{\kappa}} M(\xi_{1})^{+}$ and $T_{w_{\kappa}} M({^\vee}\xi_{2})^{+}$ are known (\cite{AVParameters}*{Table 5}, \cite{aam}*{Proposition 4.9}).  
The proof of (\ref{17.17a}) may  therefore be achieved by a case-by-case analysis of the types of $\kappa$ relative to 
$\xi_{1}, \xi_{2}, {^\vee}\xi_{1}$ and ${^\vee}\xi_{2}$  (\ref{glntypes}).

We provide some more detail in the case that $\kappa = \{\alpha, \beta \}$ is of type  \texttt{2Ci} relative to $\xi_{1}$, leaving the remaining cases to the reader.   We identify $\xi_{1}$ with its corresponding Atlas parameter $(x_{1},y_{1})$ (Lemma \ref{XXXi}) and do likewise for all other parameters to come.  Type  \texttt{2Ci} implies $\beta = x_{1} \cdot \alpha$.

Let us start by computing the left-hand side of (\ref{17.17a}). According to \cite{AVParameters}*{Proposition 9.1 and Proposition 10.4}, we have 
\begin{align*}
T_{w_\kappa}M(\xi_1)^+\ &=\ qM(\xi_1)^+ + (q+1)M(\xi)^+,
\end{align*}
where $\xi =(x,y) \in \Xi({^\vee}\mathcal{O}, {^\vee}\mathrm{R}_{\mathbb{C}/\mathbb{R}} \mathrm{GL}_{N}^{\Gamma})$ is a parameter
satisfying
\begin{equation}
  \label{crossxx1}
  x =   s_{\alpha} x_1 s_{\alpha}^{-1} =\ s_{\alpha}s_{\beta} x_1 = x_{1} s_{\alpha} s_{\beta}
  \end{equation}
(\cite{AVParameters}*{(46o)}).
(We note that the signs appearing in \cite{AVParameters}*{Proposition 9.1} are always one for the Atlas extensions.) Consequently,
\begin{align*}
\text{the coefficient of }\ M(\xi_{2})^{+}\  \text{ in }\ 
T_{w_{\kappa}} M(\xi_{1})^{+}\ \text{ is }\ \left\{
\begin{array}{rl}
q,&  \text{if }\xi_2=\xi_1\\
(q+1),& \text{if }\xi_2=\xi\\
0,& \text{otherwise.}
\end{array}
\right.
\end{align*}
To compute the $\vartheta$-integral lengths
(\ref{thetalength}), observe that 
$$
x\cdot \{\alpha,\beta\}\ =\ s_{\alpha}s_{\beta}x_{1}\cdot \{\alpha,\beta\}=\ \{-\alpha,-\beta\},$$
and for any positive root  $\gamma\in R^{+}(\lambda)$ with $\gamma\neq \alpha,\beta$, we have
$$x\cdot \{\gamma,\vartheta\gamma\}\ =\ x_{1} s_{\alpha}s_{\beta}\{\gamma,\vartheta\gamma\}
\ =\ x_1\cdot \{\gamma',\vartheta\gamma'\},$$
where $\gamma'$ is a positive root in $R^{+}(\lambda)$. Hence, after identifying the roots in $R^{+}_{\vartheta}(\lambda)$ 
with $\vartheta$-orbits in $R^{+}(\lambda)$, we deduce
\begin{align*}
\left| \{\gamma \in R^{+}_{\vartheta}(\lambda)  :
x\cdot \gamma \in R^{+}_{\vartheta}(\lambda) \} \right|\ =\   \left| \{\gamma \in
R^{+}_{\vartheta}(\lambda)  :
x_1\cdot \gamma \in R^{+}_{\vartheta}(\lambda) \} \right|-1,
\end{align*}
and 
\begin{align*}
\dim\left((H^{\vartheta})^{x}\right)=\dim\left((H^{\vartheta})^{x_1}\right)-1.
\end{align*}
It follows that $l_{\vartheta}^{I}(\xi)=l_{\vartheta}^{I}(\xi_1)-1$ and $(-1)^{l_{\vartheta}^{I}(\xi) - l_{\vartheta}^{I}(\xi_1)} =
-1$. 
The left-hand side of (\ref{17.17a}) is therefore equal to 
\begin{align}
\label{leftside}
  \left\{
\begin{array}{rl}
q,&  \text{if }\xi_2=\xi_1\\
-(q+1),& \text{if }\xi_2= \xi\\
0,& \text{otherwise.}
\end{array}
\right.
\end{align}
Let us consider the right-hand side of (\ref{17.17a}), in which
${^\vee} \kappa$ is of type \texttt{2Cr} relative to ${^\vee}
\xi_{1}$.  According to \cite{AVParameters}*{Table 5},
$M({^\vee}\xi_{1})^{+}$ occurs in $T_{w_{\kappa}}
M({^\vee}\xi_{2})^{+}$ only if one of the following holds 
\begin{enumerate}
\item $M({^\vee}\xi_{1})^{+} = M({^\vee}\xi_{2})^{+}$,
\item $M({^\vee}\xi_{1})^{+}= s_{{}^{\vee}\alpha}\times M({}^{\vee}\xi_2)^{+}$,
\item $M({^\vee}\xi_{1})^{+}=w_{{^\vee}\kappa} \times M({^\vee}\xi_{2})^{+}$.
\end{enumerate}
The  third equation holds if and only if $M({^\vee}\xi_{2})^{+} =
w_{{^\vee}\kappa} \times M({^\vee}\xi_{1})^{+}$, and for ${^\vee}
\kappa$ of type \texttt{2Cr} relative to ${^\vee} \xi_{1}$ one obtains
$$w_{{^\vee}\kappa} \times M({^\vee}\xi_{1})^{+}\ =\ 
s_{{}^{\vee}\alpha}\times \left(s_{{}^{\vee}\beta} \times M({^\vee}\xi_{1})^{+}\right)\ =\ M({^\vee}\xi_{1})^{+}.$$
Therefore the third equation is equivalent to the first one.

The second equation is equivalent to
${^\vee}\xi_{2}$ being equal to ${^\vee} \xi$ ((\ref{crossxx1}),
\cite{aam}*{Proposition 4.9}).  Moreover, ${^\vee}\kappa$ is of type \texttt{2Ci} relative to ${^\vee}\xi$.  The results
\cite{AVParameters}*{Proposition 9.1 and Proposition 10.4} 
indicate  that
\begin{align*}
T_{w_{{}^{\vee}\kappa}}M\left({}^{\vee}\xi_1\right)^+\ &=\ (q^2 - q - 1)M({}^{\vee}\xi_1)^+ + (q^2 - q)M({}^{\vee}\xi)^+,\\
T_{w_{{}^{\vee}\kappa}}M\left({}^{\vee}\xi\right)^+\ &=\  q M({}^{\vee}\xi)^+ + (q+1)M({}^{\vee}\xi_1)^+.
\end{align*}
Therefore, the right-hand side of (\ref{17.17a}) equals 
\begin{align*}
\left\{
\begin{array}{rl}
q,&  \text{if }{}^{\vee}\xi_2={}^{\vee}\xi_1\\
-(q+1),& \text{if }{}^{\vee}\xi_2={}^{\vee}\xi\\
0,& \text{otherwise.}
\end{array}
\right.
\end{align*}
In all cases we have equality with (\ref{leftside}).

\subsection{Special bases and the proof of Theorem \ref{twistpairing}}

In addition to the $\mathcal{H}(\lambda)$-action on $\mathcal{K} \Pi({^\vee}\mathcal{O}, \mathrm{GL}_{N}(\mathbb{C}),
\vartheta)$, there is a $\mathbb{Z}$-linear involution $D$ on $\mathcal{K} \Pi({^\vee}\mathcal{O}, \mathrm{GL}_{N}(\mathbb{C}),
\vartheta)$ satisfying \nomenclature{$D$}{Verdier duality map}
\begin{equation}
  \label{verdual}
  \begin{aligned}
  D(q^{1/2} M(\xi)^{+}) &= q^{-1/2} D(
  (\xi)^{+})\\
  D((T_{\kappa} + 1) M(\xi)^{+} ) &= q^{-l(w_{\kappa})}(T_{\kappa} +
  1) \, D(M(\xi)^{+})
    \end{aligned}
\end{equation}
for all $\xi \in \Xi({^\vee}\mathcal{O}, 
{^\vee}\mathrm{R}_{\mathbb{C}/\mathbb{R}} \mathrm{GL}_{N}^{\Gamma})^\vartheta$.
This is the \emph{Verdier duality map} (\cite{LV2014}*{4.8 (e)-(f)}).  The Verdier duality map is uniquely determined by the additional conditions given in \cite{LV2014}*{8.1 (a)}.

The special basis we seek for $\mathcal{K} \Pi({^\vee}\mathcal{O}, \mathrm{GL}_{N}(\mathbb{C}), \vartheta)$ is a basis of eigenvectors for $D$. It is defined in terms of the twisted KLV-polynomials
$P^{\vartheta}({\xi'},{\xi})\in \mathbb{Z}[q^{1/2}, q^{-1/2}]$  defined
in \cite{LV2014}*{Section 0.1}.  Its definition and characteristics are summarized in following theorem.
\nomenclature{$P^{\vartheta}({\xi'},{\xi})$}{twisted KLV-polynomials}
\begin{thm}[\cite{LV2014}*{Theorem 5.2}]
\label{theo:defpoly}
For every $\xi \in \Xi({^\vee}\mathcal{O},
{^\vee}\mathrm{R}_{\mathbb{C}/\mathbb{R}} \mathrm{GL}_{N}^{\Gamma})^\vartheta$, define
\begin{equation}
  \label{basis}
C^{\vartheta}(\xi)=\sum_{\xi'\in \Xi({^\vee}\mathcal{O},
{^\vee}\mathrm{R}_{\mathbb{C}/\mathbb{R}} \mathrm{GL}_{N}^{\Gamma})^\vartheta}
(-1)^{l^{I}(\xi)-l^{I}(\xi')} \,P^{\vartheta}(\xi',\xi) 
\ M(\xi')^{+},
\nomenclature{$C^{\vartheta}(\xi)$}{}
\end{equation}
an element in $\mathcal{K}
\Pi({^\vee}\mathcal{O}, \mathrm{GL}_{N}(\mathbb{C}), \vartheta)$.
Then
\begin{enumerate}
\item ${D}({C}^{\vartheta}({\xi}))=q^{-l^{I}(\xi)}\, {C}^{\vartheta}({\xi})$
\item $P^{\vartheta}(\xi,\xi)=1$
\item $P^{\vartheta}(\xi', \xi) = 0$ if $\xi' \nleq {\xi}$
\item $\deg P^{\vartheta}(\xi',{\xi}) \leq
  (l^{I}({\xi})-l^{I}({\xi'})-1)/2$ if ${\xi'}\leq{\xi}$. 
\end{enumerate}
Conversely suppose  $\{\underline C(\xi',\xi)\}$ and 
$\{\underline{P}(\xi', \xi)\}$
satisfy \eqref{basis} and (1)-(4) above. 
Then $\underline P(\xi',\xi)=P^\vartheta(\xi',\xi)$ and $\underline
C(\xi',\xi)=C^\vartheta(\xi',\xi)$ for all $\xi',\xi \in
\Xi({^\vee}\mathcal{O}, {^\vee}\mathrm{R}_{\mathbb{C}/\mathbb{R}} \mathrm{GL}_{N}^{\Gamma})^{\vartheta}$. 
\end{thm}
The third property of this theorem uses a partial order on $\Xi({^\vee}\mathcal{O},
{^\vee}\mathrm{R}_{\mathbb{C}/\mathbb{R}} \mathrm{GL}_{N}^{\Gamma} )^\vartheta$, the \emph{Bruhat
  order}, which is defined in \cite{LV2014}*{Section 5.1} (\emph{cf.} \cite{ABV}*{(7.11)(f)}).

There is also a Verdier duality map ${^\vee}D$ for the module $\mathcal{K}
{^\vee}\Pi({^\vee}\mathcal{O}, \mathrm{GL}_{N}(\mathbb{C}),\vartheta)$ which satisfies obvious analogues of (\ref{verdual}) and the additional conditions of \cite{LV2014}*{8.1 (a)}.  Furthermore, there is an obvious analogue  Theorem \ref{theo:defpoly} for the dual basis elements
\begin{equation*}
C^{\vartheta}({^\vee}\xi) =\sum_{\xi' \in \Xi({^\vee}\mathcal{O},
{^\vee}\mathrm{R}_{\mathbb{C}/\mathbb{R}} \mathrm{GL}_{N}^{\Gamma})^\vartheta
  }(-1)^{l^{I}({^\vee}\xi)-l^{I}({^\vee}\xi')}
\ {^\vee}P^{\vartheta}(\, ^{\vee}\xi', {^\vee}\xi) \, 
M({^\vee}\xi')^{+}.
\nomenclature{$C^{\vartheta}({^\vee}\xi)$}{}
\end{equation*}
According to \cite{aam}*{Proposition 4.14}, by setting $q=1$ in the polynomials of the basis elements, one obtains
\begin{equation}
  \label{4.14}
C^{\vartheta}(\xi)(1)  = \pi(\xi)^{+}\  \mbox{   and   }\ C^{\vartheta}({^\vee}\xi)(1)  = \pi({^\vee}\xi)^{+}
\end{equation}
for all $\xi \in \Xi({^\vee}\mathcal{O},
{^\vee}\mathrm{R}_{\mathbb{C}/\mathbb{R}} \mathrm{GL}_{N}^{\Gamma})^\vartheta$.
It is immediate from Equation (\ref{4.14}) and Lemma \ref{twistpairing2} that Theorem \ref{twistpairing} is proved by the following theorem.
\begin{thm}
Pairing (\ref{pair4}) satisfies
\begin{align*}
\langle C^{\vartheta}(\xi), C^{\vartheta}({^\vee}\xi')\rangle =
(-1)^{l^{I}_{\vartheta}(\xi)} \, q^{\left(l^{I}(\xi)+l^{I}({^\vee}\xi')\right)/2} \,
\delta_{\xi,\xi'}.  
\end{align*}
for all $\xi, \xi' \in \Xi({^\vee}\mathcal{O},
{^\vee}\mathrm{R}_{\mathbb{C}/\mathbb{R}}  \mathrm{GL}_{N}^{\Gamma})^\vartheta$.
\end{thm}
\begin{proof}
The isomorphism (\ref{eq:dualformalmap}) allows us to define unique elements
$\underline{C}^{\vartheta}(\xi) \in\mathcal{K}\Pi({^\vee}\mathcal{O},
\mathrm{GL}_{N}(\mathbb{C}),\vartheta)$, $\xi \in \Xi({^\vee}\mathcal{O}, {^\vee}\mathrm{R}_{\mathbb{C}/\mathbb{R}} \mathrm{GL}_{N}^{\Gamma})^\vartheta$, satisfying 
\nomenclature{$\underline{C}^{\vartheta}(\xi)$}{}
\begin{equation*}
\langle \underline{C}^{\vartheta}(\xi), C^{\vartheta}({^\vee}\xi') \rangle =
(-1)^{l^{I}_{\vartheta}(\xi')}\, 
q^{\left(l^{I}(\xi)+l^{I}({^\vee}\xi')\right)/2}\, \delta_{\xi, \xi'}.
\end{equation*}
The proof amounts to showing that
$$\underline{C}^{\vartheta}(\xi) = C^{\vartheta}({^\vee}\xi), \quad \xi \in \Xi({^\vee}\mathcal{O},
{^\vee}\mathrm{R}_{\mathbb{C}/\mathbb{R}}  \mathrm{GL}_{N}^{\Gamma})^\vartheta$$
and we may do so by showing that $\underline{C}^{\vartheta}(\xi)$ satisfies each of the four properties in Theorem \ref{theo:defpoly}.  We outline the proof of the the first property of Theorem \ref{theo:defpoly}, as it is the most involved.  We  leave the reader to verify the remaining three properties by consulting \cite{aam}*{Section 4.7}.

The first step is to prove that the two Verdier dualities are related by 
\begin{equation}
  \label{17.23}
\overline{\langle D (M(\xi)), M({^\vee}\xi') \rangle} = \langle M(\xi), {^\vee}D (M({^\vee}\xi')) \rangle, \quad \xi, \xi' \in \Xi({^\vee}\mathcal{O},
         {^\vee}\mathrm{R}_{\mathbb{C}/\mathbb{R}}  \mathrm{GL}_{N}^{\Gamma})^\vartheta
\end{equation}
where
$$\bar{\vspace{2mm}} \,: \mathbb{Z} [q^{1/2}, q^{-1/2}] \rightarrow
\mathbb{Z} [q^{1/2}, q^{-1/2}] $$ 
is the unique automorphism sending $q^{1/2}$ to $q^{-1/2}$.  To do this, one may define ${^\vee}D'$ by
$$\langle M(\xi), {^\vee}D' (M({^\vee}\xi')) \rangle = \overline{\langle D (M(\xi)), M(^{\vee}\xi') \rangle}, \quad \xi, \xi' \in \Xi({^\vee}\mathcal{O},
{^\vee}\mathrm{R}_{\mathbb{C}/\mathbb{R}}  \mathrm{GL}_{N}^{\Gamma})^\vartheta$$
and prove that ${^\vee}D' = {^\vee}D$ using the properties of (\ref{verdual}) and \cite{LV2014}*{8.1 (a)} which characterize the Verdier duality (\emph{cf.} \cite{ICIV}*{Lemma 13.4}).  Proposition \ref{Hisomorphism} plays a key role in proving the property
\begin{equation}
  \label{verdual1}
{^\vee}D'((T_{\kappa} + 1) M({^\vee}\xi')^{+} ) = q^{-l(w_{\kappa})}(T_{\kappa} +
1) \, {^\vee}D'(M({^\vee}\xi')^{+})
\end{equation}
of (\ref{verdual}), so it seems fitting to supply the arguments for it.  We are to prove that
$$\langle M(\xi),  {^\vee}D'((T_{\kappa} + 1) M({^\vee}\xi')^{+} ) \rangle= \langle M(\xi), q^{-l(w_{\kappa})}(T_{\kappa} +
1) \, {^\vee}D'(M({^\vee}\xi')^{+}) \rangle $$
for all $\xi, \xi' \in \Xi({^\vee}\mathcal{O},
{^\vee}\mathrm{R}_{\mathbb{C}/\mathbb{R}}  \mathrm{GL}_{N}^{\Gamma})^\vartheta$.  According to the definition of ${^\vee}D'$,  and Equation (\ref{17.17}) of Proposition \ref{Hisomorphism}, this is equivalent to
$$(-T_{\kappa} + q^{l(w_{\kappa})}) D (M(\xi)) = D \left((-T_{\kappa} + q^{l(w_{\kappa})}) q^{-l(w_{\kappa})} M(\xi) \right).$$
Beginning with the right-hand side and following (\cite{ICIV}*{Lemma 13.4}) with (\ref{verdual}), we compute
\begin{align*}
 & -D\left(( T_{\kappa} + 1) q^{-l(w_{k})} M(\xi) \right) + D \left( (q^{l(w_{\kappa})} + 1) q^{-l(w_{\kappa})} M(\xi) \right)\\
  & =  -q^{-l(w_{\kappa})} (T_{\kappa} + 1) D(q^{-l(w_{\kappa})} M(\xi)) + (1 + q^{l(w_{\kappa})}) D(M(\xi))\\
  &=- (T_{\kappa} + 1) D( M(\xi)) + (1 + q^{l(w_{\kappa})}) D(M(\xi))\\
  &= (-T_{\kappa} + q^{l(w_{\kappa})}) D (M(\xi)).
\end{align*}
This proves (\ref{verdual1}).  The remaining properties ensuring that ${^\vee}D' = {^\vee}D$ are also easily read from the proof of \cite{ICIV}*{Lemma 13.4}.

We now take for granted Equation (\ref{17.23}),  and may prove that $\underline{C}^{\vartheta}(\xi)$ satisfies the first property of Theorem \ref{theo:defpoly} as follows.
\begin{align*}
\langle D \underline{C}^{\vartheta}(\xi), C^{\vartheta}({^\vee}\xi')
\rangle & = \overline{ \langle \underline{C}^{\vartheta}(\xi),
  \,{^\vee}D C^{\vartheta}({^\vee}\xi') \rangle}\\ 
& = q^{l^{I}({^\vee}\xi')} \, \overline{  \langle
  \underline{C}^{\vartheta}(\xi),  C^{\vartheta}({^\vee}\xi') \rangle}
\\ 
& = (-1)^{l^{I}_{\vartheta}(\xi)} q^{l^{I}(^{\vee}\xi')}
q^{-(l^{I}(\xi) + l^{I}({^\vee}\xi'))/2}  \delta_{\xi,\xi'} \\    
&= \langle q^{-l^{I}(\xi)}\, \underline{C}^{\vartheta}(\xi),
C^{\vartheta}({^\vee}\xi') \rangle. 
\end{align*}
Since the elements $C^{\vartheta}({^\vee}\xi')$ form a basis we conclude that 
$$ D \underline{C}^{\vartheta}(\xi) = q^{-l^{I}(\xi)}\, \underline{C}^{\vartheta}(\xi)$$
and the first property of Theorem \ref{theo:defpoly} is proved.  As mentioned, the remaining properties of  Theorem \ref{theo:defpoly} follow easily from \cite{aam}*{Section 4.7}.
\end{proof}

\section{Endoscopic lifting for complex general linear groups following
  Adams-Barbasch-Vogan} 
  \label{endosec}

We proceed with a review of standard endoscopy and twisted endoscopy
from the perspective of \cite{ABV}, but restricted only to the
particular case of 
the group $\mathrm{R}_{\mathbb{C}/\mathbb{R}}\mathrm{GL}_{N}$.  A similar review was made for $\mathrm{GL}_{N}$ in \cite{aam}*{Section 5}.  The present review contains nothing new and so we shall refer to \cite{aam}*{Section 5} liberally.  The background material for this section is found in \cite{ABV}*{Section 26} and   \cite{Christie-Mezo}*{Section 5}.

\subsection{Standard endoscopy}
\label{standend}

In this section we do not place any restrictions on the infinitesimal
characters.  In addition the general framework applies to any
connected real reductive group.  For the purposes of motivation we
specialize to the group 
${^\vee}\mathrm{R}_{\mathbb{C}/\mathbb{R}}\mathrm{GL}_{N}^{\Gamma} =
({^\vee}\mathrm{GL}_{N} \times {^\vee}\mathrm{GL}_{N}) \rtimes \langle {^\vee}\delta_{0} \rangle$
as in Section \ref{extended}.  In the last two sections, we shall
apply the theory of endoscopy to unitary groups.

An \emph{endoscopic datum} for ${^\vee}\mathrm{R}_{\mathbb{C}/\mathbb{R}}\mathrm{GL}_{N}^{\Gamma}$ is a pair
$$(s, \, {^\vee}G^{\Gamma})$$
which satisfies
\begin{enumerate}
\item $s \in {^\vee}\mathrm{R}_{\mathbb{C}/\mathbb{R}}\mathrm{GL}_{N}$ is semisimple

\item ${^\vee}G^{\Gamma} \subset {^\vee}\mathrm{R}_{\mathbb{C}/\mathbb{R}} \mathrm{GL}_{N}^{\Gamma}$ is
  open in the centralizer of $s$ in ${^\vee}\mathrm{R}_{\mathbb{C}/\mathbb{R}} \mathrm{GL}_{N}^{\Gamma}$ 

\item ${^\vee}G^{\Gamma}$ is an E-group for a group $G$ (\cite{ABV}*{Definition 4.6}).  
\end{enumerate}
\nomenclature{$s$}{semisimple element in endoscopic datum}
This is a specialization of \cite{ABV}*{Definition 26.15} to
${^\vee}\mathrm{R}_{\mathbb{C}/\mathbb{R}}\mathrm{GL}_{N}^{\Gamma}$.  The groups ${^\vee}G$ and $G$ here are
isomorphic Levi subgroups.  They are products of smaller general linear groups. Consequently,
${^\vee}G$ and ${^\vee}\mathrm{R}_{\mathbb{C}/\mathbb{R}}\mathrm{GL}_{N}$ share the maximal torus
${^\vee}H$, which is two copies of the diagonal subgroup.  Similarly, $G$ and $\mathrm{R}_{\mathbb{C}/\mathbb{R}}\mathrm{GL}_{N}$ share the maximal 
torus $H$.  We shall abusively denote by $\delta_{q}$ the
strong involution on both $G$ and $\mathrm{R}_{\mathbb{C}/\mathbb{R}}\mathrm{GL}_{N}$ which correspond to the quasisplit real forms.  The group $G$ in this definition is called the
\emph{endoscopic group}. 
We do not require the concept of an E-group in this section.
From now on we assume that ${^\vee}G^{\Gamma} =
{^\vee}G \rtimes \langle {^\vee}\delta_{0} \rangle$.   In
other words, ${^\vee}G^{\Gamma}$ is an L-group for $G$.

There is a notion of equivalence for endoscopic data, and using this
equivalence we may assume without loss of generality that $s \in
{^\vee}H$.  We fix $\lambda \in {^\vee}\mathfrak{h}$. Let ${^\vee}\mathcal{O}_{G}$ be the
${^\vee}G$-orbit of $\lambda$ and ${^\vee}\mathcal{O}$ be the
${^\vee}\mathrm{R}_{\mathbb{C}/\mathbb{R}}\mathrm{GL}_{N}$-orbit of $\lambda$.    
The second property of the endoscopic datum above allows us to define
the inclusion 
\begin{equation}
\epsilon : \, {^\vee}G^{\Gamma} \hookrightarrow \, {^\vee}\mathrm{R}_{\mathbb{C}/\mathbb{R}}\mathrm{GL}_{N}^{\Gamma}.
\label{epinclusion}
\end{equation}
This inclusion induces another map (\cite{ABV}*{Corollary 6.21}),  which we abusively also denote as
\begin{equation} 
\label{epinclusion1}
\epsilon: \, X\left({^\vee}\mathcal{O}_{G}, {^\vee}G^{\Gamma}\right) \rightarrow
X\left({^\vee}\mathcal{O}, {^\vee}\mathrm{R}_{\mathbb{C}/\mathbb{R}}\mathrm{GL}_{N}^{\Gamma}\right). 
\nomenclature{$\epsilon$}{endoscopic maps}
\end{equation}
It is easily verified that the ${^\vee}G$-action on
$X({^\vee}\mathcal{O}_{G}, {^\vee}G^{\Gamma})$ is compatibly carried under
$\epsilon$ to the ${^\vee}\mathrm{R}_{\mathbb{C}/\mathbb{R}}\mathrm{GL}_{N}$-action on $X({^\vee}\mathcal{O},
{^\vee}\mathrm{R}_{\mathbb{C}/\mathbb{R}}\mathrm{GL}_{N}^{\Gamma})$ (\cite{ABV}*{(7.17)}).  As a result,
the map $\epsilon$ induces a map from the orbits of the space
$X({^\vee}\mathcal{O}_{G}, {^\vee}G^{\Gamma})$ to the orbits of  $X({^\vee}\mathcal{O},
{^\vee}\mathrm{R}_{\mathbb{C}/\mathbb{R}}\mathrm{GL}_{N}^{\Gamma})$.

The inverse image functor of $\epsilon$ on equivariant constructible
sheaves induces a homomorphism 
$$
\epsilon^{*}: K({^\vee}\mathcal{O}, {^\vee}\mathrm{R}_{\mathbb{C}/\mathbb{R}} \mathrm{GL}_{N}^{\Gamma})
\rightarrow K({^\vee}\mathcal{O}_{G}, {^\vee}G^{\Gamma}) 
\nomenclature{$\epsilon^{*}$}{inverse image functor}
$$
(\cite{ABV}*{Proposition 7.18}).
When $\epsilon^{*}$ is combined with the pairings of Theorem
\ref{ordpairing}, we obtain a map 
$$\epsilon_{*} : K_{\mathbb{C}} \Pi({^\vee}\mathcal{O}_{G}, G/\mathbb{R})
\rightarrow K_{\mathbb{C}} \Pi ({^\vee}\mathcal{O},
\mathrm{GL}_{N}(\mathbb{C}))
\nomenclature{$K_{\mathbb{C}} \Pi({^\vee}\mathcal{O}_{G}, G/\mathbb{R})$}{}
\nomenclature{$\epsilon_{*}$}{}
$$ 
defined on $\eta_{G} \in K_{\mathbb{C}} \Pi({^\vee}\mathcal{O}_{G}, G/\mathbb{R})$ by
\begin{equation}
\label{loweps}
\left\langle \epsilon_{*} \eta_{G}, \mu(\xi) \right\rangle = \left\langle \eta_{G},
\epsilon^{*} \mu(\xi)\right\rangle_{G}, \quad \xi \in
\Xi({^\vee}\mathcal{O},{^\vee}\mathrm{R}_{\mathbb{C}/\mathbb{R}}\mathrm{GL}_{N}^{\Gamma}). 
\end{equation}
Here, $K_{\mathbb{C} }= \mathbb{C} \otimes_{\mathbb{Z}} K$ and we have
placed a subscript $G$ beside the pairing on the right to distinguish
it from the pairing for $\mathrm{R}_{\mathbb{C}/\mathbb{R}}\mathrm{GL}_{N}$ on the left.  

The endoscopic lifting map $\mathrm{Lift}_{0}$ is defined to be the restriction of $\epsilon_{*}$ to the 
submodule
$$K_{\mathbb{C}} \Pi({^\vee}\mathcal{O}_{G}, G(\mathbb{R},
\delta_{q}))^{\mathrm{st}} \subset
K_{\mathbb{C}} \Pi({^\vee}\mathcal{O}_{G}, G/\mathbb{R})$$
of stable virtual characters of the quasisplit form $G(\mathbb{R},
\delta_{q})$.
Since
$G(\mathbb{R}, \delta_{q})$ is a product of complex general linear groups, stability is not an issue
and  we have  
\begin{equation*}
K_{\mathbb{C}} \Pi({^\vee}\mathcal{O}_{G}, G( \mathbb{R}, \delta_{q}))^{\mathrm{st}} =
K_{\mathbb{C}} \Pi({^\vee}\mathcal{O}_{G}, G/\mathbb{R}). 
\end{equation*}
This equality will not hold for twisted endoscopic groups in Section
\ref{twistendsec}, and so it is better two write the endoscopic lifting map as
\begin{equation}
\label{endlift}
\mathrm{Lift}_{0}: K_{\mathbb{C}} \Pi({^\vee}\mathcal{O}_{G}, G(\mathbb{R},
\delta_{q}))^{\mathrm{st}} \rightarrow K_{\mathbb{C}} \Pi({^\vee}\mathcal{O},
\mathrm{GL}_{N}(\mathbb{C})). 
\nomenclature{$\mathrm{Lift}_{0}$}{endoscopic lifting map}
\end{equation}

According to \cite{ABV}*{Lemma 18.11} and 
\cite{MW}*{Corollary IV.2.8}), a 
basis for $K_{\mathbb{C}} \Pi({^\vee}\mathcal{O}_{G}, G(\mathbb{R},
\delta_{q}))^{\mathrm{st}}$ is provided by the virtual characters
\begin{equation}
\label{etaloc}
\eta^{\mathrm{loc}}_{S_{1}}(\delta_{q})  = \sum_{\tau_{S_{1}}} M(S_{1}, \tau_{S_{1}}),
\end{equation}
where $(S_{1}, \tau_{S_{1}}) \in
\Xi({^\vee}\mathcal{O}_{G}, {^\vee}G^{\Gamma})$ runs over those complete
geometric parameters which correspond to the strong involution
$\delta_{q}$  under the local Langlands correspondence (\ref{localLanglandspure}).
As indicated at the end of Section \ref{extended}, the relevant component groups for complex general linear groups are trivial.  Therefore the representations $\tau_{S_{1}}$ are all trivial for $G$ and 
(\ref{etaloc}) reduces to  
$$\eta^{\mathrm{loc}}_{S_{1}}(\delta_{q}) = M(S_{1}, 1),
\nomenclature{$\eta^{\mathrm{loc}}_{S_{1}}$}{virtual character of a pseudopacket}$$
a single standard representation.  The following proposition describes the image of  $\eta^{\mathrm{loc}}_{S_{1}}(\delta_{q})$ under endoscopic lifting.  Its proof  follows from \cite{aam}*{Proposition 5.1} by replacing $\mathrm{GL}_{N}$ with $\mathrm{R}_{\mathbb{C}/\mathbb{R}}\mathrm{GL}_{N}$.
\begin{prop}[\cite{aam}*{Proposition 5.1}]\label{injliftord}

\begin{enumerate}[label={(\alph*)}]

\item  Suppose $S_{1} \subset X({^\vee}\mathcal{O}_{G},
  {^\vee}G^{\Gamma})$ is a ${^\vee}G$-orbit which is carried to the
  ${^\vee}\mathrm{R}_{\mathbb{C}/\mathbb{R}}\mathrm{GL}_{N}$-orbit $S$ under $\epsilon$. Then 
  $$\mathrm{Lift}_{0}\left(\eta^{\mathrm{loc}}_{S_{1}}(\delta_{q}) \right) = \eta^{\mathrm{loc}}_{S},$$
or equivalently,
$$\mathrm{Lift}_{0} \left(M(S_{1}, 1)\right)=  M(S, 1).$$

\item The endoscopic lifting map $\mathrm{Lift}_{0}$ is equal to the
  parabolic induction functor $\mathrm{ind}_{G(\mathbb{R},
    \delta_{q})}^{\mathrm{GL}_{N}(\mathbb{C})}$ on 
  $K_{\mathbb{C}} \Pi({^\vee}\mathcal{O}_{G}, G(\mathbb{R},
  \delta_{q}))^{\mathrm{st}}$. 
\end{enumerate}
\end{prop}

  It is much more difficult to compute the value of
$\mathrm{Lift}_{0}$ on the stable virtual character $\eta^{\mathrm{mic}}_{\psi_{G}}$
given in (\ref{etapsi}).   Let $\psi = \epsilon \circ \psi_{G}$.
According to \cite{ABV}*{Theorem 26.25}
\begin{equation}
\label{ordmiclift}
\mathrm{Lift}_{0}\left(\eta^{\mathrm{mic}}_{\psi_{G}}\right)= \sum_{ \xi \in
  \Xi({^\vee}\mathrm{GL}_{N}^{\Gamma}, {^\vee}\mathcal{O})}  
(-1)^{d(S_{\xi}) - d(S_{\psi})}
\,\chi_{S_{\psi}}^{\mathrm{mic}} (P(\xi) ) \,\pi(\xi) =
\eta^{\mathrm{mic}}_{\psi}. 
\end{equation}
Recall from (\ref{abvdef}) that the ABV-packets $\Pi^{\mathrm{ABV}}_{\psi_{G}}$ and
$\Pi^{\mathrm{ABV}}_{\psi}$ are defined from $\eta^{\mathrm{mic}}_{\psi_{G}}$
and $\eta^{\mathrm{mic}}_{\psi}$ respectively.  We shall see in Section \ref{glnpacket}
that these ABV-packets are singletons. 

\subsection{Twisted endoscopy}
\label{twistendsec}

Our aim in this section is to lay out the twisted versions of the concepts presented in the previous section.  We  define  twisted
endoscopic data relevant to unitary groups, the twisted endoscopic version of $\mathrm{Lift}_{0}$ (\ref{endlift}),
compute  twisted variants of $\mathrm{Lift}_{0}
(\eta^{\mathrm{loc}}_{S})$  for $S\in  X({^\vee}\mathcal{O}_{G}, {^\vee}G^{\Gamma})$, and compute twisted variants of $\mathrm{Lift}_{0} (\eta^{\mathrm{mic}}_{\psi_{G}})$.   We shall work under the assumption of (\ref{regintdom}) on the infinitesimal characters.  This assumption is made only to accommodate the definition of Atlas extensions.

An \emph{endoscopic datum} for $({^\vee}\mathrm{R}_{\mathbb{C}/\mathbb{R}}\mathrm{GL}_{N}^{\Gamma},
\vartheta)$ is a pair
$$(s, \, {^\vee}G^{\Gamma})$$
which satisfies
\begin{enumerate}
\item $s \in {^\vee}\mathrm{R}_{\mathbb{C}/\mathbb{R}} \mathrm{GL}_{N}$ is $\vartheta$-semisimple (see
\cite{KS}*{(2.1.3)}) 

\item ${^\vee}G^{\Gamma} \subset {^\vee}\mathrm{R}_{\mathbb{C}/\mathbb{R}} \mathrm{GL}_{N}^{\Gamma}$ is
  open in the fixed-point set of $\mathrm{Int}(s) \circ
  \vartheta$ in $^{\vee}\mathrm{R}_{\mathbb{C}/\mathbb{R}}\mathrm{GL}_N^\Gamma \rtimes \langle \vartheta\rangle$ 
\item ${^\vee}G^{\Gamma}$ is an E-group for a group $G$ (\cite{ABV}*{Definition 4.6}).
\end{enumerate}
This is a special case of  \cite{Christie-Mezo}*{Definition 5.1} to
${^\vee}\mathrm{R}_{\mathbb{C}/\mathbb{R}}\mathrm{GL}_{N}^{\Gamma}$.  There is a notion of equivalence for these
endoscopic data (\cite{Christie-Mezo}*{Definition 5.6}, 
\cite{KS}*{(2.1.5)-(2.1.6)}).

Let us take $s=1$ in the endoscopic pair above. Then the fixed-point subgroup
$$({^\vee}\mathrm{R}_{\mathbb{C}/\mathbb{R}}\mathrm{GL}_{N})^{\vartheta} = \{ (g, \tilde{J} (g^{\intercal})^{-1} \tilde{J}^{-1}) : g \in {^\vee}\mathrm{GL}_{N}\} \cong {^\vee}\mathrm{GL}_{N}$$
is a legitimate dual group for an endoscopic datum.  Furthermore, by setting
$${^\vee}G^{\Gamma} = ({^\vee}\mathrm{R}_{\mathbb{C}/\mathbb{R}}\mathrm{GL}_{N})^{\vartheta} \rtimes \langle {^\vee}\delta_{0} \rangle$$
with ${^\vee}\delta_{0}$ as in (\ref{del0action}) we see that ${^\vee}G^{\Gamma}$ is isomorphic to the L-group of a quasisplit unitary group as in (\ref{unitarylgroup}).  
We have just shown that  
$$(1, {^\vee}G^{\Gamma}) = (1,{^\vee}\mathrm{GL}_{N}^{\Gamma})$$
is an endoscopic datum for $({^\vee}\mathrm{R}_{\mathbb{C}/\mathbb{R}}\mathrm{GL}_{N}^{\Gamma},
\vartheta)$ and that the corresponding endoscopic group is the rank $N$ quasisplit unitary group.  This is the \emph{only} endoscopic datum of interest to us here.

Unlike the previous section, we must distinguish between maximal tori
in ${^\vee}\mathrm{R}_{\mathbb{C}/\mathbb{R}}\mathrm{GL}_{N}$ and  ${^\vee}G^{\Gamma}$.   We let ${^\vee}H$ be the
diagonal maximal torus in ${^\vee}\mathrm{R}_{\mathbb{C}/\mathbb{R}}\mathrm{GL}_{N}$, and ${^\vee}H_{G}$ be  a maximal torus in ${^\vee}G \cong {^\vee}\mathrm{GL}_{N}$.
The two tori are related by
$${^\vee}H_{G} = ({^\vee}H)^{\vartheta}.$$  
We fix $\lambda \in
{^\vee}\mathfrak{h}^{\vartheta}$.  Let ${^\vee}\mathcal{O}_{G}$ be the
${^\vee}G$-orbit of 
$\lambda$ and ${^\vee}\mathcal{O}$ be the ${^\vee}\mathrm{R}_{\mathbb{C}/\mathbb{R}}\mathrm{GL}_{N}$-orbit of
$\lambda$.

The $\epsilon$ maps of (\ref{epinclusion})-(\ref{epinclusion1}) have
obvious analogues and are equally valid in the twisted setting.
The analogue of (\ref{epinclusion}) is quite transparent as it takes the form
$$\epsilon(g) = (g, \tilde{J}^{-1} (g^{\intercal})^{-1} \tilde{J}^{-1}), \quad g \in {^\vee}G$$
and carries the element ${^\vee}\delta_{0}$ in (\ref{unitarylgroup}) to the element ${^\vee}\delta_{0}$ in (\ref{lgroupglnc}).

The crucial point in the
twisted setting is to include the action of 
$\vartheta$ into the objects pertinent to endoscopy.  In particular we 
must  extend the sheaf theory of \cite{ABV} for
${^\vee}\mathrm{R}_{\mathbb{C}/\mathbb{R}}\mathrm{GL}_{N}$ to the disconnected group
${^\vee}\mathrm{R}_{\mathbb{C}/\mathbb{R}}\mathrm{GL}_{N} \rtimes \langle \vartheta  \rangle$. This
mimics the extension of the representation theory of $\mathrm{GL}_{N}(\mathbb{C})$
to the disconnected group $\mathrm{GL}_{N}(\mathbb{C}) \rtimes \langle \vartheta
\rangle$ in Section \ref{extrepsec}.  Rather than viewing the sheaves
in $\mathcal{C}({^\vee}\mathcal{O}, {^\vee}\mathrm{R}_{\mathbb{C}/\mathbb{R}}\mathrm{GL}_{N}^{\Gamma}
;\vartheta)$ as ${^\vee}\mathrm{R}_{\mathbb{C}/\mathbb{R}}\mathrm{GL}_{N}$-equivariant with compatible
$\vartheta$-action (Section \ref{conperv}), we view them simply as $({^\vee}\mathrm{R}_{\mathbb{C}/\mathbb{R}}\mathrm{GL}_{N}
\rtimes \langle \vartheta \rangle)$-equivariant sheaves and apply the
theory of \cite{ABV} which is valid in this generality 
(\cite{Christie-Mezo}*{Section 5.4}).

Let $\xi = (S, 1)  \in \Xi({^\vee}\mathcal{O}, {^\vee}\mathrm{R}_{\mathbb{C}/\mathbb{R}}
\mathrm{GL}_{N}^{\Gamma})^\vartheta$ and $p \in S$.  Here,
$1$ is the trivial representation  of the trivial group
$({^\vee}\mathrm{R}_{\mathbb{C}/\mathbb{R}} \mathrm{GL}_{N})_{p}/ (({^\vee} \mathrm{R}_{\mathbb{C}/\mathbb{R}}\mathrm{GL}_{N})_{p} )^{0}$
with representation space $V \cong \mathbb{C}$ as in (\ref{bundle}).
We define $1^{+}$ on  
\begin{equation*}
({^\vee}\mathrm{R}_{\mathbb{C}/\mathbb{R}}\mathrm{GL}_{N})_{p} /(({^\vee} \mathrm{R}_{\mathbb{C}/\mathbb{R}}\mathrm{GL}_{N})_{p} )^{0}
\times \langle \vartheta \rangle 
\end{equation*}
 by
\begin{equation}
\label{oneextend}
1^{+} (\vartheta)  = \vartheta_{\mu(\xi)^{+} }= 1
\end{equation}
In this way, $1^{+}$ defines the local system underlying the
irreducible $({^\vee}\mathrm{R}_{\mathbb{C}/\mathbb{R}}\mathrm{GL}_{N} \rtimes \langle \vartheta
\rangle)$-equivariant constructible sheaf $\mu(\xi)^{+}$ (Lemma
\ref{cansheaf},   \cite{ABV}*{\emph{p.} 83}).   

In a similar, but completely vacuous, fashion we may include  the
trivial action of $\vartheta$ on $\mu(\xi_{1}) \in
\mathcal{C}({^\vee}\mathcal{O}_{G}, {^\vee}G)$ with $\xi_{1} = (S_{1},
\tau_{1})$ and $p_{1} \in S_{1}$.  In other words, we may regard
$\mu(\xi_{1})$ as a $({^\vee}G \times \langle \vartheta
\rangle)$-equivariant sheaf whose underlying local system is defined
by a quasicharacter $\tau_{1}^{+}$ on 
\begin{equation}
\label{twistcompgroup}
{^\vee}G_{p_{1}}/ ({^\vee} G_{p_{1}})^{0} \times \langle \vartheta \rangle
\end{equation}
 by $\tau_{1}^{+}(\vartheta) = 1$.

  The inverse image functor
  \begin{equation*}
  \epsilon^{*}:  KX({^\vee}\mathcal{O}, {^\vee}\mathrm{R}_{\mathbb{C}/\mathbb{R}}\mathrm{GL}_{N}^{\Gamma},
  \vartheta) \rightarrow KX({^\vee}\mathcal{O}_{G}, {^\vee}G^{\Gamma})
\end{equation*}
  in the present twisted setting is
defined on  $({^\vee}\mathrm{R}_{\mathbb{C}/\mathbb{R}}\mathrm{GL}_{N} \rtimes \langle \vartheta
\rangle)$-equivariant sheaves (Section \ref{conperv}).
As in standard endoscopy, we combine $\epsilon^{*}$ with a pairing,
namely the pairing of Theorem \ref{twistpairing}, to define  
$$\epsilon_{*}:  K_{\mathbb{C}} \Pi ( {^\vee}\mathcal{O}_{G} ,G/ \mathbb{R}
) \rightarrow K_{\mathbb{C}} \Pi({^\vee}\mathcal{O},
\mathrm{GL}_{N}(\mathbb{C}), \vartheta).$$ 
To be precise, the image of  any $\eta \in K_{\mathbb{C}}
\Pi({^\vee}\mathcal{O}_{G}, G/\mathbb{R})$ under  $\epsilon_{*}$ is
determined by 
\begin{equation}
\label{twistloweps}
\left\langle \epsilon_{*} \eta, \mu(\xi)^{+}\right \rangle = \left\langle
\eta, \epsilon^{*} \mu(\xi)^{+}\right\rangle_{G}, \quad \xi \in
\Xi({^\vee}\mathcal{O},{^\vee}\mathrm{R}_{\mathbb{C}/\mathbb{R}}\mathrm{GL}_{N}^{\Gamma})^\vartheta. 
\end{equation}
(\emph{cf.} (\ref{loweps})).  The twisted endoscopic lifting map 
\begin{equation}
\label{twistendlift}
\mathrm{Lift}_{0}: K_{\mathbb{C}} \Pi({^\vee}\mathcal{O}_{G}, G(\mathbb{R},
\delta_{q}))^{\mathrm{st}} \rightarrow K_{\mathbb{C}} \Pi({^\vee}\mathcal{O},
\mathrm{GL}_{N}(\mathbb{C}), \vartheta) 
\nomenclature{$\mathrm{Lift}_{0}$}{twisted endoscopic lifting map}
\end{equation}
is the restriction of $\epsilon_{*}$ to the stable submodule $K_{\mathbb{C}}
\Pi({^\vee}\mathcal{O}_{G}, G(\mathbb{R}, \delta_{q}))^{\mathrm{st}}$. 

Now, we wish to evaluate $\mathrm{Lift}_{0}$ on the basis elements
(\ref{etaloc}) of $K_{\mathbb{C}} \Pi({^\vee}\mathcal{O}_{G}, G(\mathbb{R},
\delta_{q}))^{\mathrm{st}}$.  To maintain ease of comparison with \cite{ABV} we
evaluate $\mathrm{Lift}_{0}$ on the virtual representations
$\eta_{S_{1}}^{\mathrm{loc}}(\vartheta)(\delta_{q})$ (\cite{ABV}*{\emph{p.} 279}). These virtual characters are defined by 
\begin{equation*}
\eta_{S_{1}}^{\mathrm{loc}}(\vartheta)(\delta_{q}) = \sum_{\tau_{1} }
\mathrm{Tr}( \tau_{1}^{+}(\vartheta)) \, M(S_{1}, \tau_{1}) =
\sum_{\tau_{1} } M(S_{1}, \tau_{1}), 
\nomenclature{$\eta_{S_{1}}^{\mathrm{loc}}(\vartheta)(\delta_{q})$}{}
\end{equation*}
where $\tau_{1}$ runs over all quasicharacters of ${^\vee}G_{p_{1}}/
({^\vee}G_{p_{1}})^{0}$ as in (\ref{twistcompgroup}) which correspond
to the strong involution $\delta_{q}$ (\cite{ABV}*{Definition 18.9}) under (\ref{localLanglandspure}).
It is immediate from the definition of $\tau_{1}^{+}$ following (\ref{twistcompgroup}) that
\begin{equation*}
\eta_{S_{1}}^{\mathrm{loc}}(\vartheta)(\delta_{q}) =
\eta_{S_{1}}^{\mathrm{loc}}(\delta_{q})
\end{equation*}
and so this virtual character is stable (\cite{ABV}*{Lemma 18.10}).  The proof of the following proposition is the same as the proof of Proposition 5.3 \cite{aam} once $\mathrm{GL}_{N}$ is replaced by $\mathrm{R}_{\mathbb{C}/\mathbb{R}}\mathrm{GL}_{N}$.

\begin{prop}[Proposition 5.3 \cite{aam}]
\label{twistimlift}
Suppose $S_{1} \subset X({^\vee}\mathcal{O}_{G}, {^\vee}G^{\Gamma})$ is a
${^\vee}G$-orbit which is carried to a ${^\vee}\mathrm{R}_{\mathbb{C}/\mathbb{R}}\mathrm{GL}_{N}$-orbit
$S$ under $\epsilon$. Then
$$\mathrm{Lift}_{0}\left( \eta^{\mathrm{loc}}_{S_{1}}(\vartheta)(\delta_{q}) \right) = (-1)^{l^{I}(S, 1) - l^{I}_{\vartheta} (S, 1)}
\, M(S, 1)^{+}$$
\end{prop}

\begin{prop}
\label{injlift2}
The twisted endoscopic lifting map $\mathrm{Lift}_{0}$ is injective.
\end{prop}
\begin{proof}
  Suppose $S_{1}, S_{2} \subset X({^\vee}\mathcal{O}_{G}, {^\vee}G^{\Gamma})$
 are ${^\vee}G$-orbits which are carried to the same
 ${^\vee}\mathrm{R}_{\mathbb{C}/\mathbb{R}} \mathrm{GL}_{N}$-orbit under $\epsilon$.  Then, after identifying these orbits with L-parameters (\cite{ABV}*{Proposition 6.17}), \cite{ggp}*{Theorem 8.1}  implies $S_{1} =
 S_{2}$ (\emph{cf.} \cite{Mok}*{Lemma 2.2.1}). 
  It now follows from Proposition \ref{twistimlift} that $\mathrm{Lift}_{0}$ sends the
basis  
$$\left\{\eta^{\mathrm{loc}}_{S_{G}}(\delta_{q}) : S_{G} \mbox{ a }
{^\vee}G\mbox{-orbit of } X({^\vee}\mathcal{O}_{G}, {^\vee}G^{\Gamma}) \right\}$$ 
of $K_{\mathbb{C}} \Pi({^\vee}\mathcal{O}_{G}, G(\mathbb{R},
\delta_{q}))^{\mathrm{st}}$ bijectively onto the  linearly independent subset  
$$\left\{  (-1)^{l^{I}(\epsilon(S_{G}), 1) - l^{I}_{\vartheta} (\epsilon(S_{G}), 1)} M(\epsilon( S_{G}),1)^{+}: S_{G} \mbox{ a }
      {^\vee}G\mbox{-orbit of } X({^\vee}\mathcal{O}_{G}, {^\vee}G^{\Gamma})
      \right\}$$ 
of $K_{\mathbb{C}} \Pi({^\vee}\mathcal{O}, \mathrm{GL}_{N}(\mathbb{C}), \vartheta)$.
\end{proof}

The next and final goal of this section is to provide the twisted
analogue of the endoscopic lifting of the virtual characters attached
to A-parameters as in (\ref{ordmiclift}).  As a guiding principle, it
helps to remember that in moving from $\eta^{\mathrm{loc}}_{S}$ to
$\eta^{\mathrm{loc}}_{S}(\vartheta)(\delta_{q})$ we extended the component
groups by $\langle \vartheta \rangle$ to obtain (\ref{twistcompgroup}),
and then extended the quasicharacters $\tau_{1}$ defined on the
original component groups.  We shall follow the same process with
$\eta^{\mathrm{mic}}_{\psi_{G}}$, doing our best to avoid the theory of microlocal
geometry. 

The stable virtual character (\ref{etapsi}) for the endoscopic group $G$ is 
$$\eta^{\mathrm{mic}}_{\psi_{G}} =  \sum_{\xi \in \Xi({^\vee}\mathcal{O}_{G}, {^\vee}G^{\Gamma})}
(-1)^{d(S_{\xi}) - d(S_{\psi_{G}})} \ \chi^{\mathrm{mic}}_{S_{\psi_{G}}}(P(\xi)) \, \pi(\xi)
\in K\Pi({^\vee}\mathcal{O}_{G}, G/\mathbb{R})^{\mathrm{st}}.$$
Here, $S_{\psi_{G}} \subset X({^\vee}\mathcal{O}_{G}, G^{\Gamma})$  is the
${^\vee}G$-orbit determined by the L-parameter $\phi_{\psi_{G}}$, and $\xi
= (S_{\xi}, \tau_{S_{\xi}})$.
We may rewrite $\eta^{\mathrm{mic}}_{\psi_{G}}$ using the following deep theorem in microlocal analysis.  It is a summary of \cite{ABV}*{Theorem 24.8, Corollary 24.9, Definition 24.15}.
\begin{thm}
  \label{miclocalsys}
For each $\xi \in \Xi({^\vee}\mathcal{O}_{G}, {^\vee}G^{\Gamma})$ there is a
representation $\tau^{\mathrm{mic}}_{S_{\psi_{G}}}(P(\xi))$ of ${^\vee}G_{\psi_{G}}/
({^\vee}G_{\psi_{G}})^{0}$, 
\nomenclature{$\tau^{\mathrm{mic}}_{S_{\psi_{G}}}(P(\xi))$}{representation of component group}
the component group of the centralizer in
${^\vee}G$ of the image of $\psi_{G}$, which satisfies the following
properties
\begin{enumerate}[label={(\alph*)}]

\item  $\tau^{\mathrm{mic}}_{S_{\psi_{G}}}(P(\xi))$ represents  a (possibly
   zero)  ${^\vee}G$-equivariant local system  $Q^{\mathrm{mic}}(P(\xi))$  of complex vector spaces.

\item The degree of  $\tau^{\mathrm{mic}}_{S_{\psi_{G}}}(P(\xi))$
  is equal to $\chi_{S_{\psi_{G}}}^{\mathrm{mic}}(P(\xi))$.  

\item  If  $\xi = (S_{\psi_{G}}, \tau_{S_{\psi_{G}}})$  then 
$\tau^{\mathrm{mic}}_{S_{\psi_{G}}}(P(\xi)) = \tau_{S_{\psi_{G}}} \circ i_{S_{\psi_{G}}}$,
 where 
$$i_{S_{\psi_{G}}}: {^\vee}G_{\psi_{G}}/ ({^\vee}G_{\psi_{G}})^{0}
\rightarrow {^\vee}G_{p}/ ({^\vee}G_{p})^{0}$$
is a surjective  homomorphism for $p \in S_{\psi_{G}}$. 
\nomenclature{$Q^{\mathrm{mic}}(P(\xi))$}{local system}
\nomenclature{$i_{S_{\psi_{G}}}$}{homomorphism of component groups}
\end{enumerate}
\end{thm}
By Theorem \ref{miclocalsys} (b), we may rewrite $\eta^{\mathrm{mic}}_{\psi_{G}}$ as
\begin{equation}
  \label{etapsitrace}
  \eta^{\mathrm{mic}}_{\psi_{G}} =  \sum_{\xi \in \Xi({^\vee}\mathcal{O}_{G}, {^\vee}G^{\Gamma})}
(-1)^{d(S_{\xi}) - d(S_{\psi_{G}})}
  \ \mathrm{Tr}\left(\tau^{\mathrm{mic}}_{S_{\psi_{G}}}(P(\xi))(1)\right) \, \pi(\xi). 
\end{equation}
Next, we extend ${^\vee}G_{\psi_{G}}/ ({^\vee}G_{\psi_{G}})^{0}$ trivially to
\begin{equation}
\label{artgroupext}
{^\vee}G_{\psi_{G}}/ ({^\vee}G_{\psi_{G}})^{0} \times \langle \vartheta \rangle,
\end{equation}
and extend $\tau^{\mathrm{mic}}_{S_{\psi_{G}}}(P(\xi))$ trivially to (\ref{artgroupext}) by defining $\tau^{\mathrm{mic}}_{S_{\psi_{G}}}(P(\xi))(\vartheta)$ to be the identity map.  We define
\begin{align*}
\eta^{\mathrm{mic}}_{\psi_{G}} (\vartheta) &=  \sum_{\xi \in \Xi({^\vee}\mathcal{O}_{G}, {^\vee}G^{\Gamma})}
(-1)^{d(S_{\xi}) - d(S_{\psi_{G}})} \ \mathrm{Tr}\left(\tau^{\mathrm{mic}}_{S_{\psi_{G}}}(P(\xi))(\vartheta)\right) \, \pi(\xi) \\
&=  \sum_{\xi \in \Xi({^\vee}\mathcal{O}_{G}, {^\vee}G^{\Gamma})}
(-1)^{d(S_{\xi}) - d(S_{\psi_{G}})} \ \dim\left(\tau^{\mathrm{mic}}_{S_{\psi_{G}}}(P(\xi)) \right) \, \pi(\xi). 
\end{align*}
Clearly
\begin{equation}
\label{nosigma}
\eta^{\mathrm{mic}}_{\psi_{G}} (\vartheta) = \eta^{\mathrm{mic}}_{\psi_{G}}.
\end{equation}
Finally,  define
\begin{align*}
\eta^{\mathrm{mic}}_{\psi_{G}} (\vartheta)(\delta_{q}) & =\sum_{(S_{\xi}, \tau_{S_{\xi}}) }
(-1)^{d(S_{\xi}) - d(S_{\psi_{G}})}
\ \mathrm{Tr}\left(\tau^{\mathrm{mic}}_{S_{\psi_{G}}}(P(\xi))(\vartheta)\right) \, \pi(\xi)\\
\nonumber &= \sum_{(S_{\xi}, \tau_{S_{\xi}}) }
(-1)^{d(S_{\xi}) - d(S_{\psi_{G}})}
\ \dim\left(\tau^{\mathrm{mic}}_{S_{\psi_{G}}}(P(\xi)) \right) \, \pi(\xi)
\end{align*}
in which the sum runs over only those $\xi = (S_{\xi}, \tau_{S_{\xi}}) \in
\Xi({^\vee}\mathcal{O}_{G}, {^\vee}G^{\Gamma})$ in which $\tau_{S_{\xi}}$
corresponds to the strong involution $\delta_{q}$ under (\ref{localLanglandspure}).
Therefore, by (\ref{etapsiabv})
\begin{equation*}
\eta^{\mathrm{mic}}_{\psi_{G}} (\vartheta)(\delta_{q})= \eta_{S_{\psi_{G}}}^{\mathrm{mic}}(\delta_{q})=\eta_{S_{\psi_{G}}}^{\mathrm{ABV}}.
\end{equation*}
The virtual
character $\eta^{\mathrm{mic}}_{\psi_{G}}(\vartheta) (\delta_{q})$ is a summand of the
stable virtual character $\eta^{\mathrm{mic}}_{\psi_{G}}$ and is therefore also stable
(\cite{ABV}*{Theorem 18.7}).  Consequently, $\eta^{\mathrm{mic}}_{\psi_{G}}(\vartheta)
(\delta_{q})$ lies in the domain of $\mathrm{Lift}_{0}$.  In addition,
the ABV-packet $\Pi^{\mathrm{ABV}}_{\psi_{G}}$ consists of the irreducible
characters in the support of $\eta^{\mathrm{mic}}_{\psi_{G}}(\vartheta) (\delta_{q})$ (\ref{abvdef}).   

What we have done for $\eta^{\mathrm{mic}}_{\psi_{G}}$ we begin to do for 
$\eta_{\psi}^{\mathrm{mic}+}$, which we define as
\begin{equation}
\label{etatilde}
  \eta_{\psi}^{\mathrm{mic}+} = \sum_{\xi \in \Xi({^\vee}\mathcal{O},
  {^\vee}\mathrm{R}_{\mathbb{C}/\mathbb{R}} \mathrm{GL}_{N}^{\Gamma})^\vartheta} 
(-1)^{d(S_{\xi}) - d(S_{\psi})}
\ \mathrm{Tr}(\chi^{\mathrm{mic}}_{S_{\psi}}(P(\xi))) \,
(-1)^{l^{I}(\xi) - l^{I}_{\vartheta}(\xi)} \pi(\xi)^{+}
\end{equation}
for
\begin{equation*}
\psi  = \epsilon \circ \psi_{G}.
\end{equation*}
The main difference now is that $\vartheta$ does not act trivially on
${^\vee}\mathrm{R}_{\mathbb{C}/\mathbb{R}}\mathrm{GL}_{N}$ and so the extensions require more attention.
The properties of Theorem \ref{miclocalsys} hold for
$\psi$ and $\mathrm{R}_{\mathbb{C}/\mathbb{R}}\mathrm{GL}_{N}$ as they do for $\psi_{G}$ and $G$.   

The first step is writing
$$\eta_{\psi}^{\mathrm{mic}+} =  \sum_{\xi \in \Xi({^\vee}\mathcal{O},
  {^\vee}\mathrm{R}_{\mathbb{C}/\mathbb{R}}\mathrm{GL}_{N}^{\Gamma})^\vartheta} 
(-1)^{d(S_{\xi}) - d(S_{\psi})}
\ \mathrm{Tr}(\tau^{\mathrm{mic}}_{S_{\psi}}(P(\xi))(1)) \,
(-1)^{l^{I}(\xi) - l^{I}_{\vartheta}(\xi)}  \pi(\xi)^{+}.
\nomenclature{$\eta_{\psi}^{\mathrm{mic}+}$}{virtual twisted character}$$ 
This holds from Theorem \ref{miclocalsys} (b)  as  (\ref{etapsitrace}) did for the endoscopic group $G$.  What is  simpler here is that the component group $({^\vee}\mathrm{R}_{\mathbb{C}/\mathbb{R}}\mathrm{GL}_{N})_{\psi}/ (({^\vee}\mathrm{R}_{\mathbb{C}/\mathbb{R}}\mathrm{GL}_{N})_{\psi})^{0}$ is trivial (\cite{Arthur84}*{Section 2.3}). It follows that $\tau^{\mathrm{mic}}_{S_{\psi}}(P(\xi)) $ is either  trivial  or zero.

Let us digress briefly to examine Theorem \ref{miclocalsys} (c) for $\xi = (S_{\psi}, \tau_{S_{\psi}})$.  Since the component group $({^\vee}\mathrm{R}_{\mathbb{C}/\mathbb{R}}\mathrm{GL}_{N})_{p}/ (({^\vee}\mathrm{R}_{\mathbb{C}/\mathbb{R}}\mathrm{GL}_{N})_{p})^{0}$ is trivial, the representation $\tau^{S_{\psi}}$ is trivial.  It follows that 
\begin{equation*}
\tau^{\mathrm{mic}}_{S_{\psi}}(P(S_{\psi}, \tau_{S_{\psi}})) = \tau_{S_{\psi}}\circ i_{S_{\psi}} = 1 \circ i_{S_{\psi} }= 1 \neq 0.
\end{equation*}
In particular, $\pi(S_{\psi}, 1)$ is in the support of $\eta^{\mathrm{mic}}_{\psi}$ and belongs to $\Pi^{\mathrm{ABV}}_{\psi}$.  In the next section we will prove that this is the only representation in $\Pi^{\mathrm{ABV}}_{\psi}$. 

Returning to the matter of extensions, there is an obvious extension
$$({^\vee}\mathrm{R}_{\mathbb{C}/\mathbb{R}}\mathrm{GL}_{N})_{\psi}/ (({^\vee}\mathrm{R}_{\mathbb{C}/\mathbb{R}}\mathrm{GL}_{N})_{\psi})^{0} \times \langle \vartheta \rangle$$
of the trivial component group, as $\vartheta$ fixes the image of $\psi$.  We wish to extend the representation $\tau_{S_{\psi}}^{\mathrm{mic}} (P(\xi))$ to this group for  $\xi \in  \Xi({^\vee}\mathcal{O}, {^\vee}\mathrm{R}_{\mathbb{C}/\mathbb{R}}\mathrm{GL}_{N}^{\Gamma})^\vartheta$.  
The action of $\vartheta$ on $P(\xi) \in \mathcal{P}({^\vee}\mathcal{O},
{^\vee}\mathrm{R}_{\mathbb{C}/\mathbb{R}}\mathrm{GL}_{N}^{\Gamma}; \vartheta)$ determines an action on
the stalks of the local system $Q^{\mathrm{mic}}(P(\xi))$ as in
(\ref{miclocalsys}) (\cite{ABV}*{(25.1)}).   
\cite{ABV}*{Proposition 26.23 (b)} allows us to choose a stalk over a $\vartheta$-fixed point
$p'$ (related to $S_{\psi}$) in the topological space of
$Q^{\mathrm{mic}}(P(\xi))$.  This places us in the same setting as Lemma
\ref{cansheaf}, with
$\tau_{S}$ replaced by$\tau_{S_{\psi}}^{\mathrm{mic}}(P(\xi))$ and 
 $S$ replaced by the
${^\vee}\mathrm{R}_{\mathbb{C}/\mathbb{R}}\mathrm{GL}_{N}$-orbit of $p'$.   As in that lemma,  $\vartheta$
determines a canonical isomorphism of the stalk at $p'$ equal to $1$.
In short, we define  
\begin{equation}
\label{oneextend1}
\tau_{S_{\psi}}^{\mathrm{mic}}(P(\xi)^{+})(\vartheta) =1
\end{equation}
 and extend $\tau_{S_{\psi}}^{\mathrm{mic}}(P(\xi))$ to a representation
 $\tau_{S_{\psi}}^{\mathrm{mic}}(P(\xi)^{+})$. 
 \nomenclature{$\tau_{S_{\psi}}^{\mathrm{mic}}(P(\xi)^{+})$}{}
  The representation
 $\tau_{S_{\psi}}^{\mathrm{mic}}(P(\xi)^{+})$ represents the
 $({^\vee}\mathrm{R}_{\mathbb{C}/\mathbb{R}}\mathrm{GL}_{N} \rtimes \langle \vartheta
 \rangle)$-equivariant local system of the restriction of
 $Q^{\mathrm{mic}}(P(\xi))$ to the orbit of $p'$.   
We may extend  $i_{S_{\psi}}$ in Theorem \ref{miclocalsys} (c) to include
the products with $\langle \vartheta \rangle$.  Definitions
(\ref{oneextend}) and (\ref{oneextend1})  are compatible in that 
$$\tau_{S_{\psi}}^{\mathrm{mic}}(P(S_{\psi},1))^{+}) = 1^{+} \circ i_{S_{\psi}}.$$

Finally, we define 
\begin{equation}
\label{etaplus}
\eta_{\psi}^{\mathrm{mic}+}(\vartheta) = \sum_{\xi \in \Xi({^\vee}\mathcal{O},
  {^\vee}\mathrm{R}_{\mathbb{C}/\mathbb{R}}\mathrm{GL}_{N}^{\Gamma})^\vartheta} 
(-1)^{d(S_{\xi}) - d(S_{\psi})}
\ \mathrm{Tr}(\tau^{\mathrm{mic}}_{S_{\psi}}(P(\xi)^{+})(\vartheta)) \,
(-1)^{l^{I}(\xi) - l^{I}_{\vartheta}(\xi)}  \pi(\xi)^{+}.
\nomenclature{$\eta_{\psi}^{\mathrm{mic}+}(\vartheta)$}{}
\end{equation}
It is clear from definition (\ref{etatilde}) that
$\eta_{\psi}^{\mathrm{mic}+}(\vartheta) = \eta_{\psi}^{\mathrm{mic}+}$.

The obvious definition of the representation
$\tau_{S_{\psi}}^{\mathrm{mic}}(P(\xi)^{-})$ is to take
$$\tau_{S_{\psi}}^{\mathrm{mic}}(P(\xi)^{-})(\vartheta) = -1.$$
With this
definition in place the following proposition is a consequence of
\cite{ABV}*{Corollary 24.9}. 
\begin{prop}
\label{24.9c}
The functor $\tau_{S_{\psi}}^{\mathrm{mic}}(\cdot)$, from
$({^\vee}\mathrm{R}_{\mathbb{C}/\mathbb{R}}\mathrm{GL}_{N} \rtimes \langle \vartheta \rangle
)$-equivariant perverse sheaves to representations of
$\left( {^\vee}G_{\psi}/ ({^\vee}G_{\psi})^{0} \right) \times \langle \vartheta
\rangle$, induces a map from the Grothendieck group 
$K(X({^\vee}\mathcal{O}, {^\vee}\mathrm{R}_{\mathbb{C}/\mathbb{R}}\mathrm{GL}_{N}^{\Gamma}); \vartheta)$ to the
space of virtual representations.    Furthermore the \emph{microlocal
  trace} map   
$$\mathrm{Tr} \, \left(\tau_{S_{\psi}}^{\mathrm{mic}} (\cdot)(\vartheta) \right)$$
induces a homomorphism from $K(X({^\vee}\mathcal{O},
{^\vee}\mathrm{R}_{\mathbb{C}/\mathbb{R}}\mathrm{GL}_{N}^{\Gamma}), \vartheta)$ (as in
(\ref{twistsheafgroth})) to $\mathbb{C}$.  
\end{prop}
A similar statement is true for $\tau_{S_{\psi_{G}}}^{\mathrm{mic}}$ and the
$({^\vee}G \times \langle \vartheta \rangle)$-equivariant sheaves
defined earlier.  The proof of the next theorem is identical to the proof of \cite{aam}*{Theorem 5.6}.   
\begin{thm}{\cite{aam}*{Theorem 5.6}}\label{thm:etaplus1}
\begin{enumerate}[label={(\alph*)}]
\item As a function on $K(X({^\vee}\mathcal{O},
  {^\vee}\mathrm{R}_{\mathbb{C}/\mathbb{R}}\mathrm{GL}_{N}^{\Gamma}), \vartheta)$ we have 
$$\left\langle \eta_{\psi}^{\mathrm{mic}+}(\vartheta), \cdot \right\rangle = (-1)^{d(S_{\psi})} \,
    \mathrm{Tr} \, \left(\tau_{S_{\psi}}^{\mathrm{mic}} (\cdot)(\vartheta) \right).$$

\item The stable virtual character $\eta_{\psi}^{\mathrm{mic}+}(\vartheta)$ is equal to 
  $$ (-1)^{d(S_{\psi})} \sum_{\xi \in \Xi({^\vee}\mathcal{O},
  {^\vee}\mathrm{R}_{\mathbb{C}/\mathbb{R}}\mathrm{GL}_{N}^{\Gamma})^\vartheta} \mathrm{Tr} \,
  \left(\tau_{S_{\psi}}^{\mathrm{mic}} (\mu(\xi)^{+})(\vartheta)
  \right)\,  (-1)^{l^{I}(\xi)-l^{I}_{\vartheta}(\xi)} \, M(\xi)^{+}.$$

\item $\mathrm{Lift}_{0}
  \left(\eta^{\mathrm{mic}}_{\psi_{G}}(\vartheta)(\delta_{q}) \right) =
  \eta_{\psi}^{\mathrm{mic}+}(\vartheta).$ 

\end{enumerate}
\end{thm}

\section{ABV-packets for complex general linear groups}
\label{glnpacket}

In this section we prove that any ABV-packet for $
\mathrm{R}_{\mathbb{C}/\mathbb{R}}\mathrm{GL}(\mathbb{R})=\mathrm{GL}_{N}(\mathbb{C})$ consists of a single (equivalence class of an)
irreducible representation. This implies that such an  
ABV-packet is equal to its corresponding L-packet (\cite{ABV}*{Theorem 22.7 (a)}). From the classification of the unitary dual of
$\mathrm{GL}_{N}(\mathbb{C})$ we shall deduce that the single
representation in the packet is unitary. 

In this section we let
$$\psi:W_{\mathbb{R}} \times \mathrm{SL}_{2} \rightarrow
{}^{\vee}\mathrm{R}_{\mathbb{C}/\mathbb{R}}\mathrm{GL}_N^{\Gamma}$$ 
be an arbitrary A-parameter for
$\mathrm{R}_{\mathbb{C}/\mathbb{R}}\mathrm{GL}_N$. The description of the ABV-packet
$\Pi_{\psi}^{\mathrm{ABV}}$ will be achieved in three steps. We follow the 
same proof as in  \cite{aam}*{Section 6}. First, we
treat the case of an \emph{irreducible} A-parameter. 
Second, we compute the ABV-packet for a Levi subgroup of
$\mathrm{R}_{\mathbb{C}/\mathbb{R}}\mathrm{GL}_N$, whose dual group contains  the image of
$\psi$ minimally. The final result is obtained from the second
step by considering the Levi subgroup as an endoscopic group of $\mathrm{R}_{\mathbb{C}/\mathbb{R}}\mathrm{GL}_N$ and
applying the endoscopic lifting (\ref{ordmiclift}).

According to \cite{Mok}*{Section 2.3}, any
A-parameter $\psi$ for $\mathrm{R}_{\mathbb{C}/\mathbb{R}}\mathrm{GL}_N$ may be
decomposed as a formal direct sum of A-parameters
\begin{align}\label{eq:Arthurparameterdecomposition}
\psi=\boxplus _{i=1}^{r}\ell_{i} \psi_i, 
\end{align}
with $\ell_i\in\mathbb{N}$, $\psi_i$ being an 
A-parameter of $\mathrm{R}_{\mathbb{C}/\mathbb{R}}\mathrm{GL}_{N_{i}}$, and  $N = \sum_{i=1}^{r} \ell_{i} N_{i}$.
We may identify \emph{real} L-parameters of 
$\mathrm{R}_{\mathbb{C}/\mathbb{R}}\mathrm{GL}_{N_{i}}$ with their corresponding \emph{complex} L-parameters of $\mathrm{GL}_{N_{i}}$ (\cite{borel}*{Section I.5}). This correspondence extends in an evident way to an analogous identification between A-parameters. 
The complex A-parameter of 
$\mathrm{GL}_{N_{i}}$ corresponding to each $\psi_{i}$ is given by
\begin{align*}
 \mu_{i} \boxtimes \nu_{N_{i}},
\nomenclature{$\mu_{r} \boxtimes \nu_{n_{r}}$}{}
\end{align*}
where $\nu_{N_{i}}$ is the unique irreducible 
representation of $\mathrm{SL}_{2}$ of dimension $N_{i}$,
and $\mu_{i}$ is an irreducible representation $\mathbb{C}^{\times}$. 

The parameter $\psi$ in
(\ref{eq:Arthurparameterdecomposition}) is said to be
\emph{irreducible} if  $r = 1$ and  $\ell_{1} = 1$. 
For any irreducible A-parameter $\psi$ of $\mathrm{R}_{\mathbb{C}/\mathbb{R}}\mathrm{GL}_N$
the corresponding representation $\nu_{N}$ of $\mathrm{SL}_2$ is irreducible and of dimension $N$.  As a consequence the  image of any unipotent subgroup of $\mathrm{SL}_{2}$ under 
$\nu_{N}$ is \emph{principally unipotent} (\emph{i.e.} regular and unipotent) in ${^\vee}\mathrm{GL}_{N}$.  Equivalently, the image of any unipotent subgroup of $\mathrm{SL}_{2}$ under the real A-parameter $\psi$ is principally unipotent 
in ${}^{\vee}\mathrm{GL}_N\times {}^{\vee}\mathrm{GL}_N$. \cite{arancibia_characteristic}*{Theorem 4.11 (d)}  therefore implies 
the following result  (\emph{cf.} \cite{ABV}*{Theorem 27.18}).
\begin{prop}\label{prop:irreducibleArthurparameter}
  Suppose $\psi$ is an irreducible A-parameter of
  $\mathrm{R}_{\mathbb{C}/\mathbb{R}}\mathrm{GL}_N$. 
  Then $\Pi_{\psi}^{\mathrm{ABV}}$ consists of a single unitary character. 
\end{prop}
Let us proceed to the case of a general A-parameter
$\psi$ as in Equation (\ref{eq:Arthurparameterdecomposition}). 
Define
\begin{equation}\label{eq:Hdecomposition}
  {}^{\vee}G =\prod_{i=1}^{r}\left({}^{\vee}G_{i}\right)^{\ell_{i}}
 \cong \prod_{i=1}^{r} ({}^{\vee}\mathrm{GL}_{N_{i}}\times {}^{\vee}\mathrm{GL}_{N_{i}})^{\ell_{i}}\\
  \end{equation}  
to be the obvious Levi subgroup of ${^\vee} \mathrm{GL}_{N}
\times {^\vee} \mathrm{GL}_{N}$ containing
the image of $\psi$ minimally. 
Set ${^\vee}G^{\Gamma}={}^{\vee}G\rtimes \langle {^\vee}\delta_{0} \rangle$, a subgroup of ${}^{\vee}\mathrm{R}_{\mathbb{C}/\mathbb{R}}\mathrm{GL}_N^{\Gamma}$. 
It is immediate that $\psi$ factors through an A-parameter 
\begin{align*}
 \psi:W_{\mathbb{R}} \times \mathrm{SL}_{2}
 \xrightarrow{\psi_{G}} {}^{\vee}G^{\Gamma}\hookrightarrow
            {}^{\vee}\mathrm{R}_{\mathbb{C}/\mathbb{R}}\mathrm{GL}_N^{\Gamma}, 
\end{align*}
where $\psi_{G} = \times_{i=1}^{r} \ell_{i} \psi_{G_i}$ and each  $\psi_{G_i}$ is an
irreducible A-parameter of ${}^{\vee} 
G_{i}^{\Gamma}={}^{\vee}\mathrm{R}_{\mathbb{C}/\mathbb{R}}\mathrm{GL}_{N_i}^{\Gamma}$.  
The description of the ABV-packet corresponding to $\psi_{G}$ is a straightforward
consequence of
Proposition \ref{prop:irreducibleArthurparameter}.  We must only remind ourselves that the direct product of (\ref{eq:Hdecomposition}) translates into a tensor product of ABV-packets as it passes through the process defining the packets in Section \ref{sec:ABV-packetsdef}. The proof follows exactly as for Corollary 6.2
\cite{aam}.
\begin{cor}
\label{cor:LeviABVpacket}
The $\mathrm{ABV}$-packet $\Pi_{\psi_{G}}^{\mathrm{ABV}}$ consists of a single irreducible unitary representation
$\pi(S_{\psi_{G}}, 1)$.
\end{cor}

Finally, take ${}^{\vee}G$ as in (\ref{eq:Hdecomposition}), and
    take $s \in  Z({^\vee}G) \subset {^\vee}\mathrm{GL}_{N}\times {^\vee}\mathrm{GL}_{N}$
    to be as regular as possible so that its centralizer in
    ${^\vee}\mathrm{GL}_{N}\times {^\vee}\mathrm{GL}_{N}$ is equal to 
    ${^\vee}G$.  Then $(s,
    {^\vee}G^{\Gamma})$ is an endoscopic datum for
    ${}^{\vee}\mathrm{R}_{\mathbb{C}/\mathbb{R}}\mathrm{GL}_N^{\Gamma}$  (Section \ref{standend}). 
According to (\ref{ordmiclift}), Corollary \ref{cor:LeviABVpacket},
and Proposition \ref{injliftord}  we have  
    $$\eta^{\mathrm{mic}}_{\psi} = \mathrm{Lift}_{0}\left( \eta^{\mathrm{mic}}_{\psi_{G}}\right) =
\mathrm{Lift}_{0} \left(\pi(S_{\psi_{G}}, 1)\right) =
\mathrm{ind}_{G(\mathbb{R}, \delta_{q})}^{\mathrm{GL}_{N}(\mathbb{C})}
\pi(S_{\psi_{G}}, 1).$$ 
Consequently, $\Pi_{\psi}^{\mathrm{ABV}}$ consists of a single representation, and this representation is parabolically induced from the single unitary representation of $\Pi_{\psi_G}^{\mathrm{ABV}}$. 
Since parabolic induction for general linear groups takes irreducible
unitary representations to irreducible unitary representations
(\cite{Tadic}*{Proposition 2.1}), we have just proved

\begin{prop}
\label{prop:singletonGLN}
  Let $\psi$ be an A-parameter for $\mathrm{R}_{\mathbb{C}/\mathbb{R}}\mathrm{GL}_N$ as in (\ref{eq:Arthurparameterdecomposition}).
  Then the $\mathrm{ABV}$-packet $\Pi_{\psi}^{\mathrm{ABV}}$ consists of a single irreducible unitary representation  $\pi(S_{\psi}, 1)$. 
\end{prop}
As a corollary, we have the next result. Its proof is the same as that of
\cite{aam}*{Corollary 6.4}.   
\begin{cor}
\label{etaplus2}
The stable virtual character $\eta_{\psi}^{\mathrm{mic}+}(\vartheta)$ defined in (\ref{etaplus}) is equal to 
 $(-1)^{l^{I}(\xi)-l^{I}_{\vartheta}(\xi)} \pi(\xi)^{+}$,
where $\xi = (S_{\psi}, 1)$.  In particular, 
$$\mathrm{Lift}_{0}\left(\eta^{\mathrm{mic}}_{\psi_{G}}(\vartheta)(\delta_{q})\right) = (-1)^{l^{I}(\xi)-l^{I}_{\vartheta}(\xi)} \pi(\xi)^{+}.$$ 
\end{cor}

\section{Whittaker extensions and their relationship to Atlas
  extensions}
\label{whitsec}

Recall that (\ref{canext}) defines a preferred extension to $\mathrm{GL}_{N}(\mathbb{C}) \rtimes \langle \vartheta \rangle$ of any irreducible representation of $\mathrm{GL}_{N}(\mathbb{C})$.  These are the extensions we have called ``Atlas extensions''.  Since $\mathrm{GL}_{N}(\mathbb{C})$ is quasisplit as a real group, there is another choice of preferred extension, which depends on a Whittaker datum.  This extension is introduced by Arthur (\cite{Arthur}*{Section 2.2}) and followed by Mok (\cite{Mok}*{Section 3.2}).  We call this alternative extension the \emph{Whittaker extension}.

Here is a summary of the definition of Whittaker extensions. 
We fix a unitary character $\chi$  \nomenclature{$\chi$}{unitary character of $U(\mathbb{R})$}
 on the
upper-triangular unipotent subgroup
$$U(\mathbb{R}) \subset
\mathrm{R}_{\mathbb{C}/\mathbb{R}}\mathrm{GL}_{N}(\mathbb{R}) = \mathrm{GL}_{N}(\mathbb{C}) \nomenclature{$U$}{upper-triangular unipotent subgroup}$$
which satisfies $\chi \circ \vartheta = \chi$.   In this manner   
$(U, \chi)$ \nomenclature{$(U, \chi)$}{Whittaker datum} is a
$\vartheta$-fixed Whittaker datum.  We work under the hypothesis of
(\ref{regintdom}) on an infinitesimal character $\lambda \in
{^\vee}\mathfrak{h}$ and set ${^\vee}\mathcal{O}$ to be its
${^\vee}\mathrm{R}_{\mathbb{C}/\mathbb{R}}\mathrm{GL}_{N}$-orbit.   
Let $\xi \in \Xi({^\vee}\mathcal{O},
{^\vee}\mathrm{R}_{\mathbb{C}/\mathbb{R}}\mathrm{GL}_{N}^{\Gamma})^\vartheta$ so that $\pi(\xi)$
is (an infinitesimal equivalence class of) an irreducible
representation of $\mathrm{GL}_{N}(\mathbb{C})$.   In
defining Whittaker extensions we must work with an actual
admissible group representation in this equivalence class which we 
also denote by $(\pi(\xi),V)$. If $\pi(\xi)$  is 
tempered then up to a scalar there is a unique Whittaker functional $\omega: V
\rightarrow \mathbb{C}$ satisfying
\begin{equation}
\label{whittfunctional}
\omega(\pi(\xi)(u)v) = \chi(u)\, \omega(v), \quad u \in U(\mathbb{R}),
\end{equation}
for all smooth vectors $v \in V$.  
It follows that there is a unique operator $\mathcal{I}^{\thicksim}$
which intertwines $\pi(\xi) \circ \vartheta$ with $\pi(\xi)$ and also satisfies
$\omega \circ \mathcal{I}^{\thicksim} = \omega$.  We extend
$\pi(\xi)$ to a representation $\pi(\xi)^{\thicksim}$ of
$\mathrm{GL}_{N}(\mathbb{C}) \rtimes \langle \vartheta \rangle$ 
by setting $\pi(\xi)^{\thicksim}(\vartheta) =
\mathcal{I}^{\thicksim}$.  We call this extension 
$\pi(\xi)^{\thicksim}\nomenclature{$\pi(\xi)^{\thicksim}$}{(equivalence class of) the  Whittaker extension of $\pi(\xi)$}$ 
the \emph{Whittaker extension} of $\pi(\xi)$.

If $\pi(\xi)$ is not tempered then we express it as the Langlands
quotient of a representation $M(\xi)$ induced from an essentially tempered
representation of a Levi subgroup. The $\vartheta$-stability of
$\pi(\xi)$  and the uniqueness statement in the Langlands
classification together imply the 
$\vartheta$-stability of the essentially tempered representation.  The
earlier argument for tempered representations has an obvious analogue
for the essentially tempered representation of the Levi subgroup.  We
may argue as above to extend the essentially tempered representation
to the semi-direct product of the Levi subgroup with $\langle
\vartheta \rangle$.  One then induces this extended representation to
$\mathrm{GL}_{N}(\mathbb{C}) \rtimes \langle \vartheta  \rangle$.  The
unique irreducible quotient of this representation is the canonical
extension of $\pi(\xi)$, namely the \emph{Whittaker extension}
$\pi(\xi)^{\thicksim}$ of $\pi(\xi)$.   If one omits the Langlands
quotient in this argument then we obtain, by definition, the Whittaker extension
$M(\xi)^{\thicksim}
\nomenclature{$M(\xi)^{\thicksim}$}{(equivalence class of) the Whittaker extension of
  $M(\xi)$}$ of the standard representation $M(\xi)$.
A Whittaker functional for the tempered representation of the Levi subgroup induces a Whittaker functional for $M(\xi)$ (\cite{Sha81}*{Proposition 3.2}). It is a simple exercise to prove that $M(\xi)^{\thicksim}(\vartheta)$ is the unique intertwining operator that fixes an induced  Whittaker functional as in (\ref{whittfunctional}).

How does the Whittaker extension $\pi(\xi)^{\thicksim}$ differ from the Atlas extension $\pi(\xi)^{+}$ of (\ref{canext})?  The operators $\pi(\xi)^{\thicksim}(\vartheta)$ and
$\pi(\xi)^{+}(\vartheta)$ are involutive, and  both intertwine
$\pi(\xi) \circ \vartheta$ with $\pi(\xi)$. Therefore they are equal up to a
sign, \emph{i.e.}
\begin{equation}
\label{normsigns}
\pi(\xi)^{\thicksim}(\vartheta) = \pm \, \pi(\xi)^{+}(\vartheta).
\end{equation}
A direct comparison of the two extensions is problematic in that the Whittaker extension is essentially analytic in nature, and the Atlas extension is essentially algebraic.

Happily, there is a special type of parameter $\xi \in \Xi({^\vee}\mathcal{O},
{^\vee}\mathrm{R}_{\mathbb{C}/\mathbb{R}}\mathrm{GL}_{N}^{\Gamma})^{\vartheta}$ for which the construction of the two extensions is directly comparable.  To describe this type, we must examine the construction of $\pi(\xi)$ in (\ref{canext}) in more detail.  In this construction we identify $\xi$ with its equivalent Atlas parameter $(x,y)$ (Lemma \ref{XXXi}), or its equivalent preferred extended parameter  (\ref{pextparam}).  Recall that $\theta_{x}$ is an automorphism of $H$ (\ref{thetax}).  To the parameter $(x,y)$ one associates the irreducible $(\mathfrak{h}, H^{\theta_{x}} \rtimes \langle \vartheta \rangle)$-module $\pi_{0}^{+}$ of differential $\lambda$ such that  $\pi_{0}^{+}(\vartheta) = 1$ (\cite{AVParameters}*{(20b), Lemma 5.1, Definition 5.6}).    Then to the module $\pi_{0}^{+}$ one applies the functor \cite{AVParameters}*{(20e)} to obtain $\pi(\xi)^{+} = J(x,y,\lambda)^{+}$ (\emph{cf.} \cite{Knapp-Vogan}*{(11.54b), (11.116b)}).  The functor depends on $\theta_{x}$.  In the particular circumstance that $\theta_{x}$ sends all positive roots (determined by  the Borel subgroup $B$ in the pinning of $\mathrm{R}_{\mathbb{C}/\mathbb{R}} \mathrm{GL}_{N}$) to negative roots, the Borel subalgebra $\mathfrak{b}$ is \emph{real} relative to $\theta_{x}$ and $B$ is a real parabolic subgroup (\cite{AvLTV}*{Proposition 13.12 (2)}).  Moreover, in this circumstance, the functor is equivalent to the (normalized) parabolic induction functor $\mathrm{ind}_{B(\mathbb{R}) \rtimes \langle \vartheta \rangle}^{\mathrm{GL}_{N}(\mathbb{C}) \rtimes \langle \vartheta \rangle}$ (\cite{Knapp-Vogan}*{Proposition 11.47}), and so
\begin{equation}
  \label{equext}
  M(\xi)^{+} =\mathrm{ind}_{B(\mathbb{R}) \rtimes \langle \vartheta \rangle}^{\mathrm{GL}_{N}(\mathbb{C}) \rtimes \langle \vartheta \rangle} \pi_{0}^{+}.
\end{equation}
The following lemma shows that the assumption of the $\vartheta$-stability of the parameter $\xi$ is not necessary when $\theta_{x}$ sends all positive roots to negative roots, and asserts the equality of the Atlas and Whittaker extensions.
\begin{lem}
  \label{prinsame}  
Suppose $\xi = (x,y) \in \Xi({^\vee}\mathcal{O},
{^\vee}\mathrm{R}_{\mathbb{C}/\mathbb{R}}\mathrm{GL}_{N}^{\Gamma})$ and $\theta_{x}$ sends all positive roots to negative roots. 
Then
\begin{enumerate}[label={(\alph*)}]
\item $M(\xi)$ and $\pi(\xi)$ are $\vartheta$-stable so that $\xi = (x,y) \in \Xi({^\vee}\mathcal{O},
{^\vee}\mathrm{R}_{\mathbb{C}/\mathbb{R}}\mathrm{GL}_{N}^{\Gamma})^{\vartheta}$.
\item The Whittaker and Atlas 
  extensions of  $M(\xi)$ and $\pi(\xi)$  are equal, \emph{i.e.}       $M(\xi)^{\thicksim} = M(\xi)^{+}$ and $\pi(\xi)^{\thicksim} = \pi(\xi)^{+}$.
  \end{enumerate}
\end{lem}
\begin{proof}
  The assertions  for $\pi(\xi)$ follow by definition from the assertions for $M(\xi)$. For  $M(\xi)$ we imitate the process described before the lemma and thereby obtain an irreducible $(\mathfrak{h}, H^{\theta_{x}})$-module $\pi_{0}$ of differential $\lambda$.  In addition,
$$M(\xi) = \mathrm{ind}_{B(\mathbb{R})}^{\mathrm{GL}_{N}(\mathbb{C})}  \pi_{0}.$$
  As a representation of the connected Lie group $H(\mathbb{R}) \cong (\mathbb{C}^{\times})^{N}$, $\pi_{0}$ is completely determined by its differential $\lambda$. Obviously, $\pi_{0} \circ \vartheta$ has differential $\vartheta(\lambda) = \lambda$ (\ref{regintdom}).  Consequently $ \pi_{0} \circ \vartheta = \pi_{0}$ and one may verify that  
  $$M(\xi) \circ \vartheta \cong  \mathrm{ind}_{\vartheta(B(\mathbb{R}))}^{\mathrm{GL}_{N}(\mathbb{C})} \left( \pi_{0} \circ \vartheta \right) =  \mathrm{ind}_{B(\mathbb{R})}^{\mathrm{GL}_{N}(\mathbb{C})} \pi_{0} = M(\xi).$$
This proves (a).
  
For (b) we may return to (\ref{equext}) and note that, vacuously, $\pi_{0}^{+}(\vartheta)=1$ determines the Whittaker extension of $\pi_{0}$, \emph{i.e.}  $\pi_{0}^{+} = \pi_{0}^{\thicksim}$.   Therefore, by definition, $M(\xi)^{\thicksim}$ is equal to (\ref{equext}).  (The reader may wonder in what sense (\ref{equext}) is ``standard''.  An explanation may be found in \cite{Knapp-Vogan}*{Theorem 11.129 (a)}).
\end{proof}
Lemma \ref{prinsame} provides a simple solution for the comparison of Atlas and Whittaker extensions of some special representations.  As we shall soon see, an arbitrary irreducible representation appears as a subquotient of one of these special standard representations.  Remarkably, the twisted multiplicity (\ref{twistmult}) with which this subquotient appears may be used to determine the sign in (\ref{normsigns}).

We begin this line of reasoning by recalling how irreducible generic representations appear in the characters of standard representations.  Recall that a representation is \emph{generic} if it admits a non-zero Whittaker functional as in (\ref{whittfunctional}).
\begin{lem}
\label{uniquegeneric}
Suppose $\xi \in \Xi({^\vee}\mathcal{O}, {^\vee}\mathrm{R}_{\mathbb{C}/\mathbb{R}}
\mathrm{GL}_{N}^{\Gamma})^\vartheta$.  Then  
\begin{enumerate}[label={(\alph*)}]
\item (up to infinitesimal equivalence)
there is a unique irreducible generic
representation $\pi(\xi_{0}) = M(\xi_{0})$ which occurs in $M(\xi)$ as a subquotient.  Moreover $\pi(\xi_{0})$ is $\vartheta$-stable and occurs as
a subquotient with multiplicity one;   

\item (any representative in the class of) $\pi(\xi_{0})$ embeds as a
  subrepresentation of (any representative in the class of)  $M(\xi)$; 

\item (any representative in the class of) $\pi(\xi_{0})^{\thicksim}$
  embeds as a subrepresentation of (any representative in the class
  of)  $M(\xi)^{\thicksim}$. 
\end{enumerate}
\end{lem}
\begin{proof}
A result due to Vogan and Kostant states that every standard
representation $M(\xi)$ contains a unique generic irreducible
subquotient occurring with multiplicity one (\cite{Kostant78}*{Theorems E and L}, \cite{Vogan78}*{Corollary 6.7}).  In the rest of the
proof we write $\pi(\xi_{0})$ for the actual generic representation
(not the equivalence class) for some $\xi_{0 } \in \Xi({^\vee}\mathcal{O},
{^\vee}\mathrm{R}_{\mathbb{C}/\mathbb{R}}\mathrm{GL}_{N}^{\Gamma})$.  It is straightforward to
verify that $\pi(\xi_{0}) \circ \vartheta$ satisfies
(\ref{whittfunctional}), just as $\pi(\xi_{0})$ does.  Therefore
$\pi(\xi_{0}) \circ \vartheta$ is the unique irreducible generic
subquotient of $M(\xi) \circ \vartheta \cong M(\xi)$.  By
uniqueness, $\pi(\xi_{0}) \circ \vartheta \cong \pi(\xi_{0})$ and so
$\xi_{0} \in \Xi({^\vee}\mathcal{O}, {^\vee}\mathrm{R}_{\mathbb{C}/\mathbb{R}}
\mathrm{GL}_{N}^{\Gamma})^\vartheta$. The equality $\pi(\xi_{0}) = M(\xi_{0})$ follows from \cite{Vogan78}*{Theorem 6.2 (f)}. 

For part (b), assume that $M(\xi)$ is an actual representation (not an equivalence class) with infinitesimal character $\lambda \in \mathfrak{h}$ satisfying the assumption of (\ref{regintdom}).  As a standard representation, we may write it as
$$M(\xi) = \mathrm{ind}_{P(\mathbb{R})}^{\mathrm{GL}_{N}(\mathbb{C})} \left( \pi_{M} \otimes e^{\nu} \right).$$
Here, $P$ is a  standard real parabolic subgroup with Levi subgroup $M$, $M$ has Langlands decomposition $M(\mathbb{R}) = M^{1}(\mathbb{R}) A(\mathbb{R})$, $\pi_{M}$ is an irreducible tempered representation of $M^{1}(\mathbb{R})$, and $\nu \in \mathfrak{a}^{*}$ has dominant real part. The Levi subgroup $M(\mathbb{R})$ is a product of smaller complex general linear groups (\cite{borel}*{Section 5.2}).  Since $M(\mathbb{R})$ is a direct product of complex general linear groups, \cite{Tadic}*{Proposition 2.1} allows us to write the tempered representation as a parabolically induced representation from a discrete series representation on the unique standard cuspidal Levi subgroup  $H(\mathbb{R})$.  By induction in stages we may write
$$M(\xi) = \mathrm{ind}_{B(\mathbb{R})}^{\mathrm{GL}_{N}(\mathbb{C})} \left( \pi_{0} \otimes e^{\nu}\right)$$
where we now regard $\nu$ as an element in the Lie algebra of the split component of $H(\mathbb{R})$.  The differential of $\pi_{0} \otimes e^{\nu}$ is $\lambda$.  Since $\lambda$ is integrally dominant and $\mathrm{Re}\, \nu$ is dominant, $M(\xi)$ satisfies the properties of \cite{Vogan78}*{Theorem 6.2 (e) (i)-(ii)}. 
From \cite{Vogan78}*{Theorem 6.2 (e)} we know that any irreducible subrepresentation of $M(\xi)$ is generic.  By part (a), we conclude that $M(\xi)$ has a unique irreducible subrepresentation and this subrepresentation is equivalent to $\pi(\xi_{0})$.

For part (c) we consider the standard representation
$M(\xi)$, which has 
a Whittaker functional $\omega$ induced from $\pi_{M}$ above (\cite{Sha80}*{Proposition 3.2}).  The
functional $\omega$  
restricts to a non-zero Whittaker functional on $\pi(\xi_{0})$.
It is simple to verify that  $M(\xi)^{\thicksim}(\vartheta)$ is the intertwining
operator which satisfies $\omega \circ M(\xi)^{\thicksim}(\vartheta) = \omega$.
Restricting this equation to the subrepresentation $\pi(\xi_{0})$
yields in turn that 
\begin{equation}
  \label{subrep}
\pi(\xi_{0})^{\thicksim}(\vartheta) =
M(\xi)^{\thicksim}(\vartheta)\mid_{\pi(\xi_{0})} \mbox{   and    }\pi(\xi_{0})^{\thicksim}
\hookrightarrow M(\xi)^{\thicksim}.
\end{equation}
\end{proof}
Lemma \ref{uniquegeneric}  tells us that the
multiplicity of $\pi(\xi_{0})^{\thicksim}$ in $M(\xi)^{\thicksim}$ is
one.  On the other hand the twisted multiplicity $m_{r}^{\vartheta}(\xi_{0}, \xi)$  of (\ref{twistmult}) tells us
about the ``signed multiplicity" of $\pi(\xi_{0})^{+}$ in
$M(\xi)^{+}$.  We investigate $m_{r}^{\vartheta}(\xi_{0}, \xi)$ further
before comparing the two kinds of multiplicities. 
\begin{prop}
  \label{conjq}
Suppose $\xi \in \Xi({^\vee}\mathcal{O}, {^\vee}\mathrm{R}_{\mathbb{C}/\mathbb{R}}
\mathrm{GL}_{N}^{\Gamma})^\vartheta$ and $\pi(\xi_{0})$ is the
generic subrepresentation of $M(\xi)$ (Lemma \ref{uniquegeneric}).
Then 
 \begin{equation}
  \label{qequation}m_{r}^{\vartheta}(\xi_{0}, \xi) = (-1)^{l^{I}(\xi) -
  l^{I}_{\vartheta}(\xi) + l^{I}(\xi_{0}) -
  l^{I}_{\vartheta}(\xi_{0})}.
\end{equation}
\end{prop}
\begin{proof}
It follows from the definition of the KLV-polynomials (\cite{LV2014}*{Section 0.1}), (\ref{twistgmult}) and \cite{AMR1}*{Proposition B.1} that
$${^\vee}P^{\vartheta}(\, {^\vee}\xi', {^\vee}\xi)(1) = (-1)^{d(\xi)-d(\xi')} \,
c_{g}^{\vartheta}(\xi',\xi) = (-1)^{l^{I}(\xi)-l^{I}(\xi')} \,
c_{g}^{\vartheta}(\xi',\xi).$$
(Note that the definition of $P(\xi)$ in \cite{LV2014} differs from
ours by a shift in degree $d(\xi)$ \emph{cf}. \cite{ABV}*{(7.10)(d)}.)
According to Proposition \ref{p:twist} the above equation may be written as
$${^\vee}P^{\vartheta}(\, {^\vee}\xi', {^\vee}\xi)(1) =    (-1)^{l^{I}(\xi) -
  l^{I}_{\vartheta}(\xi) + l^{I}(\xi_{0}) -
  l^{I}_{\vartheta}(\xi_{0})} \, m_{r}^{\vartheta}(\xi_{0}, \xi)$$  
Therefore the proposition is equivalent to proving
$${^\vee}P^{\vartheta}(\,{^\vee}\xi, {^\vee}\xi_{0}) (1) =1.$$
This equation follows from \cite{aam}*{Proposition 4.17} once we
establish that  ${^\vee}\xi_{0}$ satisfies the stated hypotheses. Indeed, the proof of \cite{aam}*{Proposition 4.17} depends on the group $\mathrm{R}_{\mathbb{C}/\mathbb{R}} \mathrm{GL}_{N}$ only in terms of the types of roots (\ref{glntypes}).  The proof relies crucially on the absence of some nuisance types, which fortunately do not appear in (\ref{glntypes}) either.

The hypotheses of \cite{aam}*{Proposition 4.17} refer to the \emph{block} of $\xi$ (\cite{ICIV}*{Definition 1.14}).  The hypotheses of  \cite{aam}*{Proposition 4.17} are satisfied if we can prove that ${^\vee}\xi_{0}$  is the unique maximal parameter in the
block of $\pi({^\vee}\xi)$ with respect to the (dual) Bruhat order.  This is equivalent
to establishing that $\xi_{0}$ is the unique minimal parameter in the
block of $\pi(\xi)$ (\cite{ICIV}*{Theorem 1.15}). We use  \cite{ABV}*{Proposition 1.11} to convert the Bruhat order
for the representations of $\mathrm{GL}_{N}(\mathbb{C})$ into a
closure relation between ${^\vee}\mathrm{R}_{\mathbb{C}/\mathbb{R}} \mathrm{GL}_{N}$-orbits of
$X({^\vee}\mathcal{O}, {^\vee}\mathrm{R}_{\mathbb{C}/\mathbb{R}}\mathrm{GL}_{N}^{\Gamma})$.   Moreover, this
proposition  implies that the minimality  of
$\xi_{0} = (S_{0}, 1  )$ 
is equivalent to the ${^\vee}\mathrm{R}_{\mathbb{C}/\mathbb{R}}\mathrm{GL}_{N}$-orbit  $S_{0} \subset
X({^\vee}\mathcal{O}, {^\vee}\mathrm{R}_{\mathbb{C}/\mathbb{R}}\mathrm{GL}_{N}^{\Gamma})$ being maximal.  The uniqueness and maximality of this orbit follows from
\cite{ABV}*{\emph{p.} 19}.  
\end{proof}
Equation (\ref{subrep}) gives us information about the multiplicity of the Whittaker extension of a generic representation, and Equation (\ref{qequation}) gives us information about the multiplicity of the Atlas extension of a generic representation.  In the next theorem we combine this information  with Lemma \ref{prinsame} to determine the sign between the two extensions of a generic representation. This may then be leveraged to determine the sign between the two extensions of an arbitrary irreducible representation.  

\begin{thm}
\label{wasign1}
Suppose $\xi \in \Xi({^\vee}\mathcal{O}, {^\vee}\mathrm{R}_{\mathbb{C}/\mathbb{R}}
\mathrm{GL}_{N}^{\Gamma})^\vartheta$.  Then
  $$M(\xi)^{\thicksim}(\vartheta)  = (-1)^{l^{I}(\xi) -
  l^{I}_{\vartheta}(\xi)} \ M(\xi)^{+} (\vartheta)$$
and
$$\pi(\xi)^{\thicksim}(\vartheta)  = (-1)^{l^{I}(\xi) -
  l^{I}_{\vartheta}(\xi)} \ \pi(\xi)^{+} (\vartheta).$$
\end{thm}
\begin{proof}
We first prove that there exists
$\xi_{p} =(x_{p},y_{p})\in
\Xi({^\vee}\mathcal{O},{^\vee}\mathrm{R}_{\mathbb{C}/\mathbb{R}}\mathrm{GL}_{N}^{\Gamma})^\vartheta$ satisfying the hypothesis of Lemma \ref{prinsame} such that $\pi(\xi_{p})$ lies in the block of $\pi(\xi)$.  Let $(x,y)$ be the Atlas parameter corresponding to $\xi$ (Lemma \ref{XXXi}) and $\theta_{x}$ be as in (\ref{thetax}).  A short exercise using (\ref{twistinv}) and (\ref{e:p}) shows that
$$\theta_{x} = (w, w_{0} w^{-1} w_{0}^{-1}) \delta_{0}$$
for some Weyl group element $w$ for $\mathrm{R}_{\mathbb{C}/\mathbb{R}} \mathrm{GL}_{N}$, the long Weyl group element $w_{0}$ for $\mathrm{R}_{\mathbb{C}/\mathbb{R}} \mathrm{GL}_{N}$, and $\delta_{0}$ as in (\ref{del0action}).  According to \cite{ICIV}*{Theorem 8.8} there is a transitive action on the block of $\pi(\xi)$ given by cross actions (there are no Cayley transforms for $\mathrm{R}_{\mathbb{C}/\mathbb{R}}\mathrm{GL}_{N}$).  Cross actions are recorded in the Atlas parameter $(x,y)$ by conjugating both entries with an appropriate Weyl group element (\cite{Adams-Fokko}*{(9.11)(f)}). By taking the cross action of $(x,y)$ with respect to $(w_{0}w^{-1},1) \in W(\mathrm{R}_{\mathbb{C}/\mathbb{R}}\mathrm{GL}_{N},H)$ we arrive at an Atlas parameter $(x_{p},y_{p})$ such that
$$\theta_{x_{p}} = (w_{0}, w_{0})\delta_{0}.$$
Evidently, $\theta_{x_{p}}$ sends all positive roots to negative roots.  Setting $\xi_{p} = (x_{p}, y_{p})$, we see that $\pi(\xi_{p})$ lies in the block of $\pi(\xi)$ and satisfies the hypothesis of Lemma \ref{prinsame}.
Consequently, $M(\xi_{p})^{\thicksim} = M(\xi_{p})^{+}$.  It is straightforward to show
$l^{I}(\xi_{p}) = l^{I}_{\vartheta}(\xi_{p}) = 0$.  By
(\ref{twistmult1}) and (\ref{qequation}) 
\begin{align*}
  M(\xi_{p})^{\thicksim}&= m_{r}^{\vartheta}(\xi_{0},\xi_{p}) \,
\pi(\xi_{0})^{+} \  +  \sum_{\xi' \neq \xi_{0}} m_{r}^{\vartheta}(\xi',\xi_{p})\,
\pi(\xi')^{+}\\
\nonumber &= (-1)^{l^{I}(\xi_{0}) - l^{I}_{\vartheta}(\xi_{0})}
\, \pi(\xi_{0})^{+} \   + \sum_{\xi' \neq \xi_{0}} m_{r}^{\vartheta}(\xi',\xi_{p})\, 
\pi(\xi')^{+}.
\end{align*}
According to (\ref{subrep}), with $\xi = \xi_{p}$, this equation implies
\begin{equation}
  \label{conjgen}
(-1)^{l^{I}(\xi_{0}) - l^{I}_{\vartheta}(\xi_{0})}
  \pi(\xi_{0})^{+}(\vartheta) = \pi(\xi_{0})^{\thicksim}(\vartheta).
\end{equation}
Thus, the theorem holds for $\xi = \xi_{0}$.

It remains to prove that the theorem  holds when $\xi \neq \xi_{0}$.  We
compute, using (\ref{qequation}) and (\ref{conjgen}), that 
\begin{align*}
M(\xi)^{+} &= \sum_{\xi' \neq \xi_{0}} m_{r}^{\vartheta}(\xi',\xi)
\, \pi(\xi')^{+}+ m_{r}^{\vartheta}(\xi_{0},\xi) \, \pi(\xi_{0})^{+}\\
&=\sum_{\xi' \neq \xi_{0}} m_{r}^{\vartheta}(\xi',\xi) \,
\pi(\xi')^{+} + (-1)^{l^{I}(\xi) -
  l^{I}_{\vartheta}(\xi) + l^{I}(\xi_{0}) -
  l^{I}_{\vartheta}(\xi_{0})}   \, \pi(\xi_{0})^{+}\\
& = \sum_{\xi' \neq \xi_{0}} m_{r}^{\vartheta}(\xi',\xi)
\, \pi(\xi')^{+}  + (-1)^{l^{I}(\xi) -
  l^{I}_{\vartheta}(\xi)}   \pi(\xi_{0})^{\thicksim}.
\end{align*}
This equation and Lemma \ref{uniquegeneric}  imply
$$(-1)^{l^{I}(\xi) -
  l^{I}_{\vartheta}(\xi)} \pi(\xi_{0})^{\thicksim}(\vartheta) =
M(\xi)^{+} (\vartheta)\mid_{\pi(\xi_{0})}.$$
Combining this equation with (\ref{subrep}), we see in turn that
$$(-1)^{l^{I}(\xi) -
  l^{I}_{\vartheta}(\xi)}  M(\xi)^{+} (\vartheta)_{|\pi(\xi_{0})} =
\pi(\xi_{0})^{\thicksim}(\vartheta) = M(\xi)^{\thicksim}
(\vartheta)_{|\pi(\xi_{0})}$$
and $(-1)^{l^{I}(\xi) -
  l^{I}_{\vartheta}(\xi)}  M(\xi)^{+} (\vartheta) = M(\xi)^{\thicksim}
(\vartheta)$.  By taking Langlands quotients  we
obtain
$(-1)^{l^{I}(\xi) -
  l^{I}_{\vartheta}(\xi)}  \pi(\xi)^{+} (\vartheta) = \pi(\xi)^{\thicksim}
(\vartheta)$.
\end{proof}
Theorem \ref{wasign1} allows us to replace the Atlas extensions with signed Whittaker extensions in the twisted pairing (\ref{pair2}) and the endoscopic lifting (\ref{twistendlift}).  The results are recorded in the following two corollaries.  The proofs are obtained from the analogous corollaries indicated in \cite{aam}. 
\begin{cor}{\cite{aam}*{Corollary 7.9}}
  \label{twistpairingfinal}
  The pairing \eqref{pair2}, defined by \eqref{pairdef2},
  satisfies
$$\langle M(\xi)^{\thicksim}, \mu(\xi')^{+} \rangle = \delta_{\xi, \xi'}$$
and
  $$\langle \pi(\xi)^{\thicksim}, P(\xi')^{+} \rangle = (-1)^{d(\xi)} \,
  \delta_{\xi, \xi'}
$$
for $\xi, \xi' \in \Xi(\mathcal{O},{^\vee}\mathrm{R}_{\mathbb{C}/\mathbb{R}}\mathrm{GL}_{N}^{\Gamma})^\vartheta$. 
Equivalently (\emph{cf.} Proposition \ref{p:twist}),
\begin{equation}
  \label{twist15.13a}
m_{r}^{\thicksim}(\xi',\xi)  = (-1)^{d(\xi) - d(\xi')}
c_{g}^{\vartheta}(\xi, \xi'),
\nomenclature{$m_{r}^{\thicksim}(\xi',\xi)$}{}
\end{equation}
where $m_{r}^{\thicksim}(\xi', \xi)$ is defined by the decomposition
\begin{equation*}
M(\xi)^{\thicksim} = \sum_{\xi' \in \Xi(\mathcal{O},
{^\vee}\mathrm{R}_{\mathbb{C}/\mathbb{R}}\mathrm{GL}_{N}^{\Gamma})^{\vartheta}}
m^{\thicksim}_{r}(\xi',\xi) \, \pi(\xi')^{\thicksim}. 
\end{equation*}
in  $K\Pi(\mathcal{O}, \mathrm{GL}_{N}(\mathbb{C}), \vartheta)$.
\end{cor}

\begin{cor}{\cite{aam}*{Corollary 7.10}}
  \label{cortrans}
  Suppose  that $S_G \subset
  X(\mathcal{O}_{G},{^\vee}G^{\Gamma})$ is a 
  ${^\vee}G$-orbit and let 
  $\epsilon(S_{G})  \subset X(\mathcal{O}, {^\vee}\mathrm{R}_{\mathbb{C}/\mathbb{R}}\mathrm{GL}_{N}^{\Gamma})$ be the
  ${^\vee}\mathrm{R}_{\mathbb{C}/\mathbb{R}}\mathrm{GL}_{N}$-orbit of the image of $S_{G}$ under
  $\epsilon$ (\ref{epinclusion1}). Then
  \begin{enumerate}[label={(\alph*)}]
  \item \begin{equation*}
    \mathrm{Lift}_{0} (\eta^{\mathrm{loc}}_{S_{G}}(\upsigma)(\delta_{q}))
    = M(\epsilon(S_{G}), 1)^{\thicksim},
    \end{equation*}

  \item  $$\mathrm{Lift}_0 =\mathrm{Trans}_{G(\mathbb{R})}^{\mathrm{GL}_{N}(\mathbb{C}) \rtimes \vartheta}$$
on $K_{\mathbb{C}}\Pi(\mathcal{O}_{G}, G(\mathbb{R}, \delta_{q}))^{\mathrm{st}}$.
    \end{enumerate}
  
\end{cor}

\section{The equality of $\Pi_{\psi_{G}}^{\mathrm{Mok}}$ and $\Pi_{\psi_{G}}^{\mathrm{ABV}}$ for
  regular infinitesimal character}
\label{equalapacketreg}

In this section we prove the equality of the stable virtual characters $\eta^{\mathrm{Mok}}_{\psi_{G}}$
and $\eta^{\mathrm{ABV}}_{\psi_{G}}$ under a regularity condition on the infinitesimal character.
We shall work using the framework of Section \ref{twistendsec}.  In
particular, $\psi_{G}$ and $\psi = \epsilon \circ \psi_{G}$ are A-parameters
with respective infinitesimal characters ${^\vee}\mathcal{O}_{G}$ and
${^\vee}\mathcal{O}$.  The assumption on the infinitesimal characters is that
they are regular with respect to $\mathrm{R}_{\mathbb{C}/\mathbb{R}}\mathrm{GL}_N$.  This assumption
shall be removed in the next section.

The definition of $\eta^{\mathrm{Mok}}_{\psi_{G}}$ was outlined in the introduction. Let us follow \cite{Mok} and provide a few more details. 
All  we need is contained in the following lemma, which is a version of Lemma 8.1 \cite{aam} for unitary groups.
\begin{lem}
  \label{2.2.2}
  Let  $S_{\psi} \subset X({^\vee}\mathcal{O},{}^{\vee}\mathrm{R}_{\mathbb{C}/\mathbb{R}}\mathrm{GL}_N^{\Gamma})$ be the
  ${^\vee}\mathrm{R}_{\mathbb{C}/\mathbb{R}} \mathrm{GL}_{N}$-orbit corresponding to $\phi_{\psi}$ (\cite{ABV}*{Proposition 6.17}).
\begin{enumerate}[label={(\alph*)}]
\item  There
    exist integers $n_{S}$ such that 
    \begin{equation}
\label{piwhittdecomp}
\pi(S_{\psi}, 1)^{\thicksim} = \sum_{(S, 1)  \in
\Xi({^\vee}\mathcal{O},{}^{\vee}\mathrm{R}_{\mathbb{C}/\mathbb{R}}\mathrm{GL}_N^{\Gamma})^{\vartheta} } n_{S} \,
M(S, 1)^{\thicksim}
\end{equation}
    in $K\Pi({^\vee}\mathcal{O}, \mathrm{GL}_{N}(\mathbb{C}), \vartheta)$.
\item  For every $S$ such that $n_{S} \neq 0$ in
  (\ref{piwhittdecomp}) there exists a unique
${^\vee}G$-orbit $S_{G} \subset X({^\vee}\mathcal{O}_{G},
{^\vee}G^{\Gamma})$  which is carried to $S$ under $\epsilon$.

\item Writing
  $$S = \epsilon(S_{G})$$ for the orbits in part (b), we have
\begin{equation}\label{piwhittdecomp1} 
  \begin{aligned}
\pi(S_{\psi}, 1)^{\thicksim} &=
 \mathrm{Trans}_{G(\mathbb{R})}^{\mathrm{GL}_N(\mathbb{C}) \rtimes \vartheta}\left(
 \sum_{S_{G}}  n_{\epsilon(S_{G})} \, \eta_{S_{G}}^{\mathrm{loc}} (\delta_{q})
 \right)\\
  &=
 \mathrm{Lift}_{0} \left(
 \sum_{S_{G}}  n_{\epsilon(S_{G})} \, \eta_{S_{G}}^{\mathrm{loc}} (\delta_{q})
 \right) . 
\end{aligned}
\end{equation}

\end{enumerate}
\end{lem}
\begin{proof} The proof follows almost exactly as for \cite{aam}*{Lemma 8.1}.  The only difference is that the existence of the orbit $S_{G}$ in part (b) is established in the proof of \cite{Mok}*{Proposition 8.2.1}.
\end{proof}

Let us recall that Mok defines $\eta^{\mathrm{Mok}}_{\psi_{G}}$ through the identity 
$$
\pi(S_{\psi}, 1)^{\thicksim}=\mathrm{Trans}_{G(\mathbb{R})}^{\mathrm{GL}_N(\mathbb{C}) \rtimes \vartheta}\left(\eta^{\mathrm{Mok}}_{\psi_{G}}\right).
$$
Since $\mathrm{Trans}_{G(\mathbb{R})}^{\mathrm{GL}_N(\mathbb{C}) \rtimes \vartheta}=
\mathrm{Lift}_{0}$ (\ref{cortrans}) and
$\mathrm{Lift}_{0}$ is injective (Proposition \ref{injlift2}), it follows by 
Equation (\ref{piwhittdecomp1}) that
\begin{equation}
\label{Arthurstabchar}
\eta^{\mathrm{Mok}}_{\psi_{G}} = \sum_{S_{G}} n_{\epsilon(S_{G})} \,
\eta^{\mathrm{loc}}_{S_{G}}(\delta_{q})  \in K_{\mathbb{C}} 
\Pi({^\vee}\mathcal{O}_{G}, G(\mathbb{R}, \delta_{q}))^{\mathrm{st}}
\nomenclature{$\eta^{\mathrm{Mok}}_{\psi_{G}}$}{stable virtual character defining Arthur's A-packet}
\end{equation}
(\emph{cf.} \cite{Mok}*{Proposition 8.2.1}).  By definition, the A-packet
$\Pi_{\psi_{G}}^{\mathrm{Mok}}
\nomenclature{$\Pi_{\psi_{G}}^{\mathrm{Mok}}$}{Mok's A-packet}
$ consists of 
those irreducible representations in $\Pi({^\vee}\mathcal{O}_{G}, G(\mathbb{R},
\delta_{q}))$ which occur with non-zero multiplicity when
(\ref{Arthurstabchar}) is expressed as a linear combination of irreducible representations.

\begin{thm}
  \label{finalthm}
  Let $\psi_G$ be an A-parameter for $G$ with regular infinitesimal character. Then
  $$\eta^{\mathrm{Mok}}_{\psi_{G}} = \eta_{\psi_{G}}^{\mathrm{mic}}(\delta_{q}) =
  \eta^{\mathrm{ABV}}_{\psi_{G}}   \quad\mbox{ and }\quad 
  \Pi_{\psi_{G}}^{\mathrm{Mok}} = \Pi_{\psi_{G}}^{\mathrm{ABV}}.$$
  \end{thm}
  \begin{proof}
The proof is the same as that of \cite{aam}*{Theorem 8.2 (a)}.
We repeat it here in the case of unitary groups for the convenience of the reader.    
Let $\xi=(S_{\psi},1)$
    as in Corollary \ref{etaplus2}.


    $$
    \begin{aligned}
\mathrm{Lift}_{0} \, (\eta^{\mathrm{mic}}_{\psi_{G}}(\delta_{q})) &=  \mathrm{Lift}_{0} \,(
\eta^{\mathrm{mic}}_{\psi_{G}}(\vartheta) (\delta_{q}))\quad(\text{by } \eqref{nosigma})\\ 
&=(-1)^{l^{I}(\xi)-l^{I}_{\vartheta}(\xi)} \pi(\xi)^{+}
\quad(\text{Corollary }\ref{etaplus2})
\\
&=\pi(\xi)^\thicksim\quad(\text{Theorem }\ref{wasign1})\\
&=\pi(S_{\psi},1)^\thicksim\\
&=  \mathrm{Trans}_{G(\mathbb{R})}^{\mathrm{GL}_{N}(\mathbb{C}) \rtimes
  \vartheta}\left( \sum_{S_{G}}  n_{\epsilon(S_{G})} \,
  \eta_{S_{G}}^{\mathrm{loc}} (\delta_{q}) \right)
\quad(\mathrm{Lemma }\, \ref{2.2.2}, (\ref{Arthurstabchar}))\\
& =  \mathrm{Lift}_{0}  \left( \eta^{\mathrm{Mok}}_{\psi_{G}} \right)\quad(\mathrm{Corollary }\ \ref{cortrans}(b)).\\ 
\end{aligned}
$$
The equality of the stable virtual characters follows from the injectivity of
$\mathrm{Lift}_{0}$  (Proposition \ref{injlift2}).
The equality of packets follows immediately.
\end{proof}

\section{The equality of $\Pi_{\psi_{G}}^{\mathrm{Mok}}$ and $\Pi_{\psi_{G}}^{\mathrm{ABV}}$ for
  singular infinitesimal character} 
\label{equalapacketsing}

To conclude our comparison of stable virtual characters,  we  retain
the setup of the previous 
section, but without the hypothesis of regularity on
the infinitesimal character.  In other words, the orbits ${^\vee}\mathcal{O}_{G}$ and  ${^\vee}\mathcal{O}$
 are now allowed to be orbits of singular
infinitesimal characters and the reader should think of them as such.
In order to prove $\eta^{\mathrm{Mok}}_{\psi_{G}} = \eta^{\mathrm{ABV}}_{\psi_{G}}$ for singular $\lambda \in {^\vee}\mathcal{O}_{G}$, we must extend the
pairing of Theorem \ref{twistpairing} and extend the twisted
endoscopic lifting (\ref{twistendlift}) to 
include representations with singular infinitesimal character.  
This was done in Section 9 of \cite{aam} through the use of the
\emph{Jantzen-Zuckerman translation principle}, to which we refer 
from now on simply as \emph{translation}. 
The  same arguments can be used to extend the results 
of the previous sections to singular representations of $\mathrm{GL}_N(\mathbb{C})$.
 For the convenience of the reader we are including it here again,  adding 
the needed modifications to the context of the current paper.

In essence, translation allows one to transfer results for regular infinitesimal character to results for singular infinitesimal character.  Applying this principle to the results of the previous section will allow us to
compare $\Pi_{\psi_{G}}$ with $\Pi_{\psi_{G}}^{\mathrm{ABV}}$ with no restriction on
the infinitesimal character.

The translation principle 
begins with the existence of a
regular orbit ${^\vee}\mathcal{O}' \subset {^\vee}\mathfrak{gl}_{N}\times {^\vee}\mathfrak{gl}_N$ and a \emph{translation datum} 
$\mathcal{T}$ from ${^\vee}\mathcal{O}$ to ${^\vee}\mathcal{O}'$ (\cite{ABV}*{Definition 8.6, Lemma 8.7}). 
\nomenclature{$\mathcal{T}$}{translation datum}
As ${^\vee}\mathcal{O}$ is the ${^\vee}\mathrm{R}_{\mathbb{C}/\mathbb{R}}\mathrm{GL}_{N}$-orbit
of $\lambda \in {^\vee}\mathfrak{h}^{\vartheta}$ we may take
${^\vee}\mathcal{O}'$ to be the  ${^\vee}\mathrm{R}_{\mathbb{C}/\mathbb{R}}\mathrm{GL}_{N}$-orbit of  
\begin{equation}
\label{lambdaprime}
\lambda' = \lambda + \lambda_{1} \in {^\vee}\mathfrak{h}
\end{equation}
 where $\lambda_{1} \in X_{*}(H)$ is regular and  dominant with respect to the
 positive system of $R^{+}(\mathrm{R}_{\mathbb{C}/\mathbb{R}}\mathrm{GL}_N,H)$.   We may and shall
 take $\lambda_{1}$ to be the sum of the positive roots.  In this way,
each of $\lambda$, $\lambda_{1}$ and $\lambda'$ are fixed by
$\vartheta$.  The translation 
 datum $\mathcal{T}$ 
 induces a ${^\vee}\mathrm{R}_{\mathbb{C}/\mathbb{R}}\mathrm{GL}_{N}$-equivariant morphism 
\begin{equation}
\label{ftee}
f_{\mathcal{T}} :  X({^\vee}\mathcal{O}', {}^{\vee}\mathrm{R}_{\mathbb{C}/\mathbb{R}}\mathrm{GL}_N^{\Gamma}) \rightarrow
X({^\vee}\mathcal{O},{}^{\vee}\mathrm{R}_{\mathbb{C}/\mathbb{R}}\mathrm{GL}_N^{\Gamma}) \nomenclature{$f_{\mathcal{T}}$}{}
\end{equation}
of geometric parameters  (\cite{ABV}*{Proposition 8.8}).  The morphism
has connected fibres of fixed dimension, a fact we shall use
when comparing orbit dimensions.  The
${^\vee}\mathrm{R}_{\mathbb{C}/\mathbb{R}}\mathrm{GL}_{N}$-equivariance of (\ref{ftee}) is tantamount to
a coset map commuting with left-multiplication by ${^\vee}\mathrm{R}_{\mathbb{C}/\mathbb{R}}\mathrm{GL}_{N}$
(\cite{ABV}*{(6.10)(b)}).
Since both $\lambda$ and $\lambda'$ are fixed by $\vartheta$, it is
just as easy to see that the action of $\vartheta$ commutes with the
same coset map.  We leave this exercise to the reader, taking for
granted the resulting ${^\vee}\mathrm{R}_{\mathbb{C}/\mathbb{R}}\mathrm{GL}_{N}$-equivariance of (\ref{ftee}).

According to \cite{ABV}*{Proposition 7.15}, the morphism $f_{\mathcal{T}}$ induces an inclusion 
\begin{equation*}
f^{*}_{\mathcal{T}}: \Xi({^\vee}\mathcal{O},{}^{\vee}\mathrm{R}_{\mathbb{C}/\mathbb{R}}\mathrm{GL}_N^{\Gamma})
\hookrightarrow \Xi({^\vee}\mathcal{O}',{}^{\vee}\mathrm{R}_{\mathbb{C}/\mathbb{R}}\mathrm{GL}_N^{\Gamma})\nomenclature{$f^{*}_{\mathcal{T}}$}{}
\end{equation*}
of complete geometric parameters.
The $\vartheta$-equivariance of
(\ref{ftee}) implies that this inclusion restricts to an inclusion (denoted by the same symbol)
$$f^{*}_{\mathcal{T}}: \Xi({^\vee}\mathcal{O},{}^{\vee}\mathrm{R}_{\mathbb{C}/\mathbb{R}}\mathrm{GL}_N^{\Gamma})^\vartheta
\hookrightarrow \Xi({^\vee}\mathcal{O}',{}^{\vee}\mathrm{R}_{\mathbb{C}/\mathbb{R}}\mathrm{GL}_N^{\Gamma})^\vartheta. $$
The (Jantzen-Zuckerman) translation functor (\cite{AvLTV}*{(17.8j)})
$$T_{\lambda'}^{\lambda} = T_{\lambda +  \lambda_{1}}^{\lambda}
\nomenclature{$T_{\lambda'}^{\lambda}$}{Jantzen-Zuckerman translation}
$$
is an exact functor on
a category of Harish-Chandra modules, which we shall often regard
as a homomorphism
\begin{equation}
  \label{transfunct}
  T_{\lambda + \lambda_{1}}^{\lambda}: K \Pi({^\vee}\mathcal{O}',
  \mathrm{GL}_{N}(\mathbb{C}) \rtimes \langle \vartheta \rangle )
  \rightarrow  K \Pi({^\vee}\mathcal{O},
  \mathrm{GL}_{N}(\mathbb{C}) \rtimes \langle \vartheta \rangle )
\end{equation}
of Grothendieck groups.
It is  surjective  (\cite{AvLTV}*{Corollary 17.9.8}).
This translation functor is an extended version of the usual
translation functor (\cite{AvLTV}*{(16.8f)}), which we
also  denote by
\begin{equation}
  \label{ordtrans}
  T_{\lambda + \lambda_{1}}^{\lambda}: K \Pi({^\vee}\mathcal{O}',
  \mathrm{GL}_{N}(\mathbb{C}) )
  \rightarrow  K \Pi({^\vee}\mathcal{O},
  \mathrm{GL}_{N}(\mathbb{C}) ).
\end{equation}
Let us take a moment to make (\ref{transfunct}) more precise.  The sum of the
positive roots $\lambda_{1}$ is the infinitesimal character of a
finite-dimensional representation of $\mathrm{GL}_{N}(\mathbb{C})$.
Therefore, $\lambda_{1}$ is the differential of a $\vartheta$-fixed
quasicharacter $\Lambda_{1}$ of the
real diagonal torus $H(\mathbb{R})$, which matches the weight of this
finite-dimensional representation. The quasicharacter
$\Lambda_{1}$ may be extended to a quasicharacter $\Lambda_{1}^{+}$ of the
semi-direct product 
$H(\mathbb{R}) \rtimes \langle \vartheta \rangle$  by
setting
\begin{equation*}
  \Lambda_{1}^{+}(\vartheta) = 1.
\end{equation*}

We define translation in the extended setting of (\ref{transfunct})
using this representation of the extended group.  Since the extension
is evident here we continue to write
$T_{\lambda+\lambda_1}^{\lambda}$ instead of $T_{\lambda+\Lambda_1^+}^{\lambda}$.

In the ordinary setting of (\ref{ordtrans}) we have
\begin{align*}
\pi(\xi) &= T_{\lambda + \lambda_{1}}^{\lambda} \left(
  \pi(f^{*}_{\mathcal{T}}(\xi)) \right),\\
  \nonumber  M(\xi) &= T_{\lambda + \lambda_{1}}^{\lambda} \left(
  M(f^{*}_{\mathcal{T}}(\xi)) \right), \quad \xi \in \Xi({^\vee}\mathcal{O},
  {}^{\vee}\mathrm{R}_{\mathbb{C}/\mathbb{R}}\mathrm{GL}_N^{\Gamma})
\end{align*}
(\cite{AvLTV}*{Corollary 16.9.4, 16.9.7 and 16.9.8}, or \cite{ABV}*{ Theorem 16.4 and
   Proposition 16.6}).
We \emph{define} the Atlas extensions of 
$\pi(\xi)$ and $M(\xi)$, with $\xi \in
\Xi({^\vee}\mathcal{O},{}^{\vee}\mathrm{R}_{\mathbb{C}/\mathbb{R}}\mathrm{GL}_N^{\Gamma})^\vartheta$, by
\begin{align*}
\pi(\xi)^{+} &= T_{\lambda + \lambda_{1}}^{\lambda}
  (\pi(f^{*}_{\mathcal{T}}(\xi))^{+})\\
   M(\xi)^{+} &= T_{\lambda + \lambda_{1}}^{\lambda}
  (M(f^{*}_{\mathcal{T}}(\xi))^{+}) .
\end{align*}

With the definition of Atlas extensions in place, the discussion of
Section \ref{grothchar} is valid, and we see that
$T_{\lambda + \lambda_{1}}^{\lambda}$ factors to a homomorphism of $K
\Pi({^\vee}\mathcal{O}, \mathrm{GL}_{N}(\mathbb{C}), \vartheta)$ (see
(\ref{twistgroth1})).  We use the same notation
$T_{\lambda + \lambda_{1}}^{\lambda}$ to denote the functor of Harish-Chandra modules,
and either of the earlier homomorphisms.  The reader will be reminded
of the context when it is important.

The definition of a Whittaker extension does not depend on the
regularity of the infinitesimal character.   The following proposition
shows that  translation sends Whittaker extensions to
Whittaker extensions.
\begin{prop}
  \label{whittowhit}
Suppose $\xi \in \Xi( {^\vee}\mathcal{O}, {}^{\vee}\mathrm{R}_{\mathbb{C}/\mathbb{R}}\mathrm{GL}_N^{\Gamma})^\vartheta$.  Then (as Harish-Chandra
modules)
$$T_{\lambda + \lambda_{1}}^{\lambda}\left(
M(f^{*}_{\mathcal{T}}(\xi))^{\thicksim} \right) =
M(\xi)^{\thicksim},$$
and
$$T_{\lambda + \lambda_{1}}^{\lambda} \left( \pi(f^{*}_{\mathcal{T}}(\xi))^{\thicksim}\right) =
\pi(\xi)^{\thicksim}.$$ 
\end{prop}
\begin{proof} The proof runs along the same lines as that of 
\cite{aam}*{Proposition 9.1}. The proof is even simpler in the current context, 
since every $\vartheta$-invariant irreducible representation of $\mathrm{GL}_N(\mathbb{C})$
appears as a subquotient of a principal series representation 
as given in Lemma \ref{prinsame}.
\end{proof}

Our translation datum $\mathcal{T}$ for $\mathrm{R}_{\mathbb{C}/\mathbb{R}}\mathrm{GL}_N$ is defined by
(\ref{lambdaprime}), in which both $\lambda$ and $\lambda'$ are fixed
by  $\vartheta$.  For this
reason  (\ref{lambdaprime}) also determines a translation datum
$\mathcal{T}_{G}$ from   
${^\vee}G$-orbits  ${^\vee}\mathcal{O}_{G}$ to   ${^\vee}\mathcal{O}'_{G}$ for
the twisted endoscopic group $G$ (\cite{ABV}*{Definition 8.6 (e)}).
Just as
for $\mathrm{R}_{\mathbb{C}/\mathbb{R}}\mathrm{GL}_N$, we have maps
\begin{align*}
 f_{\mathcal{T}_{G}} &:  X({^\vee}\mathcal{O}', {^\vee}G^{\Gamma}) \rightarrow
X({^\vee}\mathcal{O}, {^\vee}G^{\Gamma}) \\
\nonumber f^{*}_{\mathcal{T}_{G}}&: \Xi({^\vee}\mathcal{O},{^\vee}G^{\Gamma})
\hookrightarrow \Xi({^\vee}\mathcal{O}',{^\vee}G^{\Gamma})
\end{align*}
and the  translation functor $T_{\lambda +
  \lambda_{1}}^{\lambda}$ which satisfies
$$\pi(\xi) = T_{\lambda + \lambda_{1}}^{\lambda}\left(
\pi(f^{*}_{\mathcal{T}_{G}}(\xi)) \right), \quad \xi \in
\Xi({^\vee}\mathcal{O}_{G}, {^\vee}G^{\Gamma})$$
(\cite{ABV}*{Proposition 16.6}, \cite{AvLTV}*{Section 16}).

The translation data $\mathcal{T}$ and $\mathcal{T}_{G}$ allow us to
transport properties of our pairings at regular infinitesimal
character  to the same properties
for pairings at singular  infinitesimal character. More precisely, as explained
at the end of \cite{ABV}*{page 178}, the translation datum $\mathcal{T}_G$ applied to Proposition \ref{ordpairingequiv}
allows us to extend Theorem \ref{ordpairing} to any infinitesimal character.
In a similar fashion, the translation datum $\mathcal{T}$
applied this time to Equation (\ref{twist15.13a}),
allows us to transport Corollary \ref{twistpairingfinal}
from regular to singular infinitesimal character.  These two procedures have the following result.
\begin{prop}[\cite{aam}*{Proposition 9.2}]
  \label{finalpairing}
  Define the pairing
\begin{equation}\label{eq:finalpairing}
\langle \cdot, \cdot \rangle:  K \Pi({^\vee}\mathcal{O},
\mathrm{GL}_{N}(\mathbb{C}), \vartheta) 
\times K X({^\vee}\mathcal{O}, {}^{\vee}\mathrm{R}_{\mathbb{C}/\mathbb{R}}\mathrm{GL}_N^{\Gamma}, \vartheta)
\rightarrow \mathbb{Z}
\end{equation}
by
$$\langle M(\xi)^{\thicksim}, \mu(\xi')^{+} \rangle = \delta_{\xi, \xi'}.$$
Then
  $$\langle \pi(\xi)^{\thicksim}, P(\xi')^{+} \rangle = (-1)^{d(\xi)} \,
  \delta_{\xi, \xi'}$$
where $\xi, \xi' \in
\Xi({^\vee}\mathcal{O}, {}^{\vee}\mathrm{R}_{\mathbb{C}/\mathbb{R}}\mathrm{GL}_N^{\Gamma})^\vartheta$.
\end{prop}
Proposition \ref{finalpairing} is the final version of the twisted
pairing, and we use it to extend the definition of endoscopic lifting
$\mathrm{Lift}_{0}$ 
to include singular infinitesimal characters ((\ref{twistloweps}),
(\ref{twistendlift})).  In fact, all of the remaining results used in Section
\ref{equalapacketreg} easily carry over to the more general setting.
In particular, using the pairing 
(\ref{eq:finalpairing})
in the proof of 
Proposition    \ref{twistimlift}, we see that 
for any  ${^\vee}G$-orbit $S_G \subset
  X({^\vee}\mathcal{O}_{G},{^\vee}G^{\Gamma})$ we still have
$$\mathrm{Lift}_{0} \left(\eta^{\mathrm{loc}}_{S_{G}}(\vartheta)(\delta_{q})\right)
    = M(\epsilon(S_{G}), 1)^{\thicksim}.$$
Finally, since the injectivity of $\mathrm{Lift}_{0}$ still holds for singular infinitesimal character, the same argument used in the proof of Theorem \ref{finalthm}
allows us to conclude
  \begin{thm}
    \label{finalthm1}
Let $\psi_G$ be an A-parameter for $G$. Then
  $$\eta^{\mathrm{Mok}}_{\psi_{G}} = \eta_{\psi_{G}}^{\mathrm{mic}}(\delta_{q}) =
  \eta^{\mathrm{ABV}}_{\psi_{G}}   \quad\mbox{and}\quad 
  \Pi_{\psi_{G}}^{\mathrm{Mok}} = \Pi_{\psi_{G}}^{\mathrm{ABV}}.$$
  \end{thm}

\section{Some consequences for standard endoscopy}
  \label{consec}
  
  In this section we explore standard endoscopic lifting from an endoscopic group $G'$ of our quasisplit unitary group $G$. Both  \cite{ABV} and \cite{Mok}  provide formulae for this lifting, and it is not obvious why the two types of formulae should be equal.  Our goal is to show how Theorem \ref{finalthm1} implies the equality of the formulae, arguing  as in \cite{aam}*{Section 10}.

  A brief review of the definition of $G'$ may be found in \cite{ABV}*{Section 26}. The endoscopic group $G'$ is defined to
be a quasisplit form of a complex reductive group whose dual ${^\vee}G'$ is the identity component of the 
centralizer in ${^\vee}G$ of a semisimple element $s \in {^\vee}G$
(\emph{cf.} Section \ref{endosec}).
We further assume that  the element $s$ centralizes the
image of a fixed A-parameter $\psi_{G}$ as in (\ref{aparameter}).
There is a natural embedding
\begin{equation}
  \label{epsprime}
\epsilon': {^\vee}(G')^{\Gamma} \hookrightarrow
         {^\vee}G^{\Gamma}
\end{equation}
(\emph{cf.} \cite{Mok}*{(2.1.13)}), and an A-parameter $\psi_{G'}$ 
for $G'$ such that 
$$\psi_{G} = \epsilon' \circ \psi_{G'}$$
(\emph{cf.} \cite{Arthur}*{\emph{p.} 36}).
We write $\Pi_{\psi_{G}}$ for $\Pi_{\psi_{G}}^{\mathrm{Mok}} = \Pi_{\psi_{G}}^{\mathrm{ABV}}$.  For any $\pi \in \Pi_{\psi_{G}}$, we write $\tau^{\mathrm{mic}}_{\psi_{G}}(\pi) = \tau^{\mathrm{mic}}_{S_{\psi_{G}}}(\pi)$ for the representation of $A_{\psi_{G}}={^\vee}G_{\psi_{G}}/
({^\vee}G_{\psi_{G}})^{0}$ introduced in (\ref{miclocalsys}).

Let us first look at endoscopic lifting from the perspective of \cite{ABV}.  Let $\mathrm{Lift}_{G'(\mathbb{R})}^{G(\mathbb{R})}$ denote the endoscopic lifting map from stable virtual characters of the quasisplit group $G'(\mathbb{R})$ to the virtual characters of $G(\mathbb{R}) = G(\mathbb{R},\delta_{q})$ as given in \cite{ABV}*{Definition 26.18}.  According to \cite{ABV}*{Theorem 22.7 and Theorem 26.25} the stable virtual character $\eta^{\mathrm{ABV}}_{\psi_{G'}}$ of $G'(\mathbb{R})$ satisfies 
  \begin{equation}
    \label{abvordlift}
    \mathrm{Lift}_{G'(\mathbb{R})}^{G(\mathbb{R})} (\eta^{\mathrm{ABV}}_{\psi_{G'}}) = \sum_{\pi \in \Pi_{\psi_{G}}} (-1)^{d(\pi) - d(S_{\psi_{G}})}\,  \mathrm{Tr}\left( \tau^{\mathrm{mic}}_{\psi_{G}} (\pi) (\bar{s}) \right) \, \pi.
  \end{equation}
  Here, $d(S_{\psi_{G}})$ is the dimension of the ${^\vee}G$-orbit $S_{\psi_{G}}$, and  $\bar{s}$ is the coset of $s$ in $A_{\psi_{G}}$.
    
  The analogue of  formula (\ref{abvordlift}) from Mok's perspective is \cite{Mok}*{(8.2.4)}.  It describes the endoscopic lifting map $\mathrm{Trans}_{G'(\mathbb{R})}^{G(\mathbb{R})}$ defined by Shelstad (\cite{shelstad_endoscopy}) on a stable virtual character $\eta^{\mathrm{Mok}}_{\psi_{G'}}$ (\cite{Mok}*{Theorem 3.2.1}).  With this notation Mok's formula is
  \begin{equation}
    \label{mordlift}
\mathrm{Trans}_{G'(\mathbb{R})}^{G(\mathbb{R})} ( \eta^{\mathrm{Mok}}_{\psi_{G'}} ) = \sum_{\sigma \in \Sigma_{\psi_{G}} } \langle s_{\psi_{G}} \bar{s} , \sigma \rangle \, \sigma.
  \end{equation}
  Here, $\Sigma_{\psi_{G}}$ is a
finite set of non-negative integral linear combinations  
$$\sigma = \sum_{\pi \in  \Pi_{\mathrm{unit}}(G(\mathbb{R}))} m(\sigma, \pi) \, \pi$$
of irreducible unitary characters of $G(\mathbb{R})$.  Furthermore, there is an injective map from
$\Sigma_{\psi_{G}}$ into the set of those
quasicharacters of $A_{\psi_{G}}$ which are trivial on the centre of
${^\vee}G$.  The injection is denoted by 
$$\sigma \mapsto \langle \cdot, \sigma \rangle.$$
The element $s_{\psi_{G}}$ is the image of
\small
$$\psi_{G} \left(1, \begin{bmatrix} -1 & 0\\ 0 &
  -1 \end{bmatrix} \right)$$\normalsize 
in $A_{\psi_{G}}$.

If one takes $s = 1$ in  (\ref{mordlift}) then one recovers $G' = G$ and the equation reduces to 
\begin{align}
  \label{preveq}
  \eta^{\mathrm{Mok}}_{\psi_{G}}  &= \sum_{\sigma \in \Sigma_{\psi_{G}} } \langle s_{\psi_{G}}  , \sigma  \rangle \, \sigma \\
\nonumber  &=  \sum_{\sigma \in \Sigma_{\psi_{G}} } \langle s_{\psi_{G}}  , \sigma  \rangle \, \sum_{\pi \in  \Pi_{\mathrm{unit}}(G(\mathbb{R}))} m(\sigma, \pi) \, \pi\\
\nonumber  &= \sum_{\pi \in \Pi_{\psi_{G}}} \left( \sum_{\sigma \in \Sigma_{\psi_{G}} }  m(\sigma, \pi)\,   \langle s_{\psi_{G}}  , \sigma \rangle \right) \, \pi.
\end{align}
For each $\pi \in \Pi_{\psi_{G}}$ define
$$\tau_{\psi_{G}}(\pi) =  \sum_{\sigma \in \Sigma_{\psi_{G}} }  m(\sigma, \pi)\,   \langle \cdot  , \sigma \rangle,$$
which is apparently a finite sum of quasicharacters of $A_{\psi_{G}}$.
Moeglin and Renard have proven that $\tau_{\psi_{G}}(\pi)$ is actually
irreducible (\cite{MR4}*{Theorem 1.3}).  Therefore (\ref{preveq}) becomes
$$\eta^{\mathrm{Mok}}_{\psi_{G}} =  \sum_{\pi \in \Pi_{\psi_{G}}}  \tau_{\psi_{G}} (\pi) (s_{\psi_{G}})  \, \pi,$$ 
and more generally (\ref{mordlift}) becomes
\begin{equation}
  \label{mordlift1}
  \mathrm{Trans}_{G'(\mathbb{R})}^{G(\mathbb{R})}( \eta^{\mathrm{Mok}}_{\psi_{G'}}) = \sum_{\pi \in \Pi_{\psi_{G}}}   \tau_{\psi_{G}}(\pi) (s_{\psi_{G}} \bar{s})  \, \pi. 
\end{equation}

Our goal is to show that the two equations, (\ref{mordlift1}) and  (\ref{abvordlift}), are identical.  In order to compare the two equations, we must choose the element $s \in {^\vee}G$ defining the endoscopic group $G'$ carefully.   Recall that $s$ lies in the centralizer of the image of $\psi_{G}$.  We may suppose $\bar{s} \in A_{\psi_{G}}$ is not trivial. An explicit description of this centralizer is given in \cite{Mok}*{(2.4.13)} (following \cite{ggp}*{Section 4}).  In consideration of this description, it is not difficult to choose for each, $\bar{s} \in A_{\psi_{G}}$, a diagonal representative $\dot{s}$ in the centralizer with eigenvalues $\pm 1$.
The endoscopic group $G'(\dot{s})$ determined by $\dot{s}$ is a direct product $G'_{1}(\dot{s}) \times G'_{2}(\dot{s})$ in which each of the two factors is a quasisplit unitary group whose rank is less than $G$ (\emph{cf}.  \cite{rog}*{Proposition 4.6.1}).  The A-parameter $\psi_{G'(\dot{s})}$ decomposes accordingly as a product $\psi_{G'_{1}(\dot{s})} \times \psi_{G'_{2}(\dot{s})}$ of A-parameters. Similarly, Mok's stable virtual character $\eta_{\psi_{G'(\dot{s})}}^{\mathrm{Mok}}$ is defined as the tensor product $\eta_{\psi_{G'_{1}(\dot{s})}}^{\mathrm{Mok}} \otimes \eta_{\psi_{G'_{2}(\dot{s})}}^{\mathrm{Mok}}$.   For these choices of endoscopic data, (\ref{mordlift1}) reads as
\begin{equation}
  \label{mordliftsdot}
\mathrm{Trans}_{G'(\mathbb{R})}^{G(\mathbb{R})}(\eta_{\psi_{G'_{1}(\dot{s})}}^{\mathrm{Mok}} \otimes \eta_{\psi_{G'_{2}(\dot{s})}}^{\mathrm{Mok}}) = \sum_{\pi \in \Pi_{\psi_{G}}} \tau_{\psi_{G}}(\pi)(s_{\psi_{G}} \bar{s}) \pi.
\end{equation}
We now turn to rewriting the left-hand side of (\ref{mordliftsdot}) so as to match it with the left-hand side of (\ref{abvordlift}).
First, it is noted on  \cite{ABV}*{\emph{p.} 289} that
\begin{equation}
  \label{translift}
\mathrm{Trans}_{G'(\mathbb{R})}^{G(\mathbb{R})} =
\mathrm{Lift}_{G'(\mathbb{R})}^{G(\mathbb{R})}.
\end{equation}
Second, using the arguments in the proof of Corollary \ref{cor:LeviABVpacket}, we see that 
$$\eta_{\psi_{G'(\dot{s})}}^{\mathrm{ABV}} = \eta_{\psi_{G'_{1}(\dot{s})}}^{\mathrm{ABV}} \otimes \eta_{\psi_{G'_{2}(\dot{s})}}^{\mathrm{ABV}}.$$
Third, since $G'_{1}(\dot{s})$ and $G_{2}'(\dot{s})$ are both quasisplit unitary groups, Theorem \ref{finalthm1}  tells us that 
\begin{equation}
  \label{factorend}
\eta_{\psi_{G'_{j}(\dot{s})}}^{\mathrm{Mok}} =
\eta_{\psi_{G'_{j}(\dot{s})}}^{\mathrm{ABV}}, \quad j = 1,2.
\end{equation}
Taking these three observations together we conclude
\begin{align*}
\mathrm{Trans}_{G'(\mathbb{R})}^{G(\mathbb{R})}\left(\eta_{\psi_{G'(\dot{s})}}^{\mathrm{Mok}} \right)&= \mathrm{Lift}_{G'(\mathbb{R})}^{G(\mathbb{R})}\left(\eta_{\psi_{G_{1}'(\dot{s})}}^{\mathrm{Mok}} 
\otimes \eta_{\psi_{G_{2}'(\dot{s})}}^{\mathrm{Mok}} \right)\\ 
& = \mathrm{Lift}_{G'(\mathbb{R})}^{G(\mathbb{R})} \left(\eta_{\psi_{G'_{1}(\dot{s})}}^{\mathrm{ABV}} \otimes 
\eta_{\psi_{G'_{2}(\dot{s})}}^{\mathrm{ABV}}\right)\\
& = \mathrm{Lift}_{G'(\mathbb{R})}^{G(\mathbb{R})}\left(\eta_{\psi_{G'(\dot{s})}}^{\mathrm{ABV}}\right).
\end{align*}
It is now immediate from (\ref{mordliftsdot}) and (\ref{abvordlift}) that
$$ \sum_{\pi \in \Pi_{\psi_{G}}} \tau_{\psi_{G}}(\pi)(s_{\psi_{G}} \bar{s}) \pi
= \sum_{\pi \in \Pi_{\psi_{G}}}
(-1)^{d(\pi) - d(S_{\psi_{G}})}
\ \mathrm{Tr} \left(\tau^{\mathrm{ABV}}_{\psi_{G}}(\pi)(\bar{s}) \right) \, \pi$$
for any $\bar{s} \in A_{\psi}$.
By the linear independence of characters on $G(\mathbb{R})$
$$\tau_{\psi_{G}}(\pi)(s_{\psi_{G}} \bar{s})
= (-1)^{d(\pi) - d(S_{\psi_{G}})}
\ \mathrm{Tr} \left(\tau^{\mathrm{ABV}}_{\psi_{G}}(\pi)(\bar{s}) \right)$$ 
for any $\bar{s} \in A_{\psi}$.  This may be regarded as an equality between virtual characters on $A_{\psi}$. By appealing to the linear independence of these characters we conclude that
$$\tau_{\psi_{G}}(\pi)(s_{\psi_{G}}) = (-1)^{d(\pi) - d(S_{\psi_{G}})}$$
and
$$\tau_{\psi_{G}}(\pi) = \tau^{\mathrm{ABV}}_{\psi_{G}}(\pi).$$
In particular, $\tau_{\psi_{G}}^{\mathrm{ABV}}(\pi)$ is irreducible.
This completes our comparison of (\ref{abvordlift}) and
(\ref{mordlift1}).

\section{Pure inner forms}
\label{puresec}

Fix a pure strong involution $\delta$ of $G^{\Gamma}$ as given in
(\ref{pureinv}).  As $\delta$ runs over the equivalence classes of pure strong
involutions, the real forms $G(\mathbb{R}, \delta)$ run over
all isomorphism classes of indefinite unitary groups
$\mathrm{U}(p,q)$,  $p+q = N$ (\cite{adams11}*{Section 9}).  Recall that
$\delta_{q}$ is a pure strong involution.  We shall
extend the results of Section \ref{consec} by replacing
$G(\mathbb{R}) = G(\mathbb{R},\delta_{q})$ with any $G(\mathbb{R},
\delta)$.  The virtual characters $\eta^{\mathrm{Mok}}_{\psi_{G}}$
 for quasisplit unitary groups
are replaced by virtual characters $\eta^{\mathrm{MR}}_{\psi_{G}}$,
defined and studied by Moeglin and Renard (\cite{MR3}, \cite{MR4}).
With these virtual characters in place the arguments are more or less
the same as in Section \ref{consec}.

There is one additional wrinkle in the non-quasisplit setting, which
is to prove that the endoscopic maps of Shelstad and \cite{ABV} are
equal.  We were 
able to manage this in Section \ref{consec} by using  canonical
transfer factors for quasisplit groups and citing an exercise on
\cite{ABV}*{\emph{p.} 289} (\emph{cf.} (\ref{translift})).  For
non-quasisplit groups we refer to methods of Kaletha and \cite{ARM}.
Let us settle this matter straightaway.  We continue with the initial
setup of Section \ref{consec}, namely with an endoscopic group $G'$
obtained from a semisimple 
element $s \in {^\vee}G$, but temporarily without the assumption of $s$ being in
the centralizer of $\psi_{G}$.

The first significant departure from the previous
section arises in the nature of the spectral transfer map
$\mathrm{Trans}_{G'(\mathbb{R})}^{G(\mathbb{R}, \delta)}$  from stable virtual characters of the
quasisplit group $G'(\mathbb{R})$ to the virtual characters of
$G(\mathbb{R},\delta)$.  If $\delta = \delta_{q}$ so that
$G(\mathbb{R})=G(\mathbb{R}, \delta)$ is quasisplit, a  Whittaker
datum fixes a set of constants, the \emph{transfer factors}, which
specify the map $\mathrm{Trans}_{G'(\mathbb{R})}^{G(\mathbb{R})}$
(\cite{ShelstadIII}*{Corollary 11.7}).  When
$G(\mathbb{R}, \delta)$ is not quasisplit, Shelstad does not provide
canonical transfer factors, and as a result
$\mathrm{Trans}_{G'(\mathbb{R})}^{G(\mathbb{R}, \delta)}$ is defined
only up to a non-zero constant.  This ambiguity is rectified by
Kaletha in \cite{kal}, where canonical transfer factors 
for $G(\mathbb{R}, \delta)$ are again defined relative to a fixed
Whittaker datum.  This specifies the transfer map
$\mathrm{Trans}_{G'(\mathbb{R})}^{G(\mathbb{R},\delta)}$ on tempered
L-packets and agrees with
$\mathrm{Trans}_{G'(\mathbb{R})}^{G(\mathbb{R})}$ when $G(\mathbb{R},
\delta)$ is quasisplit (\cite{kal}*{(5.11), Proposition 5.10}).  The
extension  of
$\mathrm{Trans}_{G'(\mathbb{R})}^{G(\mathbb{R},\delta)}$ to the space
of stable virtual characters follows from the characterization of this
space (\cite{MW}*{Corollary IV.2.8}), and an
analytic continuation argument from the tempered setting (\emph{cf.}
Lemma 4.8 \cite{Adams-Johnson}).
More concretely, the map
$\mathrm{Trans}_{G'(\mathbb{R})}^{G(\mathbb{R},\delta)}$ is
characterized by
\begin{equation}
  \label{transid}
\mathrm{Trans}_{G'(\mathbb{R})}^{G(\mathbb{R},\delta)} \left(
\sum_{\tau_{S_{1}}} M(S_{1}, \tau_{S_{1}})  \right) =
e(\delta) \sum_{\tau_{S}} \tau_{S}(s)\, M(S,\tau_{S})
\end{equation}
(\emph{cf.} \cite{kal}*{(5.9)}).
On the left, $\mathrm{Trans}_{G'(\mathbb{R})}^{G(\mathbb{R},\delta)}$
is evaluated on a stable virtual character (\ref{etaloc}) for
$(G')^{\Gamma}$.  On the right, $e(\delta) = \pm 1$ is the
Kottwitz invariant of $G(\mathbb{R}, \delta)$ (\cite{ABV}*{Definition 15.8}), and the
${^\vee}G$-orbit $S \subset X({^\vee}\mathcal{O}_{G}, {^\vee}G^{\Gamma})$ is the one determined by
the image of the ${^\vee}G'$-orbit $S_{1}$ under (\ref{epsprime}).
The sum on the right runs over certain characters $\tau_{S}$ of
${^\vee}G_{p}/({^\vee}G_{p})^{0}$, for fixed $p \in S$ (see the discussion following (\ref{abv6.17})).  They are the
characters  which correspond to the pure strong involution
$\delta$ under \cite{kal}*{(5.13), Corollary 5.4, Theorem 5.2}.  It follows from
\cite{ARM}*{Theorem 1.1} that this correspondence between $\tau_{S}$
and $\delta$ coincides with the one given by
(\ref{localLanglandspure}) in \cite{ABV}.  (This is just another way of 
saying that we are justified in denoting the standard representations
on the right-hand side of (\ref{transid}) by $M(S,\tau_{S})$.)

Let $\mathrm{Lift}_{G'(\mathbb{R})}^{G(\mathbb{R},\delta)}$ denote the
endoscopic lifting map written as $\mathrm{Lift}_{0}(\delta)$ in
\cite{ABV}*{Definition 26.18}.  By \cite{ABV}*{Proposition 26.7},
$$\mathrm{Lift}_{G'(\mathbb{R})}^{G(\mathbb{R},\delta)} \left(
\sum_{\tau_{S_{1}}} M(S_{1}, \tau_{S_{1}})  \right) =
e(\delta) \sum_{\tau_{S}} \tau_{S}(s)\, M(S,\tau_{S}),$$
where the terms are identical to those of (\ref{transid}).  In
consequence,
\begin{equation}
  \label{translift1}
\mathrm{Trans}_{G'(\mathbb{R})}^{G(\mathbb{R},\delta)} =
\mathrm{Lift}_{G'(\mathbb{R})}^{G(\mathbb{R},\delta)}.
\end{equation}

We now adopt all of the assumptions of Section \ref{consec}, so
that  $s \in {^\vee}G$ is taken to centralize the image of a fixed
A-parameter $\psi_{G}$ and  $\bar{s}$ is the coset of $s$ in $A_{\psi_{G}}$.
In \cite{MR4} Moeglin and Renard prove
Arthur's local conjectures for all pure inner forms of unitary groups
(\emph{cf.} \cite{Kaletha-Minguez}*{Section 1.6.1}).
Let us give a brief description of their solution to the conjectures.  

For each $\bar{s}\in A_{\psi_G}$ we choose a representative $s \in
{^\vee}G$ and its endoscopic group $G'$ as in the paragraph following Equation (\ref{mordlift1}).  
Then there is an A-parameter  $\psi_{G'}$ for $G'$ such that
$$\psi_{G} = \epsilon' \circ \psi_{G'}.$$
For each $\bar{s}\in A_{\psi_G}$, Moeglin and Renard define the
virtual character
$\eta_{\psi_G}^{\mathrm{MR}}( \bar{s})$  of $G(\R,\delta)$
by
$$
e(\delta)\, \eta_{\psi_G}^{\mathrm{MR}}(\bar{s})=
\mathrm{Trans}_{G'(\R)}^{G(\R,\delta)}\left(\eta_{\psi_{G'}}^{\mathrm{Mok}}\right)
$$
(\cite{MR3}*{Sections 1 and 2.1}).
At first glance, the virtual character
$\eta_{\psi_G}^{\mathrm{MR}}( \bar{s})$ is a finite
linear combination of irreducible representations  with complex
coefficients. In fact, the coefficients are integers
and the irreducible representations are unitary. More precisely, in
\cite{MR4}*{Theorem 5.3}
the authors define a finite set 
\begin{align*}
\Pi_{\psi_G}^{\mathrm{MR}}(G(\R,\delta))
\end{align*}
of unitary representations, together with a mapping  
\begin{align*}
\pi \longmapsto \tau_{\psi_{G}}(\pi),
\end{align*} 
from $\Pi_{\psi_G}^{\mathrm{MR}}(G(\R,\delta))$ to the group of irreducible representations of $A_{\psi_G}$, such that
\begin{align*}
\eta_{\psi_G}^{\mathrm{MR}}(\bar{s})=
\sum_{\pi \in \Pi_{\psi_{G}}^{\mathrm{MR}}(G(\R,\delta))} \tau_{\psi_{G}}(\pi)(s_{\psi_{G}} \bar{s})\, \pi,\quad
\bar{s}\in A_{\psi_G}.
\end{align*}
In particular, for $s=1$, $G'(\mathbb{R}) = G(\mathbb{R}, \delta_{q})$
and 
$$e(\delta) \, \eta_{\psi_G}^{\mathrm{MR}}(1)=
\mathrm{Trans}_{G(\R,\delta_{q})}^{G(\R,\delta)}\left(\eta_{\psi_{G}}^{\mathrm{Mok}}
\right).$$
This virtual character is stable since $\mathrm{Trans}_{G(\R,\delta_{q})}^{G(\R,\delta)}$
carries stable virtual characters to stable virtual characters
(\emph{cf.} (\ref{transid}) with $s=1$).
The set 
$\Pi_{\psi_G}^{\mathrm{MR}}(G(\R,\delta))$
is the Arthur packet attached to the stable virtual character
$\eta_{\psi_G}^{\mathrm{MR}} = \eta_{\psi_G}^{\mathrm{MR}}(1)$.

We would like to compare these constructions of Moeglin and Renard
with their analogues in \cite{ABV}*{Theorem 22.7 and Definition 26.8}.
For this we return to the discussion preceding (\ref{abvordlift}) and
specialize to the real form $G(\mathbb{R},\delta)$.  The analogous
Arthur packet $\Pi^{\mathrm{ABV}}_{\psi_{G}}(G(\mathbb{R},\delta))$
was defined in (\ref{abvdefdelta}).  
We use the representations $\tau^{\mathrm{mic}}_{\psi_{G}}(\pi)$ given in
(\ref{miclocalsys}) to define
  \begin{equation}
    \eta^{\mathrm{ABV}}_{\psi_{G}}(\delta)(\bar{s}) = \sum_{\pi \in
      \Pi^{\mathrm{ABV}}_{\psi_{G}}(G(\mathbb{R},\delta))} e(\delta)
    \, (-1)^{d(\pi) - d(S_{\psi_{G}})}\,  \mathrm{Tr}\left( \tau^{\mathrm{mic}}_{\psi_{G}} (\pi) (\bar{s}) \right) \, \pi
  \end{equation}
(\cite{ABV}*{\emph{p.} 281}).  In particular, 
  $\eta^{\mathrm{ABV}}_{\psi_{G}}(\delta)(1)$ is the stable virtual character
$\eta_{\psi_{G}}^{\mathrm{ABV}}(\delta)$ defined in (\ref{etapsiabvdelta}).
  \begin{thm}
    \label{purethm}
Let $\psi_G$ be an A-parameter for $G$. Then for any pure inner form
$G(\mathbb{R},\delta)$ we have
$$
e(\delta)\, \eta_{\psi_G}^{\mathrm{MR}}(\bar{s})=
\eta_{\psi_G}^{\mathrm{ABV}}(\delta) (\bar{s}),\quad \bar{s}\in A_{\psi_G}.
$$
In particular,
$$
\Pi_{\psi_G}^{\mathrm{MR}}(G(\R,\delta))=
\Pi_{\psi_G}^{\mathrm{ABV}}(G(\R,\delta)), \quad \tau_{\psi_{G}}(\pi) =  \tau^{\mathrm{ABV}}_{\psi_G}(\pi)
$$
and
$$
 \tau_{\psi_{G}}(\pi)(s_{\psi_{G}}) = (-1)^{d(\pi) - d(S_{\psi_{G}})}.
$$
\end{thm} 
\begin{proof} For any $\bar{s}\in A_{\psi_G}$, we have
\begin{align*}
e(\delta)\, \eta_{\psi_G}^{\mathrm{MR}}(\bar{s})&=
\mathrm{Trans}_{G'}^{G(\R,\delta)}(\eta_{\psi_{G'}}^{\mathrm{Mok}}) \\
&=\mathrm{Lift}_{G'(\R)}^{G(\R,\delta)}(\eta_{\psi_{G'}}^{\mathrm{Mok}}) \quad\text{(Equation (\ref{translift1}))}\\
&=\mathrm{Lift}_{G'(\R)}^{G(\R,\delta)}(\eta_{\psi_{G'}}^{\mathrm{ABV}})
\quad\text{(Theorem \ref{finalthm1} and Equation (\ref{factorend}))}\\
&=  \eta_{\psi_G}^{\mathrm{ABV}}(\delta)(\bar{s}) \quad
\text{(\cite{ABV}*{Theorem 26.25})}.
\end{align*}
The equality of packets follows immediately after taking $s =1$. 
The remaining equalities follow the argument used at the end of Section \ref{consec}.
\end{proof}

\printnomenclature



\end{document}